\documentclass[reqno, 12pt]{amsart}

\usepackage{amsmath}
\usepackage{amsfonts}
\usepackage{amssymb}
\usepackage{amsthm}
\usepackage{mathrsfs}
\usepackage{bbm}
\usepackage{graphicx, color}
\usepackage{tikz}
\usetikzlibrary{calc,decorations.markings}
\usetikzlibrary{arrows.meta}
\usetikzlibrary{positioning}
\usetikzlibrary{cd, intersections, calc, decorations.pathmorphing, arrows, decorations.pathreplacing}

\usepackage{adjustbox}

\usepackage{bmpsize}

\usepackage{mathtools}

\usepackage{enumerate}

\definecolor{refkey}{rgb}{0.9451,0.2706,0.4941}
\definecolor{labelkey}{rgb}{0.9451,0.2706,0.4941}

\definecolor{mygreen}{rgb}{0,0.7,0.3}
\definecolor{myblue}{rgb}{0,0.50,1.20}
\definecolor{myorange}{rgb}{1,0.5,0.1}

\usepackage{subfiles}
\usepackage[colorlinks, linkcolor=blue, citecolor=red, urlcolor=cyan,pagebackref]{hyperref}

\allowdisplaybreaks

\numberwithin{equation}{section}

\usepackage{cleveref}
\crefname{thm}{Theorem}{Theorems}
\crefname{cor}{Corollary}{Corollaries}
\crefname{lem}{Lemma}{Lemmas}
\crefname{sublem}{Sublemma}{Sublemmas}
\crefname{prop}{Proposition}{Propositions}
\crefname{dfn}{Definition}{Definitions}
\crefname{defi}{Definition}{Definitions}
\crefname{ex}{Example}{Examples}
\crefname{claim}{Claim}{Claims}
\crefname{conj}{Conjecture}{Conjectures}
\crefname{conv}{Notation}{Notations}
\crefname{rem}{Remark}{Remarks}
\crefname{rmk}{Remark}{Remarks}
\crefname{prob}{Problem}{Problems}
\crefname{figure}{Figure}{Figures}
\crefname{table}{Table}{Tables}
\crefname{section}{Section}{Sections}
\crefname{subsection}{Section}{Sections}
\crefname{appendix}{Appendix}{Appendices}
\crefname{introthm}{Theorem}{Theorems}
\crefname{introcor}{Corollary}{Corollaries}
\crefname{introconj}{Conjecture}{Conjectures}

\newtheorem{thm}{Theorem}[section]
\newtheorem{prop}[thm]{Proposition}
\newtheorem{cor}[thm]{Corollary}
\newtheorem{lem}[thm]{Lemma}

\newtheorem{introthm}{Theorem}

\theoremstyle{definition}
\newtheorem{dfn}[thm]{Definition}

\newtheorem{ex}[thm]{Example}

\theoremstyle{remark}
\newtheorem{rmk}[thm]{Remark}
\newtheorem{rem}[thm]{Remark}

\newcommand{\bZ}{\mathbb{Z}}

\newcommand{\bC}{\mathbb{C}}

\newcommand{\bT}{\mathbb{T}}
\newcommand{\bM}{\mathbb{M}}

\newcommand{\bP}{\mathbb{M}_\circ}

\newcommand{\bA}{\mathbb{A}}
\newcommand{\bB}{\mathbb{B}}

\newcommand{\bI}{\mathbb{I}}

\newcommand{\A}{\mathcal{A}}

\newcommand{\cF}{\mathcal{F}}

\newcommand{\cL}{\mathcal{L}}

\newcommand{\cO}{\mathcal{O}}

\def\P{{\mathcal{P}}}

\newcommand{\cR}{\mathcal{R}}

\newcommand{\X}{\mathcal{X}}

\newcommand{\cZ}{\mathcal{Z}}

\newcommand{\sfa}{\mathsf{a}}
\newcommand{\sfx}{\mathsf{x}}

\newcommand{\sfT}{{\mathsf{T}}}
\newcommand{\bi}{\boldsymbol{i}}

\newcommand{\tri}{\triangle}
\newcommand{\sgn}{\mathrm{sgn}}

\newcommand{\Tri}{\mathrm{Tri}}

\newcommand{\Teich}{Teichm\"uller}

\newcommand{\ve}{\varepsilon}

\newcommand{\Sk}[1]{\mathscr{S}_{#1}^q(\bM)}
\newcommand{\sSk}[1]{\overline{\mathscr{S}}^q_{\!#1}(\bB)}

\newcommand{\sfe}{\mathsf{e}}
\newcommand{\sff}{\mathsf{f}}
\newcommand{\wtri}{{\widehat{\tri}}}

\tikzset{
    squigarrow/.style={-{Classical TikZ Rightarrow[length=4pt]}, decorate, decoration={snake, amplitude=1.8pt, pre length=2pt, post length=3pt}}
}

\newcommand{\BA}{\mathbb A}
\newcommand{\BB}{\mathbb B}

\newcommand{\BD}{\mathbb D}

\newcommand{\BZ}{\mathbb Z}

\newcommand{\CA}{\mathcal A}

\newcommand{\CI}{\mathcal I}

\newcommand{\CO}{\mathcal O}

\newcommand{\CR}{\mathcal R}

\newcommand{\CZ}{\mathcal Z}

\newcommand{\fo}{\mathfrak o}

\newcommand{\al}{\alpha}
\newcommand{\sS}{\mathscr{S}}
\newcommand{\BM}{\mathbb{M}}
\newcommand{\BMp}{\mathbb{M}_\partial}
\newcommand{\BMc}{\BM_\circ}

\newcommand{\ol}{\overline}

\newcommand{\sfB}{\mathsf{B}}

\newcommand{\trt}{{\rm Tr}_T}

\newcommand{\CZbl}{(\CZ_\tri^q)_{\mathrm{bl}}}

\newcommand{\Ibad}{\CI^{{\rm bad}}}

\newcommand{\pSigma}{\partial \Sigma}

\newcommand{\sSsq}{\sS_\Sigma^q}

\newcommand{\congr}{{\rm cong}}

\newcommand{\olSs}{\overline{\sS}_{\!\Sigma}^q}

\DeclareMathOperator{\interior}{\mathrm{int}}
\DeclareMathOperator{\Int}{\mathrm{Int}}

\DeclareMathOperator{\Frac}{\mathrm{Frac}}

\makeatletter
\newcommand{\oset}[3][0ex]{%
  \mathrel{\mathop{#3}\limits^{
    \vbox to#1{\kern-2\ex@
    \hbox{$\scriptstyle#2$}\vss}}}}
\makeatother

\makeatletter
\newcommand{\osetnear}[3][0ex]{%
  \mathrel{\mathop{#3}\limits^{
    \vbox to#1{\kern-.3\ex@
    \hbox{$\scriptstyle#2$}\vss}}}}
\makeatother

\newcommand\qarrow[2]{\draw[->,shorten >=2pt,shorten <=2pt] (#1) -- (#2) [thick];} 

\def\centerarc(#1)(#2:#3:#4)
    {($(#1)+({#4*cos(#2)},{#4*sin(#2)})$) arc(#2:#3:#4) }

\tikzset{
  mid arrow/.style={postaction={decorate,decoration={
        markings,
        mark=at position .5 with {\arrow[#1]{stealth}}
      }}},
}

\tikzset{
    partial ellipse/.style args={#1:#2:#3}{
        insert path={+ (#1:#3) arc (#1:#2:#3)}
    }
}

\newcommand{\bline}[3]{
    \path (#1)++(0,-#3) coordinate(m1);
    \path (#2)++(0,-#3) coordinate(m2);
    \filldraw[gray!30] (m1) -- (#1) -- (#2) -- (m2) --cycle;
    \draw[thick] (#1) -- (#2);
}
\newcommand{\tline}[3]{
    \path (#1)++(0,#3) coordinate(m1);
    \path (#2)++(0,#3) coordinate(m2);
    \filldraw[gray!30] (m1) -- (#1) -- (#2) -- (m2) --cycle;
    \draw[thick] (#1) -- (#2);
}

\tikzset{->-/.style 2 args={
	postaction={decorate},
	decoration={markings, mark=at position #1 with {\arrow[thick, #2]{>}}} 
    },
    ->-/.default={0.5}{}
}
\tikzset{-<-/.style 2 args={
	postaction={decorate},
	decoration={markings, mark=at position #1 with {\arrow[thick, #2]{<}}} 
    },
    -<-/.default={0.5}{}
}

\tikzset{
	overarc/.style={
		white, double=red, double distance=1.2pt, line width=2.4pt
	}
}

\pgfdeclarelayer{bg}    
\pgfsetlayers{bg,main}  

\title[Quantum duality maps, skein algebras and their ensemble compatibility]{Quantum duality maps, skein algebras and their ensemble compatibility}

\author[Tsukasa Ishibashi]{Tsukasa Ishibashi}
\address{Tsukasa Ishibashi, Mathematical Institute, Tohoku University, 
6-3 Aoba, Aramaki, Aoba-ku, Sendai, Miyagi 980-8578, Japan.}
\email{tsukasa.ishibashi.a6@tohoku.ac.jp}

\author[Hiroaki Karuo]{Hiroaki Karuo}
\address{Hiroaki Karuo, Department of Mathematics, Gakushuin University, Mejiro, Toshima-ku, Tokyo, Japan.}
\email{hiroaki.karuo@gakushuin.ac.jp}

\date{\today}

\begin{document}
\maketitle

\begin{abstract}
We generalize the quantum duality map $\mathbb{I}_{\mathcal{A}}$ of Allegretti--Kim \cite{AK} for punctured closed surfaces to general marked surfaces. 
When the marked surface has no interior marked points, 
we investigate its compatibility with the quantum duality map $\mathbb{I}_{\mathcal{X}}$ on the dual side based on the quantum bracelets basis \cite{Th,MQ}. Our construction factors through reduced stated skein algebras, based on the quantum trace maps \cite{Le_triangular} together with an appropriate way of \emph{skein lifting} of integral $\mathcal{A}$-laminations. We also give skein theoretic proofs for some expected properties of Laurent expressions, and positivity of structure constants for marked disks and a marked annulus. 
\end{abstract}

\setcounter{tocdepth}{1}
\tableofcontents

\section{Introduction}

\subsection{Background: quantum \Teich\ theory and the skein algebra}
Quantization of the \Teich\ theory is a program driven by physical observations given in the Verlinde's paper \cite{Ver90}, and has also attracted many people in mathematics. While the original description is based on the Fenchel--Nielsen coordinates, Chekhov--Fock \cite{CF99} developed a quantum \Teich\ theory based on the shear coordinates associated with ideal triangulations. Here to admit ideal triangulations, the surface we consider should have a specified set of marked points (namely, \emph{marked surfaces}). 
One expects that the Poisson algebra of trace functions along loops can be deformed into a non-commutative algebra (``algebra of line operators"), which leads to a quantum ordering problem of their classical coordinate expressions, as is usual in quantum theories (see \emph{e.g.}, \cite[Section 3]{CP07} for a discussion). 

The quantum ordering problem was beautifully solved by Bonahon--Wong \cite{BW11}, based on the skein algebra. Let us briefly recall its definition. 
For a unital ring $\CR$ with a quantum parameter $q \in \cR$ and an oriented surface $\Sigma$, 
Przytycki \cite{Prz91} and Turaev \cite{Tu} independently introduced the (ordinary) \emph{skein algebra} $\mathscr{S}^q_{\Sigma}$ of $\Sigma$. 
It is the $\CR$-algebra generated by all isotopy classes of framed links in $\Sigma\times (-1,1)$ as a module, subject to the following relations and equipped with multiplication by stacking:  
\begin{align*}({\rm A}) \begin{array}{c}\includegraphics[scale=0.23]{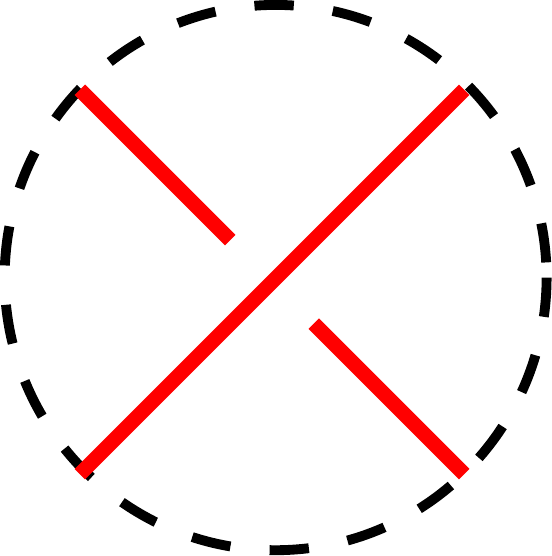}\end{array}=q\begin{array}{c}\includegraphics[scale=0.23]{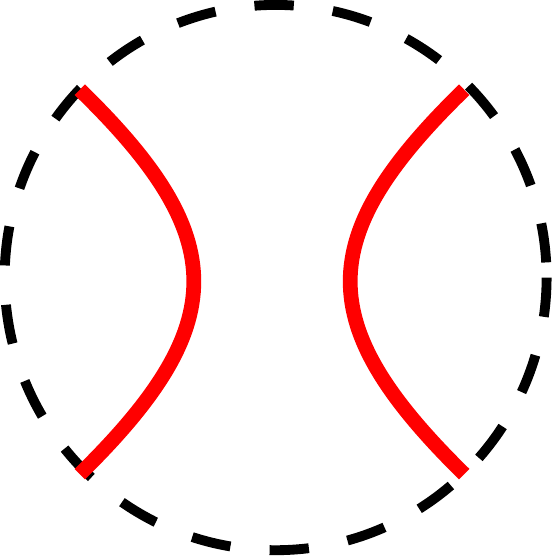}\end{array}+q^{-1}\begin{array}{c}\includegraphics[scale=0.23]{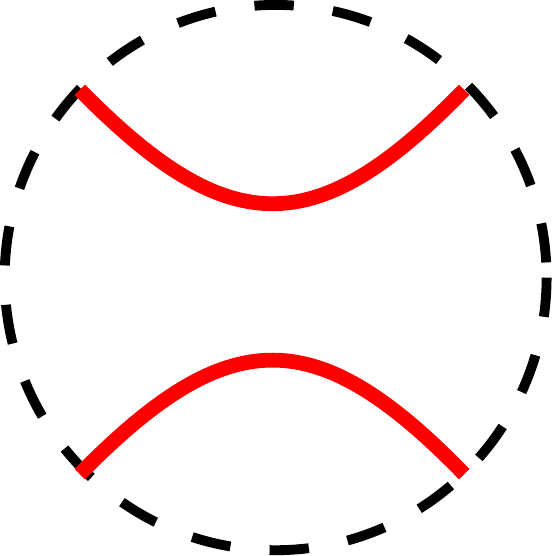}\end{array},\ \ \ ({\rm B}) \begin{array}{c}\includegraphics[scale=0.23]{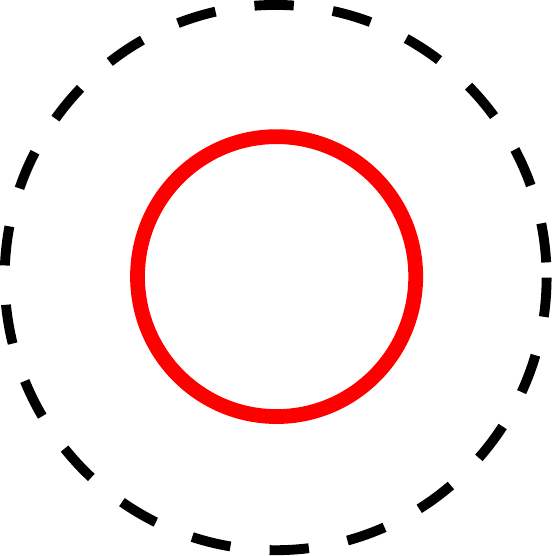}\end{array}=(-q^2-q^{-2})\begin{array}{c}\includegraphics[scale=0.23]{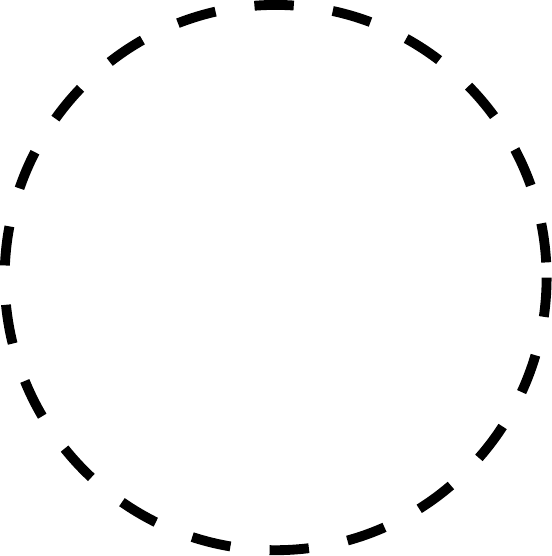}\end{array}.\end{align*}
The skein algebra is known to give a deformation quantization of the ${\rm SL}_2(\bC)$-character variety of $\Sigma$ \cite{Bu}. 
Its algebraic structure has been studied in various ways.

In the case where $\Sigma$ has empty boundary (but with interior marked points), Bonahon--Wong constructed \emph{quantum trace maps}, which give algebra embeddings of $\mathscr{S}^q_{\Sigma}$ into quantum tori quantizing the shear coordinates. In this construction the quantum ordering problem is solved, where the ordering is controlled by the elevation in the $(-1,1)$-direction.

Nowadays, Chekhov--Fock's quantum \Teich\ theory is far generalized to the framework of \emph{quantum cluster varieties} by Fock--Goncharov \cite{FG09} and its representation theory \cite{FG08}. In our case, the \emph{quantum cluster Poisson algebra} (``quantized function algebra" of the cluster Poisson variety) associated with a marked surface has the form
\begin{align*}
    \cO_v(\X_\Sigma) = \bigcap_{\tri} \X_\tri^v,
\end{align*}
where $\tri$ runs over all the \emph{tagged} triangulations of $\Sigma$ \cite{FST}, and $\X_\tri^v$ is the quantum torus deforming the shear coordinates (a.k.a. cluster Poisson variables) with quantum parameter $v=q^{-2}$ (see \cref{eq:emb_to_CF,rem:q-parameters}). The non-commutative algebra $\cO_v(\X_\Sigma)$ should be interpreted as the algebra of line operators. To clarify the situation, let us also consider a similarly defined algebra $\cO_v(\X^\circ_\Sigma) \supset \cO_v(\X_\Sigma)$, where $\tri$ runs only over the ideal triangulations without self-folded triangles. We have $\cO_v(\X^\circ_\Sigma)= \cO_v(\X_\Sigma)$ in the absence of interior marked points. 

Since the classical expression of trace functions in the vector representation $\bC^2$ involves square-roots of shear coordinates\footnote{More precisely, one lifts a $PGL_2(\bC)$-representation to an $SL_2(\bC)$-representation, as discussed in \cite[Section 1.3]{BW11} by using the hyperbolic structure (or the spin structure determined by it). The result involves the square-roots of cluster Poisson variables (see \cite[Lemma 3]{BW11}), which take positive values on the enhanced \Teich\ space (=positive real part of the cluster Poisson variety).}, one is led to introduce \emph{square-root Chekhov--Fock algebra} $\cZ_\tri^q$, which is larger than $\X_\tri^v$. 
In this language, the Bonahon--Wong's quantum trace maps are homomorphisms
\begin{align*}
    \mathrm{Tr}_\tri: \mathscr{S}^q_{\Sigma} \to \cZ_\tri^q
\end{align*}
defined for ideal triangulations $\tri$, which are compatible with the square-root quantum version of coordinate transformations among shear coordinates \cite{Hi}. 
Their restrictions to a subalgebra $\mathscr{S}^q_{\Sigma,\congr} \subset \mathscr{S}^q_{\Sigma}$ (cf. \cref{def:congruent}) take values in $\X_\tri^v$, and combine to give a homomorphism
\begin{align}\label{eq:BW_trace}
    \mathrm{Tr}_\Sigma: \mathscr{S}^q_{\Sigma,\congr} \to \cO_v(\X^\circ_\Sigma)
\end{align}
for any marked surfaces without boundary.

Some of the expected properties of (quantum) trace functions are generalized to the axioms of \emph{(quantum) duality maps} for cluster varieties \cite[Section 4]{FG09}. 
Allegretti--Kim \cite{AK} upgraded the map \eqref{eq:BW_trace} into a (candidate for) quantum duality map
\begin{align}\label{eq:AK_duality}
    \bI_\A: \A_\Sigma(\bZ^\sfT) \to \mathscr{S}^q_{\Sigma,\congr} \to \cO_v(\X^\circ_\Sigma),
\end{align}
where $\A_\Sigma(\bZ^\sfT)$ denotes the set of integral $\A$-laminations on $\Sigma$ \cite{FG07}. In general, a quantum duality map is a map (of sets) from the tropical points of some cluster variety, whose image gives a basis of the function algebra of the dual cluster variety satisfying certain strong axioms, such as the positivity of structure constants.

On the other extreme, let us consider the case where $\Sigma$ has marked points only on the boundary. In this case, we also have another quantum algebra $\cO_q(\A_\Sigma)$, called the \emph{quantum upper cluster algebra} \cite{BZ} (``quantized function algebra" of the cluster $K_2$-variety). Muller \cite{Muller} introduced a version of skein algebra $\Sk{\Sigma}$ (\cref{dfn_Muller}) for marked surfaces $\Sigma$ of this case, and provided an embedding
\begin{align}\label{eq:cutting}
    \mathrm{Cut}_\Sigma: \Sk{\Sigma}[\partial^{-1}]\to \cO_q(\A_\Sigma),
\end{align}
where $[\partial^{-1}]$ stands for the boundary-localization. If $\Sigma$ has at least two marked points, $\mathrm{Cut}_\Sigma$ is an isomorphism \cite[Theorem 9.8]{Muller}. 
In view of this, there is a construction of quantum duality map
\begin{align}\label{eq:bracelets_duality}
    \bI_\X: \X_\Sigma(\bZ^\sfT) \to\Sk{\Sigma}[\partial^{-1}]\to \cO_q(\A_\Sigma),
\end{align}
whose image constitutes the \emph{bracelets basis} \cite{Th}. 
Since $\cO_q(\A_\Sigma)$ is related to $\cO_v(\X_\Sigma)$ via the \emph{ensemble map} $p_\Sigma^\ast: \cO_v(\X_\Sigma) \to \cO_q(\A_\Sigma)$ (\cref{prop:ensemble_mutation}), this relation of skein and cluster algebras also gives us a hint to construct a quantum duality map for $\cO_v(\X_\Sigma)$. Here we use a correction on frozen variables in defining $p_\Sigma^\ast$ (see \eqref{eq:ensemble_linear}) following Goncharov--Shen \cite{GS19}, which is crucial to our results below. 

In the particular case of marked disks, Allegretti \cite{All} gives a construction of quantum duality map 
\begin{align}\label{eq:Allegretti_duality}
    \bI'_\A: \A_\Sigma(\bZ^\sfT) \to \cO_v(\X_\Sigma)
\end{align}
based on the quantum $F$-polynomials, together essentially with this ensemble map (with a related correction on frozen variables).

\subsection{A construction of quantum duality maps via stated skein algebras}
Our aim in this paper is to give a skein theoretic construction of quantum duality maps for general marked surfaces, generalizing the aforementioned works. 

Inspired by the work of Bonahon--Wong, L\^e \cite{Le_triangular} introduced the \emph{stated skein algebras} $\mathscr{S}^q_{\Sigma}(\bB)$ of marked surfaces $\Sigma$ for a better understanding of the quantum trace maps. It consists of framed tangles having endpoints on the boundary, together with \emph{states} $+$ or $-$ on their endpoints. When $\partial\Sigma=\emptyset$, it is just the ordinary skein algebra. 
The stated skein algebra of the bigon is isomorphic to the quantized coordinate ring $\CO_{q^2}({\rm SL}(2))$ as Hopf algebras \cite[Section 3]{CL}. 
The stated skein algebra is related to the Muller's skein algebra $\Sk{\Sigma}$: there is an embedding of $\Sk{\Sigma}$ into $\mathscr{S}_\Sigma^q(\bB)$ as the part consisting only of $-$ states, which induces an isomorphism \cite{LY22}
\begin{align}\label{eq:LY_isom}
    \Sk{\Sigma}[\partial^{-1}]\xrightarrow{\sim} \sSk{\Sigma}.
\end{align}
Here $\sSk{\Sigma}$ is the quotient of $\mathscr{S}^q_{\Sigma}(\bB)$ by an ideal, called the \emph{reduced stated skein algebra} (\cref{def:reduced_stated}).

Based on these backgrounds, it is natural to consider a generalization of Allegretti--Kim's duality map \eqref{eq:AK_duality} for a general marked surface in terms of stated skein algebras, which also gives a reformulation of Allegretti's duality map \eqref{eq:Allegretti_duality} in this language.

For any 
marked surface $\Sigma$, we construct a quantum duality map
\begin{align}\label{eq:our_duality}
    \bI_\A: \A_\Sigma(\bZ^\sfT) \to \cO_v(\X_\Sigma^\circ)
\end{align}
with $v=q^{-2}$, generalizing \eqref{eq:AK_duality}. The main scope of this paper is the case without interior marked points (which is opposite to \cite{AK}), with a focus on the compatibility with the bracelets construction \eqref{eq:bracelets_duality} on the dual side. The case with interior marked points will be discussed in \cite{IKp}, with a refinement of \eqref{eq:AK_duality} into $\cO_v(\X_\Sigma)$. 
The following is our main theorem:
\begin{introthm}[Ensemble compatibility: \cref{thm:compatibility}]\label{introthm:compatibility}
When $\Sigma$ has no interior marked points,
the map \eqref{eq:our_duality} fits into the following commutative diagram (in the category of sets):
\begin{equation*}
    \begin{tikzcd}
    \A_\Sigma(\bZ^\sfT) \ar[r,"\bI_\A"] \ar[d,"\check{p}_\Sigma^\sfT"'] & \cO_v(\X_\Sigma) \ar[d,"p_\Sigma^\ast"] \\
    \X_\Sigma(\bZ^\sfT) \ar[r,"\bI_\X"'] & \cO_q(\A_\Sigma).
    \end{tikzcd}
\end{equation*}
Here $v=q^{-2}$, $\bI_\X$ is the map \eqref{eq:bracelets_duality}, and $\check{p}_\Sigma^\sfT$ denotes the tropicalization of the Langlands dual ensemble map \cite{Ish}. 
\end{introthm}
Our construction factors through (stated) skein algebras, where the asserted diagram is decomposed as 

\begin{equation*}
    \begin{tikzcd}
    \A_\Sigma(\bZ^\sfT) \ar[r,"S_\A"] \ar[d,"\check{p}_\Sigma^\sfT"'] & \sSk{\Sigma}_{\congr} \ar[dr,dashed] \ar[r,"\mathrm{Tr}_\Sigma"] \ar[d,"\Phi_\Sigma"] & \cO_v(\X_\Sigma) \ar[d,"p_\Sigma^\ast"] \\
    \X_\Sigma(\bZ^\sfT) \ar[r,"S_\X"'] & \Sk{\Sigma}[\partial^{-1}] \ar[r,"\mathrm{Cut}_\Sigma"'] & \cO_q(\A_\Sigma).
    \end{tikzcd}
\end{equation*}
Here:
\begin{itemize}
    \item $\sSk{\Sigma}_{\congr} \subset \sSk{\Sigma}$ denotes the \emph{congruent subalgebra} (\cref{def:congruent}). 
    \item The map $\Phi_\Sigma$ is the restriction of the inverse of the isomorphism \eqref{eq:LY_isom}. 
    \item The horizontal maps $S_\A$ and $S_\X$ are appropriate ``skein lifting" maps (\cref{subsec:lifting}). 
    \item The map $\mathrm{Tr}_\Sigma$ is based on the quantum trace map of L\^e \cite{Le_triangular}
    , and $\mathrm{Cut}_\Sigma$ is the map \eqref{eq:cutting}. 
\end{itemize}
The quantum duality maps are defined to be $\bI_\A:=\mathrm{Tr}_\Sigma\circ S_\A$ and $\bI_\X:=\mathrm{Cut}_\Sigma \circ S_\X$. To our best knowledge, the construction of $S_\A$ is new. The classical aspect ($q^{1/2}=1$) of \cref{introthm:compatibility} is studied in \cite{Ish}.

We remark that our construction agrees with the general construction based on the quantum theta functions of Davison--Mandel \cite{DM}, thanks to the recent result of Mandel--Qin \cite{MQ}. 
When $\Sigma$ is a marked disk, we see that the map $\bI_\A$ also recovers Allegretti's construction in \cref{rem:Allegretti}. 

A similar compatibility result has been also obtained in the work of L\^e--Yu \cite{LY22}. Roughly speaking, their result is concerned with the top-right triangle involving the dashed arrow in the diagram above. A precise relation is discussed in \cref{rem:LY_compatibility}.

With a help of the results of Mandel--Qin \cite{MQ}, we get the following ``skein realization" of the quantum cluster Poisson algebra:
\begin{introthm}[\cref{thm:skein_q-Poisson_isom}]\label{introthm:isomorphism}
The quantum trace map gives an algebra isomorphism 
\begin{align*}
    \mathrm{Tr}_\Sigma: \sSk{\Sigma}_\congr \xrightarrow{\sim} \cO_v(\X_\Sigma)
\end{align*}
if $\bM_\circ=\emptyset$ and $|\bM_\partial|\geq 2$.
\end{introthm}

We remark that the results on the quantum theta basis is only needed for the surjectivity of $\mathrm{Tr}_\Sigma$. One can prove that the image $S_\A(\A_\Sigma(\bZ^\sfT)) \subset \sSk{\Sigma}_\congr$ of the skein lifting map forms a $\bZ[q^{\pm 1}]$-basis (\cref{cor:basis}), without referring to the quantum theta basis.

\subsection{Skein proof of structure constant positivity for small surfaces}
D. Thurston \cite{Th} considered the positivity of the structure constants of specific bases of (the classical limits of) skein algebras over $\BZ[q^{\pm 1/2}]$. 
In particular, Thurston conjectured that 
the \emph{bracelets bases} defined by the Chebyshev polynomials of the first kind of ordinary skein algebras are positive bases, i.e. the structure constants are in $\BZ_{\geq 0}[q^{\pm 1/2}]$.

This conjecture is true for commutative skein algebras and the skein algebra of the torus from the product-to-sum formula in \cite{FG00}. 
Moreover, the cases of the $4$-punctured sphere and the $1$-punctured torus were solved affirmatively by Bousseau \cite{Bo} in the context of mirror symmetry. 
Based on Thurston's work, L$\hat{{\rm e}}$ \cite{Le_positivity} gave the definition of bracelets bases for Muller skein algebras and tried to clarify a necessary condition to be positive bases.  
From these results, the bracelets bases have been recognized as candidates for positive bases of Muller skein algebras. 
Similar to Chebyshev polynomials the first kind, there is the positivity conjecture for those of the second kind. In \cite{LTY}, L$\hat{{\rm e}}$--Thurston--Yu gave upper and lower bounds of positive bases for Muller skein algebras in an appropriate sense, where the lower bound is the bracelets basis and the upper bound is defined using the Chebyshev polynomials of the second kind. 
Later, Queffelec showed that the upper bound is actually a positive basis for ordinary skein algebras using foams \cite{Qu}.

In the context of quantum duality maps, we say that $\bI_V$ for $V \in \{\A,\X\}$ is \emph{strongly positive} if its image has structure constants in $\bZ_{\geq 0}[q^{\pm 1/2}]$. 
In this paper, we give a proof of the positivity for the bracelets bases of the Muller skein algebras of marked disks and the annulus with two marked points, and the strong positivity results as consequences:

\begin{introthm}[\cref{thm_skein,thm_S_A}]\label{thm_intro}
Let $(\Sigma,\bM)$ be a marked disk with at least 3 special points or a marked annulus with exactly one marked point on each boundary component. Then we have the following.
\begin{enumerate}
    \item The bracelets basis $\mathsf{B}^b(\Sigma,\bM)$ (\cref{def:bracelets}) of the Muller skein algebra $\Sk{\Sigma}$ is positive.
    \item The quantum duality maps $\bI_\A$ and $\bI_\X$ are strongly positive. 
\end{enumerate}

\end{introthm}

While this result is an immediate corollary of the general results proved in \cite{MQ}, one of the advantages of this paper is that we can see explicit structure constants of the bracelets bases. In particular, we see that the 
expansions of two basis elements 
contain the Chebyshev polynomials of the second kind. From this, we expect a
background behind the positivity result originating from the representation theory of ${\rm SL}_2(\bC)$.

\subsection*{Acknowledgements} The authors would like to express their gratitude to Dylan Allegretti, Hyun Kyu Kim and Wataru Yuasa for helpful comments. 
H. K. is grateful to Thang T. Q. L\^e for valuable comments before the authors began the project. T. I. was partially supported by JSPS KAKENHI Grant-in-Aid for Research Activity Start-up (No.~20K22304). 
H.K. was partially supported by JSPS KAKENHI Grant Numbers JP22K20342, JP23K12976.

\section{Marked surfaces and associated quantum tori}

\subsection{Marked surfaces and their triangulations}\label{subsec:notation_marked_surface}
A marked surface $(\Sigma,\bM)$ is a compact oriented surface $\Sigma$ together with a fixed non-empty finite set $\bM \subset \Sigma$ of \emph{marked points}. 
A marked point is called a \emph{puncture} if it lies in the interior of $\Sigma$, and a \emph{special point} otherwise. 
Let $\bP=\bP(\Sigma)$ (resp. $\bM_\partial=\bM_\partial(\Sigma)$) denote the set of punctures (resp. special points), so that $\bM=\bP \sqcup \bM_\partial$. 
Let $\Sigma^*:=\Sigma \setminus \bP$. 
We always assume the following conditions:
\begin{enumerate}
    \item[(S1)] Each boundary component (if exists) has at least one marked point.\footnote{There is another convention that a boundary component without marked points is regarded as a puncture, in which case one assumes $\bM \subset \partial\Sigma$.}
    \item[(S2)] $-2\chi(\Sigma^*)+|\bM_\partial| >0$.
    \item[(S3)] $(\Sigma,\bM)$ is not a once-punctured monogon.
\end{enumerate}
When the choice of $\bM$ is clear from the context, we simply denote $(\Sigma,\bM)$ by $\Sigma$. 
We say that $\Sigma$ is \emph{unpunctured} if $\bM_\circ=\emptyset$. 

We call a connected component of the punctured boundary $\partial^\ast \Sigma:=\partial\Sigma\setminus \bM_\partial$ a \emph{boundary interval}. The set of boundary intervals is denoted by $\bB=\bB(\Sigma)$. Note that $|\bB|=|\bM_\partial|$. 
By convention, we endow each boundary interval $\alpha \in \bB$ with the orientation induced from $\partial\Sigma$. Let $m^+_\alpha$ (resp. $m^-_\alpha$) denote its initial (resp. terminal) marked point. 

An \emph{ideal arc} in $(\Sigma,\bM)$ is 
an immersed arc in $\Sigma^\ast$ with endpoints in $\bM$ having no self-intersections except possibly at its endpoints, and not contractible in $\Sigma^\ast$. 
An \emph{ideal triangulation} is a triangulation $\tri$ of $\Sigma$ whose set of $0$-cells (vertices) coincides with $\bM$, $1$-cells (edges) being ideal arcs. 
In this paper, we always consider an ideal triangulation without any \emph{self-folded triangle} where two of its sides are identified. 
The conditions (S1)--(S3) ensure the existence of such an ideal triangulation. See, for instance, \cite[Lemma 2.13]{FST}. 

For an ideal triangulation $\tri$, denote the set of edges (resp. interior edges, triangles) by $e(\tri)$ (resp. $e_{\interior}(\tri)$, $t(\tri)$). Since the boundary intervals belong to any ideal triangulation, we always have $e(\tri)=e_{\interior}(\tri) \sqcup \bB$. By a computation on the Euler characteristics, we get
\begin{align*}
    &|e(\tri)|=-3\chi(\Sigma^*)+2|\bM_\partial|, \quad |e_{\interior}(\tri)|=-3\chi(\Sigma^*)+|\bM_\partial|, \\
    &|t(\tri)|=-2\chi(\Sigma^*)+|\bM_\partial|.
\end{align*}
Let $n(\Sigma):=|e_{\interior}(\tri)|$ denote the number of interior edges. 

\subsection{Quantum torus}\label{def:q-torus}
For a formal quantum parameter $q^{1/2}$, let $\bZ_q:=\bZ[q^{\pm 1/2}]$. 
Let $\Lambda$ be a lattice equipped with a skew-symmetric form $\omega:\Lambda \times \Lambda \to \bZ$. Then the associated \emph{based quantum torus} over $\bZ_q$ is the associative $\bZ_q$-algebra $\bT_{(\Lambda,\omega)}$ having
\begin{itemize}
    \item a free $\bZ_q$-basis $B_\lambda$ parametrized by $\lambda \in \Lambda$, and
    \item the multiplication given by $B_\lambda\cdot B_\mu = q^{\omega(\lambda,\mu)/2} B_{\lambda+\mu}$. 
\end{itemize}
The \emph{rank} of $\bT_{(\Lambda,\omega)}$ is defined to be the rank of $\Lambda$. The map $B:\Lambda \to \bT_{(\Lambda,\omega)}$, $\lambda \mapsto B_\lambda$ is called the \emph{framing} of $\bT_{(\Lambda,\omega)}$. 
The $q$-commutation relation of the form 
\begin{align*}
    B_\lambda\cdot B_\mu = q^{\omega(\lambda,\mu)} B_{\mu}\cdot B_{\lambda}
\end{align*}
is frequently used. For $\lambda_1,\dots,\lambda_k \in \Lambda$, the element $B_{\lambda_1+\cdots+\lambda_k}$ is called the \emph{Weyl normalized product} of $B_{\lambda_1},\dots,B_{\lambda_k}$. It is also written as
\begin{align}\label{eq:Weyl_ordering}
    [B_{\lambda_1}\cdots B_{\lambda_k}]:=B_{\lambda_1+\cdots+\lambda_k} = q^{-\frac 1 2 \sum_{i < j}\omega_{ij}} B_{\lambda_1}\cdots B_{\lambda_k},
\end{align}
where $\omega_{ij}:=\omega(\lambda_i,\lambda_j)$. It is designed as the non-commutative product which is independent of the order. 

\subsection{Matrices and quantum tori associated with triangulations}
To (the isotopy class of) an ideal triangulation $\tri$ of a marked surface $\Sigma$, associated are the following matrices and the corresponding quantum tori.

\subsubsection{The Chekhov--Fock square-root algebra}\label{subsub:exchange}
For a 
triangle $T$ of $\tri$ and two edges $\alpha,\beta \in e(\tri)$, define
\begin{align*}
    \ve_{\alpha\beta}(T):= \begin{cases}
        1 & \mbox{if $T$ has $\alpha$ and $\beta$ as its consecutive edges in the clockwise order}, \\
        -1 & \mbox{if the same holds with the counter-clockwise order}, \\
        0 & \mbox{otherwise}.
    \end{cases}
\end{align*}
Then we set $\ve_{\alpha\beta}^\tri:=\sum_T \ve_{\alpha\beta}(T)$, where $T$ runs over all triangles of $\tri$. We need a modification if $\tri$ has self-folded triangles: see \cite[Definition 4.1]{FST}.
\begin{dfn}
We call $\ve^\tri:=(\ve^\tri_{\alpha\beta})_{\alpha,\beta \in e(\tri)}$ the \emph{exchange matrix} associated with $\tri$ \cite{FST}. It is the same as the \emph{face matrix} in \cite{Le_quantum_teich}.
\end{dfn}

Consider the lattice $N^\tri:=\bigoplus_{\alpha \in e(\tri)} \bZ \sfe^\tri_\alpha$ freely generated by the formal symbols $\sfe^\tri_\alpha$ parametrized by the edges $\alpha \in e(\tri)$. The exchange matrix induces a skew-symmetric form $\omega_Z: N^\tri \times N^\tri \to \bZ$ by
\begin{align*}
    \omega_Z(\sfe^\tri_\alpha,\sfe^\tri_\beta):=-\ve^\tri_{\alpha\beta}.
\end{align*}

\begin{dfn}\label{def:CF_torus}
Let $\cZ^q_\tri:=\bT_{(N^\tri,\omega_Z)}$ denote the associated based quantum torus over $\bZ_q$, equipped with a framing $\mathbf{Z}^\tri:N^\tri \to \cZ^q_\tri$, $\lambda \mapsto \mathbf{Z}^\tri(\lambda)$. The quantum torus $\cZ_\tri^q$ is called the \emph{Chekhov--Fock square-root algebra} \cite{BW11}. The elements $Z^\tri_\alpha:=\mathbf{Z}^\tri(\sfe_\alpha^\tri)$ satisfy the $q$-commutation relation
\begin{align}\label{eq:Le_relation}
        Z^\tri_\alpha Z^\tri_\beta = q^{-\ve^\tri_{\alpha\beta}} Z^\tri_\beta Z^\tri_\alpha.
\end{align}
for $\alpha,\beta \in e(\tri)$. 
\end{dfn}

\begin{rem}
The $q$-commutation relation \eqref{eq:Le_relation} agrees with the defining relation of the quantum torus $\mathcal{Y}(\tri)$ introduced in \cite[Section 6.1]{Le_quantum_teich}. 
\end{rem}

\begin{dfn}
An element $\lambda=\sum_{\alpha} \lambda_\alpha \sfe_\alpha^\tri \in N^\tri$ is said to be \emph{balanced} if $\lambda_{\alpha_1}+\lambda_{\alpha_2}+\lambda_{\alpha_3} \in 2\bZ$ for any triple $(\alpha_1,\alpha_2,\alpha_3)$ that constitute a triangle in $\tri$. Balanced elements form a sub-lattice $K^\tri \subset N^\tri$. We call the subalgebra $(\cZ_\tri^q)_{\mathrm{bl}} \subset \cZ_\tri^q$ corresponding to $K^\tri$ the \emph{balanced Chekhov--Fock (square-root) algebra} associated with $\tri$. 
\end{dfn}

\subsubsection{Quantum cluster Poisson torus}
Consider the same lattice $N^\tri$ but with the skew-symmetric form $\omega_X:N^\tri \times N^\tri \to \bZ$ defined by
\begin{align*}          \omega_X(\sfe^\tri_\alpha,\sfe^\tri_\beta):=2\ve^\tri_{\alpha\beta}.
\end{align*}
We introduce a different quantum parameter $v$ which satisfies $v=q^{-2}$. 

\begin{dfn}\label{def:Poisson_torus}
The associated based quantum torus $\X^v_\tri:=\bT_{(N^\tri,\omega_X)}$ over $\bZ_v$, equipped with a framing $\mathbf{X}^\tri:N^\tri \to \X^v_\tri$, $\lambda \mapsto \mathbf{X}^\tri(\lambda)$, is called the \emph{quantum cluster Poisson torus}. The elements $X^\tri_\alpha:=\mathbf{X}^\tri(\sfe_\alpha^\tri)$ satisfy the $q$-commutation relation
\begin{align}\label{eq:FG_relation}
        X^\tri_\alpha X^\tri_\beta = v^{2\ve^\tri_{\alpha\beta}} X^\tri_\beta X^\tri_\alpha.
\end{align}
for $\alpha,\beta \in e(\tri)$. 
\end{dfn}

\begin{rem}
The $v$-commutation relation \eqref{eq:FG_relation} agrees with the Fock--Goncharov convention \cite[(49)]{FG09} by interpreting our $v$ as $q$. It is crucial to have even $v$-exponents to get a correct quantum cluster Poisson transformation, in view of the $v$-difference relation \eqref{eq:q-difference}.
\end{rem}

As the name indicates, the variables $Z_\alpha^\tri$ should be viewed as ``square-roots'' of (the inverses of) $X_\alpha^\tri$. 

\begin{lem}\label{eq:emb_to_CF}
We have an algebra embedding
\begin{align*}
    \iota:\X_\tri^v \to (\cZ_\tri^q)_{\mathrm{bl}}, \quad X_\alpha^\tri \mapsto (Z_\alpha^\tri)^{-2}
\end{align*}
with $v=q^{-2}$. 
\end{lem}

\begin{proof}
For $\alpha,\beta \in e(\tri)$, we have
\begin{align*}
    \iota(X_\alpha^\tri) \iota(X_\beta^\tri) = (Z_\alpha^\tri)^{-2} (Z_\beta^\tri)^{-2} = q^{-4\ve_{\alpha\beta}^\tri} (Z_\beta^\tri)^{-2} (Z_\alpha^\tri)^{-2} = v^{2\ve_{\alpha\beta}^\tri} \iota(X_\beta^\tri) \iota(X_\alpha^\tri)
\end{align*}
as desired. 
\end{proof}

\begin{rem}\label{rem:q-parameters}
Let us explain the relations to the several quantum parameters introduced in the literature. The parameter $q$ in this paper will be the same as $q$ in \cite{Le_quantum_teich}. It is the parameter both for the stated and Muller skein algebras. 

In the foundational paper \cite{BW11}, they introduced three parameters $A,q,\omega$. Let us denote them by $\hat{A},\hat{q},\hat{\omega}$. 
Then we have the relations
\begin{align*}
    &\hat{q} = \hat{\omega}^4, \quad \hat{A} = \hat{\omega}^{-2}, \quad \hat{q} = \hat{A}^{-2}, \\
    &v = \hat{q}, \quad q=\hat{A}, \quad v=q^{-2}.
\end{align*}
In \cite{BW11}, $\hat{A}$ is the skein parameter, while $\hat{q}$ and $\hat{\omega}$ are parameters for the quantum cluster Poisson algebra (a.k.a. the Chekhov--Fock algebra) and its square-root version, respectively. 
\end{rem}

\subsubsection{Quantum cluster $K_2$-torus}\label{subsub:compatibility}
Suppose that $\Sigma$ has no punctures. For $\alpha \in e(\tri)$, let $\alpha_1,\alpha_2$ denote its two ends (with an arbitrary numbering). For two edges $\alpha,\beta \in e(\tri)$, define 
\begin{align*}
    \pi_{\alpha_i,\beta_j}:=\begin{cases}
    1 & \mbox{if $\alpha_i$ is clockwise to $\beta_j$ at a common special point}, \\
    -1 & \mbox{if $\alpha_i$ is counter-clockwise to $\beta_j$  at a common special point}, \\
    0 & \mbox{otherwise}
    \end{cases}
\end{align*}
and set $\pi_{\alpha\beta}:=\sum_{i,j=1}^2 \pi_{\alpha_i,\beta_j}$. 

\begin{dfn}\label{def:comp_matrix}
The matrix $\Pi^{\tri}=(\pi_{\alpha\beta})_{\alpha,\beta \in e(\tri)}$ is called the \emph{compatibility matrix} \cite{Muller}. It is the same as the \emph{vertex matrix} in \cite{Le_quantum_teich}.
\end{dfn}

Consider the lattice $M^\tri:=\bigoplus_{\alpha \in e(\tri)} \bZ \sff^\tri_\alpha$ freely generated by the formal symbols $\sff^\tri_\alpha$ parametrized by the edges $\alpha \in e(\tri)$. The compatibility matrix induces a skew-symmetric form $\omega_A: M^\tri \times M^\tri \to \bZ$ by
\begin{align*}
    \omega_A(\sff^\tri_\alpha,\sff^\tri_\beta):=\pi_{\alpha\beta}.
\end{align*}
We regard $M^\tri$ as the dual lattice of $N^\tri$ via the pairing $\langle \sff_\alpha^\tri,\sfe_\beta^\tri\rangle:=\delta_{\alpha,\beta}$. 

\begin{dfn}\label{def:K2_torus}
The associated based quantum torus $\A^q_\tri:=\bT_{(M^\tri,\omega_A)}$ over $\bZ_q$, equipped with a framing $\mathbf{A}^{\!\tri}:M^\tri \to \A^q_\tri$, $\lambda \mapsto \mathbf{A}^{\!\tri}(\lambda)$, is called the \emph{quantum cluster $K_2$-torus}.  
The elements $A_\alpha:=\mathbf{A}^{\!\tri}(\sff_\alpha^\tri)$ satisfy the $q$-commutation relation
\begin{align*}
        A_\alpha A_\beta = q^{\pi_{\alpha\beta}} A_\beta A_\alpha.
\end{align*}
for $\alpha,\beta \in e(\tri)$. 
\end{dfn}
Here we omitted the superscript $\tri$ from the cluster $K_2$-variables. 
\section{Skein algebras}\label{sec:skein}

In what follows, let $\CR$ be a unital commutative ring 
with a distinguished invertible element $q^{1/2} \in \cR$. and let $(\Sigma,\BM)$ be an oriented marked surface. Recall the notation $\Sigma^\ast=\Sigma\setminus \BMc$.
Let $\Int \Sigma^\ast$ denote the interior of $\Sigma^\ast$.

\subsection{$\BMp$-tangles and $\BB$-tangles}
In this paper, we use the following two kinds of tangles in the 3-manifolds $\Sigma \times (-1,1)$ with different boundary conditions. 

\begin{dfn}
An {\em $\BMp$-tangle} in $\Sigma\times(-1,1)$ is a properly embedded unoriented $1$-dimensional submanifold $\gamma$ of $\Sigma^\ast \times (-1,1)$  equipped with a framing such that 
\begin{enumerate}
\item $\partial \gamma \subset \BMp\times (-1,1)$ 
\item the framing at each of $\partial \gamma$ is vertical, i.e. the normal vector points towards the direction of $1$. 
\end{enumerate}

Two $\BMp$-tangles are {\it isotopic} if they are isotopic in the class of $\BMp$-tangles. 
\end{dfn}

\begin{dfn}
A {\em $\BB$-tangle} in $\Sigma\times (-1,1)$ is a properly embedded  unoriented $1$-dimensional submanifold $\gamma$ of $(\Sigma\setminus \BM) \times (-1,1)$ equipped with a framing such that the endpoints have distinct heights over each boundary interval, and the framing at each of $\partial\gamma$ is vertical. 
Two $\BB$-tangles are {\it isotopic} if they are isotopic in the class of $\BB$-tangles. 
\end{dfn}

\paragraph{\textbf{Picture conventions.}}
For $(p,t)\in \Sigma\times (-1,1)$, we call $t$ the {\it height} of $(p,t)$. 
\begin{description}
\item[$\BMp$-tangle diagram] Every $\BMp$-tangle in $\Sigma\times(-1,1)$ is isotopic to one with a vertical framing. 
An $\BMp$-tangle in $\Sigma\times(-1,1)$ with a vertical framing is in {\em general position} if its image under the vertical projection $\Sigma\times(-1,1)\to \Sigma\times\{0\}=\Sigma$ has only transverse points as singular points, especially only transverse double points in $\Int \Sigma^\ast$. 
In figures, we will show an $\BMp$-tangle in general position by its image under the vertical projection, assigning over-/under-passing information to each singular point according to the heights. 
The resulting diagram is refered to as the {\em $\BMp$-tangle diagram} of the $\BMp$-tangle.

\item[$\BB$-tangle diagram]
Every $\BB$-tangle in $\Sigma\times(-1,1)$ is isotopic to one with a vertical framing. 
A $\BB$-tangle in $\Sigma\times(-1,1)$ with a vertical framing is in {\em general position} if its image under the vertical projection has only transverse double points as singular points in $\Int \Sigma^\ast$, and the images of endpoints on boundary intervals are distinct. 
Similarly to $\bM_\partial$-tangles, we represent a $\BB$-tangle in general position by its {\em $\BB$-tangle diagram}, the image under the vertical projection with over-/under-passing information to each double point. 
Moreover, we assign the total order to the endpoints on each boundary edge with respect to their heights. 
If a boundary edge is assigned an orientation in a picture, the heights of the shown arcs incident to that edge are assumed to be increasing with respect to the orientation. 
However, the heights are not always increasing without a specified boundary orientation. 
\item[local relation] 
In each local relation (such as (A)--(F) below), the parts bounded by the dotted curves and the horizontal or vertical lines will stand for the same region of $\Sigma$ in both sides. 
The thick horizontal lines (as in (C) and (D)) are parts of $\partial \Sigma$, and the bullets on them are special points.
In each relation, the \textcolor{red}{red} strands
are parts of $\BMp$- or $\BB$-tangle diagrams, which are assumed to be the same outside the shown region. 
\end{description}

\subsection{Muller skein algebras}
\begin{dfn}[The Muller skein module/algebra \cite{Muller,BL}]\label{dfn_Muller}
The {\it Muller skein module} of a marked surface $(\Sigma, \BM)$ is the $\CR$-module freely spanned by the isotopy classes of $\BMp$-tangles in $(\Sigma, \BM)$ subject to the interior skein relations (A), (B) and the following relations (C) and (D):
\begin{align*}
&({\rm C})\ q^{-1/2}\begin{array}{c}\includegraphics[scale=0.27]{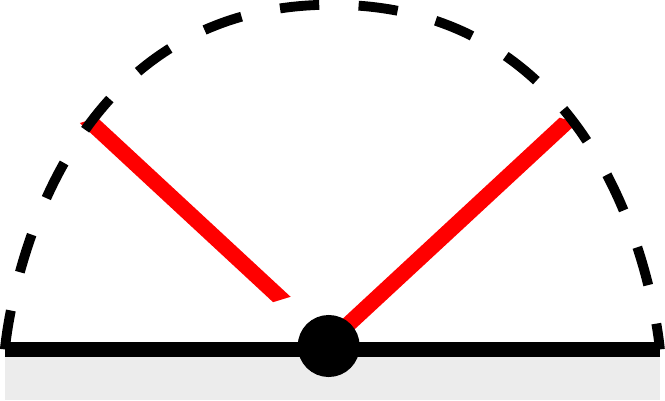}\end{array}=\begin{array}{c}\includegraphics[scale=0.27]{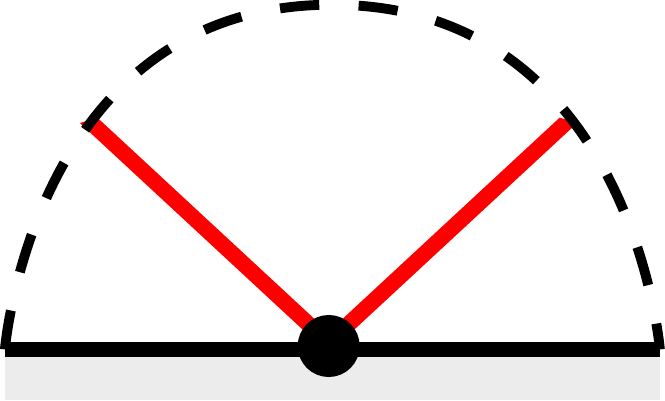}\end{array}=q^{1/2}\begin{array}{c}\includegraphics[scale=0.27]{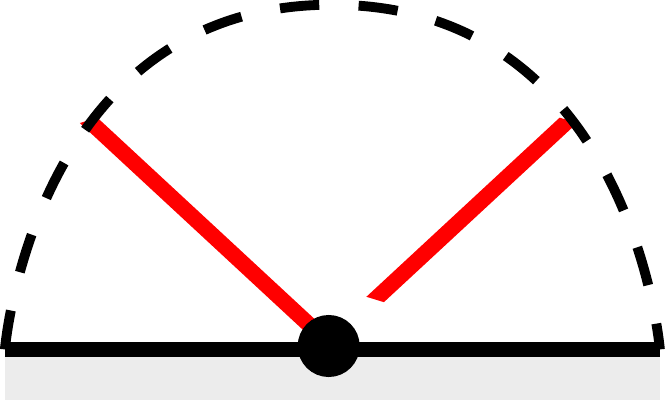}\end{array},\\
&({\rm D})\ \begin{array}{c}\includegraphics[scale=0.27]{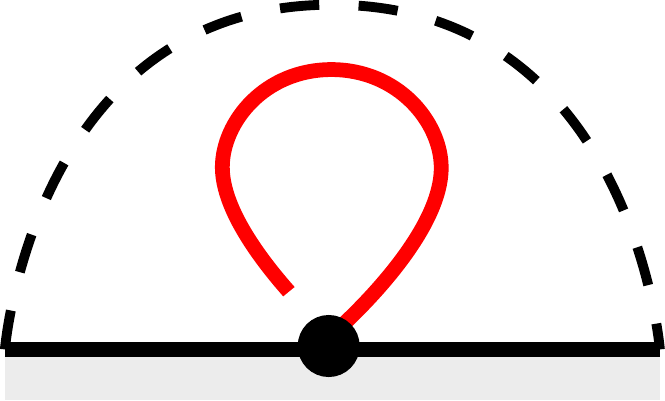}\end{array}=0=\begin{array}{c}\includegraphics[scale=0.27]{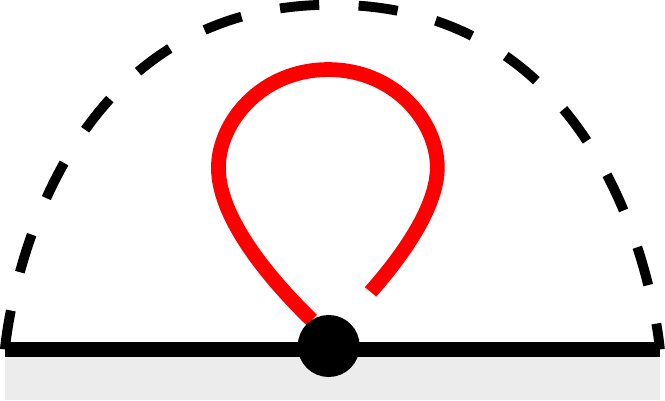}\end{array}
\end{align*}
Note that the middle term in (C), which we call the \emph{simultaneous crossing}, is not an $\BMp$-tangle diagram: the relation (C) can be understood as its defining relation. 
A product of two $\BMp$-tangles $\al$ and $\beta$, denoted by $\al\beta$, is defined by stacking $\al$ over $\beta$ and then rescaling them to be in $\Sigma \times (-1,1)$.  
With this multiplication, the Muller skein module becomes an associative algebra $\sSsq(\bM)$, which we call the \emph{Muller skein algebra}. 
\end{dfn}

The defining relations (A)--(D) of Muller skein modules implies the \emph{Reidemeister moves} shown in Figure \ref{Reidemeister}. 
See \cite[Proposition 3.2]{Muller} for more details. 

\begin{figure}[ht]
    \centering
    \includegraphics[width=400pt]{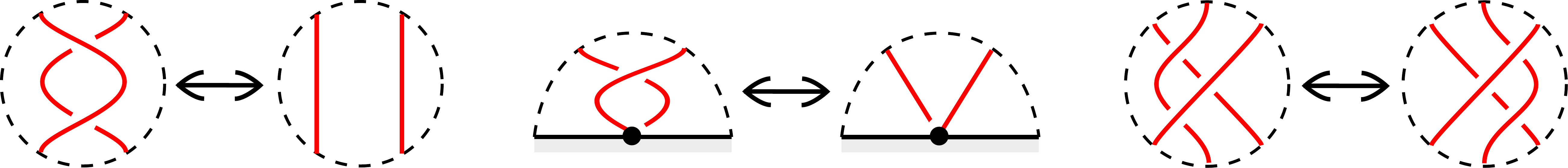}
    \caption{Left: Reidemeister move II, Middle: Reidemeister move II', Right: Reidemeister move III}\label{Reidemeister}
\end{figure}

\paragraph{\bf{Muller's basis of Muller skein algebras}}
A {\it multicurve} on $\Sigma$ is an immersion $\al\colon (\bigsqcup [0,1]) \sqcup (\bigsqcup \,S^1)\to \Sigma^\ast$ such that 
\begin{enumerate}
\item the restriction  of $\al$ on $(\bigsqcup (0,1)) \sqcup (\bigsqcup \,S^1)$ is a proper embedding into $\Sigma^\ast$, 
\item $\al\left(\bigsqcup\, \partial [0,1]\right)\subset \BMp$,
\item at each point of $\BMp$, the ends of $\al$ incident to the special point are totally ordered. 
\end{enumerate}

A multicurve on $\Sigma$ is {\it simple} if none of its connected component bounds an embedded disk in $\Sigma$.

For each multicurve $\al$ on $\Sigma$, there is an $\BMp$-tangle in $\Sigma\times (-1,1)$ with vertical framing such that its image under the vertical projection coincides with $\alpha$ and the over-/under-passing information at each special point is compatible with the ordering of the ends of $\al$ there.
In this way, we will occasionally regard a multicurve as an $\BMp$-tangle in $\Sigma\times (-1,1)$.

When we regard a multicurve $\al$ as an element of $\sSsq(\BM)$, 
it is equal (up to a power of $q^{1/2}$) to the product of loops and arcs on $\Sigma$ 
with simultaneous crossings at $\BMp$
defined by (C). 
The latter product is called the Weyl normalization of the multicurve $\al$, and is denoted by $[\al]$.

\begin{dfn}\label{prop:Muller_basis}
Let $(\Sigma,\bM)$ be a marked surface. 
\begin{enumerate}
    \item The set
    $$
    \sfB(\Sigma,\BM):=\{[\al]\,|\, \text{$\al$ is a simple multicurve on $\Sigma$}\}
    $$
    forms an $\cR$-basis of $\sSsq(\BM)$ \cite[Lemma 4.1]{Muller}, which we call the \emph{Muller's basis} (a.k.a. \emph{bangles basis} \cite{Th}).  
    \item Recall the Chebyshev polynomials of the first kind given by the following initial data and the recurrence relation: 
    \begin{align*}
        T_0(x)=2,\quad T_1(x)=x,\quad T_{n+2}(x)=x T_{n+1}(x)-T_{n}(x)\quad (n\geq0).
    \end{align*}
    Then we get another $\cR$-basis $\sfB^b(\Sigma,\BM)$ of $\sSsq(\BM)$, called the \emph{bracelets basis}, obtained by replacing each $k$-parallel copy of a simple loop $\gamma$ with the element $T_k(\gamma) \in \Sk{\Sigma}$ in each $[\al]\in \sfB(\Sigma,\BM)$, where $k$ is the maximal number of parallel copies of $\gamma$ in $\al$. See \cref{sec:positivity} for more details.
\end{enumerate}
\end{dfn}

Let us assume that $(\Sigma,\bM)$ is an unpunctured marked surface. 
Given a triangulation $\tri$ of $\Sigma$, each edge of $e(\tri)$ (or any subset of $e(\tri)$) can be naturally regarded as a simple multicurve. Let us write $A_\al:=[\al] \in \sSsq(\BM)$ for $\al \in e(\tri)$. 
Then it can be verified from the relation (C) that
\begin{align}\label{eq:skein_cluster_A_relation}
    A_\al A_\beta = q^{\pi_{\al\beta}} A_\beta A_\al
\end{align}
holds for all $\al,\beta \in e(\tri)$, where $\pi_{\al\beta}$ is the compatibility matrix (\cref{def:comp_matrix}). The Weyl normalized product $[A_\al A_\beta]$ in the sense of \eqref{eq:Weyl_ordering} is the same as the Weyl normalization of the multicurve $\al \cup \beta$. 

For an ideal triangulation $\tri$, let $\sSsq(\BM)[\tri^{-1}]$ denote the localization along the multiplicative set generated by $\{A_\al \mid \al \in e(\tri)\}$.  

\begin{thm}[Muller, {\cite[Theorem 6.14]{Muller}}]
For an unpunctured marked surface $(\Sigma, \bM)$, the localization $\sSsq(\BM)[\tri^{-1}]$ is injective Ore. 
Moreover, we have an $\CR$-algebra isomorphism 
$\sSsq(\BM)[\tri^{-1}]\cong \CA^q_\tri$.
\end{thm}

\begin{dfn}\label{def:cutting}
In particular, we have an embedding 
\begin{align*}
    \mathrm{Cut}_\tri: \Sk{\Sigma}[\partial^{-1}] \to \A_\tri^q,
\end{align*}
where $\Sk{\Sigma}[\partial^{-1}]$ denotes the localization along the multiplicative set generated by $\{A_\al \mid \al \in \BB\}$. We call this map the \emph{cutting map}.
\end{dfn}
For later use, we reinterpret the cutting map as a composite
\begin{align}\label{eq:cut_composite}
    \mathrm{Cut}_\tri = \mathrm{Cut}_{\alpha_1}\circ \mathrm{Cut}_{\alpha_2} \circ \cdots\circ \mathrm{Cut}_{\alpha_{n(\Sigma)}},
\end{align}
where $e_{\interior}(\tri)=\{\alpha_1,\alpha_2,\dots,\alpha_{n(\Sigma)}\}$ and each $\mathrm{Cut}_{\alpha_i}$ is defined as follows.

For an ideal arc $\alpha$ and a multicurve $C$ on $\Sigma$, the product $A_\alpha^{\bi(\alpha,C)} \cdot C$ can be written an $\cR$-linear combination of multicurves on $\Sigma\setminus \alpha$. Dividing the resulting element by $A_\alpha^{\bi(\alpha,C)}$, we get
\begin{align*}
    \mathrm{Cut}_\al:\Sk{\Sigma}[\partial^{-1}] \to \Sk{\Sigma}[\partial^{-1},\al^{-1}].
\end{align*}
For two ideal arcs $\alpha,\beta$ which do not intersect except for their endpoints, we have $\mathrm{Cut}_\alpha \circ \mathrm{Cut}_\beta =\mathrm{Cut}_\beta \circ \mathrm{Cut}_\alpha$.
Then the composite \eqref{eq:cut_composite} does not depend on the order of composition, and it cuts any multicurve $C$ into a Laurent polynomial of ideal arcs in $e(\tri)$, which lives in $\Sk{\Sigma}[\tri^{-1}] \cong \A_\tri^q$. This description of $\mathrm{Cut}_\tri$ will be in parallel with that for the splitting map $\Psi_\tri$ (\cref{def:q-trace}).

\subsection{Stated skein algebras}

A {\em state} on a $\BB$-tangle $\gamma$ is a map $s:\partial \gamma \to \{+,-\}$.
Then the pair $(\gamma,s)$ is called a \emph{stated $\BB$-tangle}. 
If there is no confusion, we abbreviate it to $\gamma$. 
Similarly, a \emph{stated $\BB$-tangle diagram} is a $\BB$-tangle diagram $\gamma$ equipped with a state $s\colon \partial \gamma\to \{+,-\}$.

\begin{dfn}[The stated skein module/algebra \cite{Le_triangular}]
The {\em stated skein module} of a marked surface $(\Sigma,\BM)$ is the $\CR$-module generated by all isotopy classes of stated $\BB$-tangles in $\Sigma\times (-1,1)$, subject to the interior skein relations (A) and (B) and the following relations: 
\begin{align*}
&({\rm E}) \begin{array}{c}\includegraphics[scale=0.25]{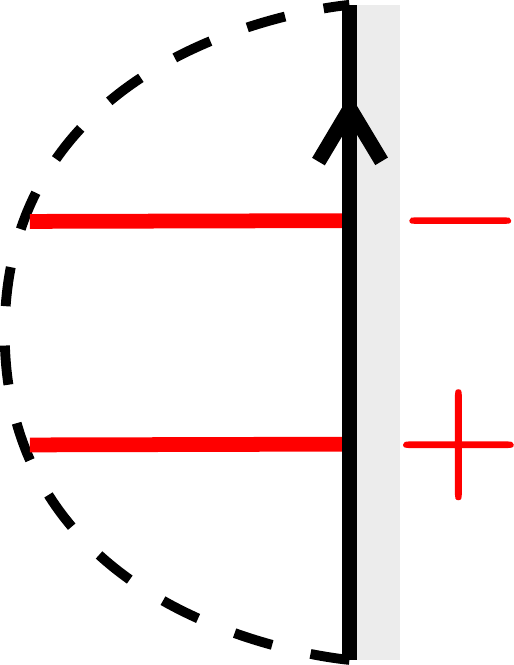}\end{array}=q^2\begin{array}{c}\includegraphics[scale=0.25]{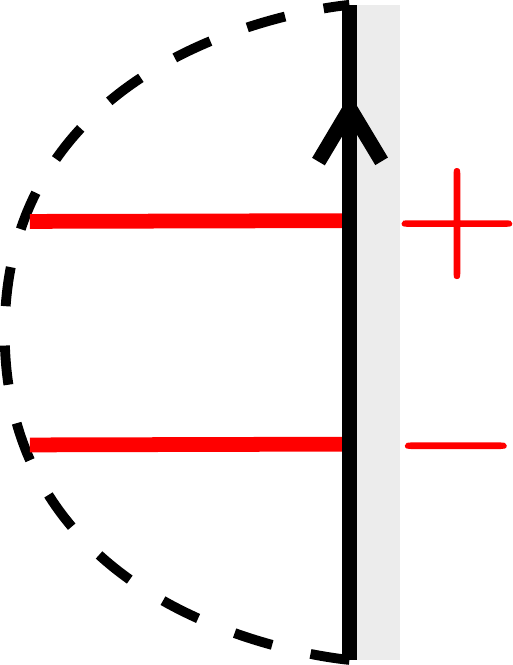}\end{array}+q^{-1/2}\begin{array}{c}\includegraphics[scale=0.25]{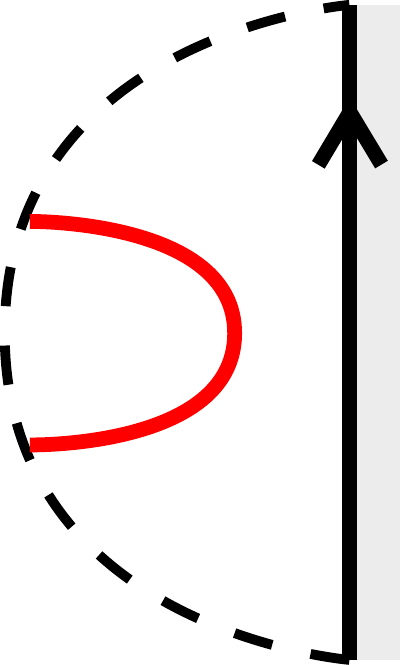}\end{array},\\
&({\rm F})\ \begin{array}{c}\includegraphics[scale=0.25]{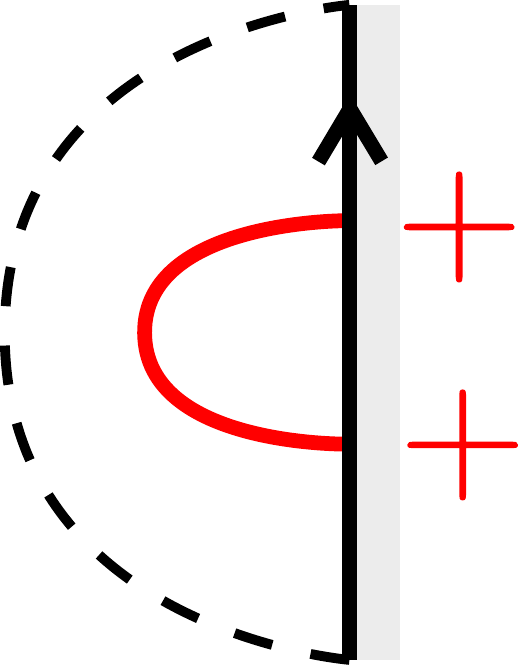}\end{array}=\begin{array}{c}\includegraphics[scale=0.25]{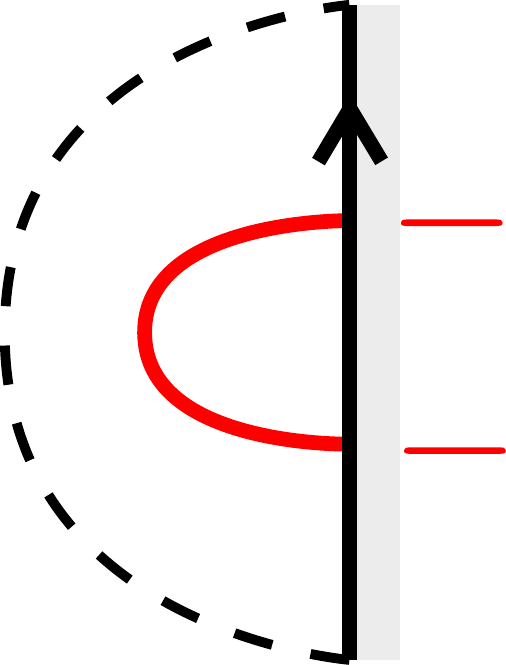}\end{array}=0,\ \ \ \begin{array}{c}\includegraphics[scale=0.25]{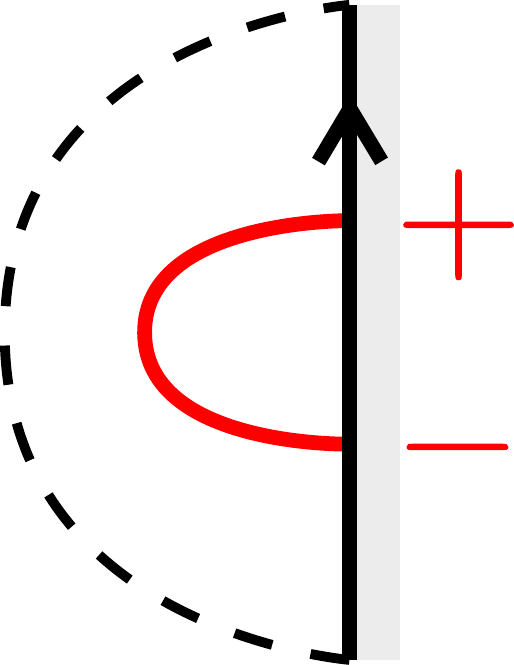}\end{array}=q^{-1/2}\begin{array}{c}\includegraphics[scale=0.25]{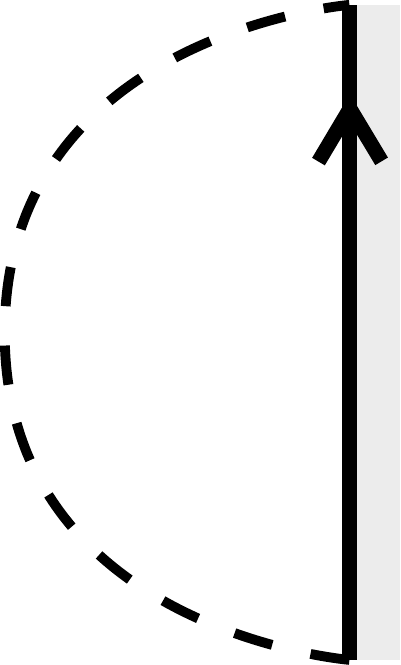}.\end{array}
\end{align*}
Similarly to the (Muller) skein algebras, the stated skein module of $(\Sigma,\BM)$ has an associative multiplication given by the vertical stacking. 
We call the resulting $\CR$-algebra the {\em stated skein algebra} of $(\Sigma, \BM)$, and denote by $\sSsq(\BB)$. 
\end{dfn}

\paragraph{\textbf{L\^e's basis of stated skein algebras}}
Let $\fo$ be an orientation of $\partial^\ast \Sigma$, which may or may not be compatible with the orientation induced from that of $\Sigma$. 
If $\fo$ is the orientation induced from $\Sigma$, then 
$\fo$ is called the {\em positive order}. 

A $\BB$-tangle diagram is {\em simple} if it has neither double points in $\Int\Sigma^\ast$ nor connected components bounding an embedded disk or homotopic to a part of a boundary interval relative to $\pSigma$. 
Given an orientation $\fo$, a simple $\BB$-tangle diagram $\gamma$ is {\em $\fo$-ordered} if the partial order of $\gamma$ with respect to the height is increasing when we go along the direction of $\fo$ for each boundary interval. 

Note that every $\BB$-tangle can be presented by an $\fo$-ordered $\BB$-tangle diagram after a deformation by isotopy. 
In particular, if $\fo$ is the positive order, we call an $\fo$-ordered $\BB$-tangle diagram a positively ordered $\BB$-tangle diagram.

We assign an order on the set $\{+,-\}$ so that $+$ is greater than $-$. 
For a stated $\BB$-tangle diagram $\gamma$, 
a state $s\colon \partial \gamma\to \{+,-\}$ is {\em increasing} if it satisfies $s(x)\geq s(y)$
for $x,y\in \partial \gamma$ such that the height of $x$ is greater than that of $y$.

Let $\sfB(\Sigma,\BB)$ be the set of isotopy classes of increasingly stated, positively ordered simple $\BB$-tangle diagrams in $(\Sigma,\BM)$. 
Then, L\^e showed the following theorem. 
\begin{thm}[{\cite[Theorem 2.8]{Le_triangular}}]
For a marked surface $(\Sigma,\BM)$, the set $\sfB(\Sigma,\BB)$ forms an $\CR$-basis of the stated skein algebra $\sSsq(\BB)$. 
\end{thm}

\paragraph{{\bf Reduced stated skein algebras}}

Let $\Ibad$ denote the two-sided ideal generated by stated $\BB$-tangle diagrams containing at least one bad arc, 
where a {\em bad arc} is a corner arc (i.e. an arc which cuts off from $\Sigma$ a triangle with one special point) with the states $+$ and $-$ 
in this clockwise order depicted as in Figure \ref{bad_arc}.  
\begin{figure}[ht]\centering\includegraphics[width=80pt]{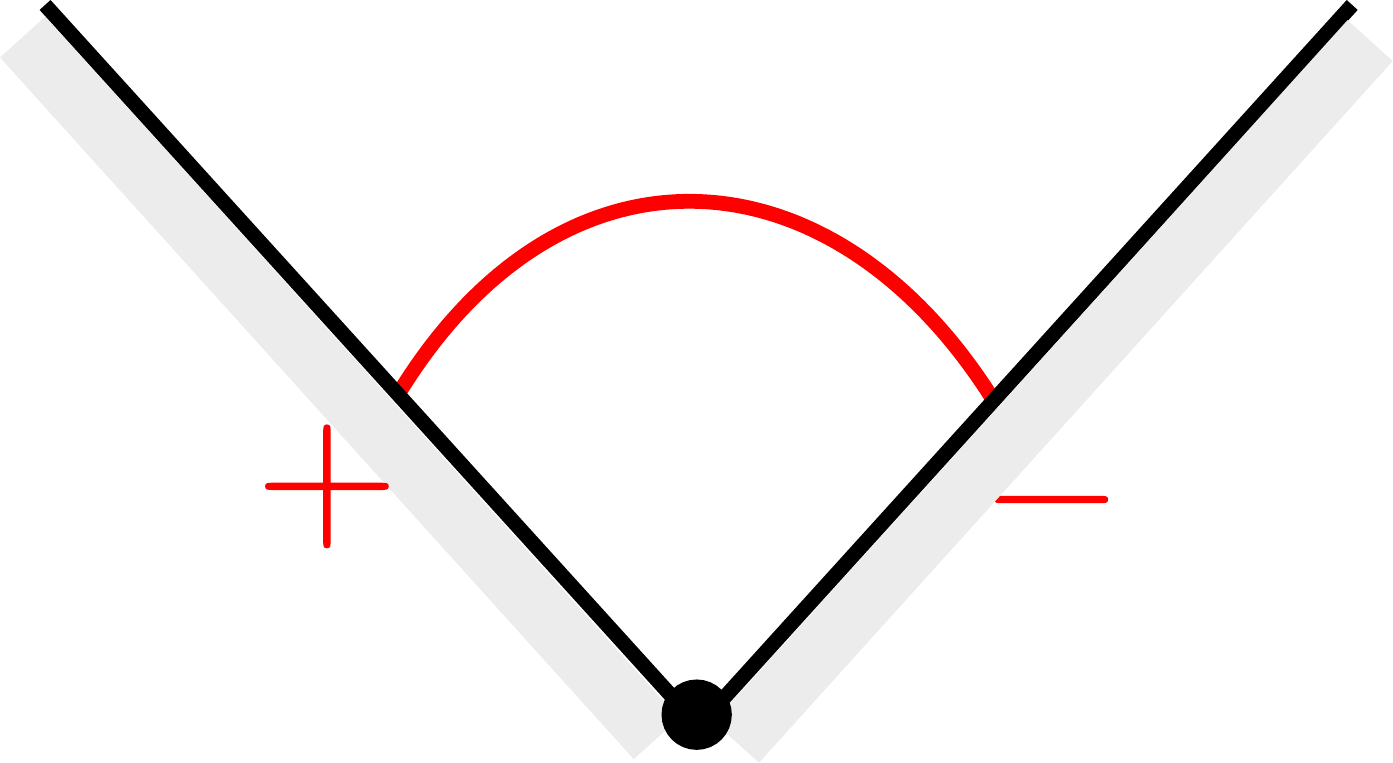}\caption{A bad arc}\label{bad_arc}\end{figure}

\begin{dfn}[\cite{CL}] \label{def:reduced_stated}
The {\em reduced stated skein algebra} of $(\Sigma, \BM)$ is defined to be the quotient algebra $\olSs(\BB):=\sSsq(\BB)/\Ibad$.
\end{dfn}

\begin{figure}[ht]\centering\includegraphics[width=100pt]{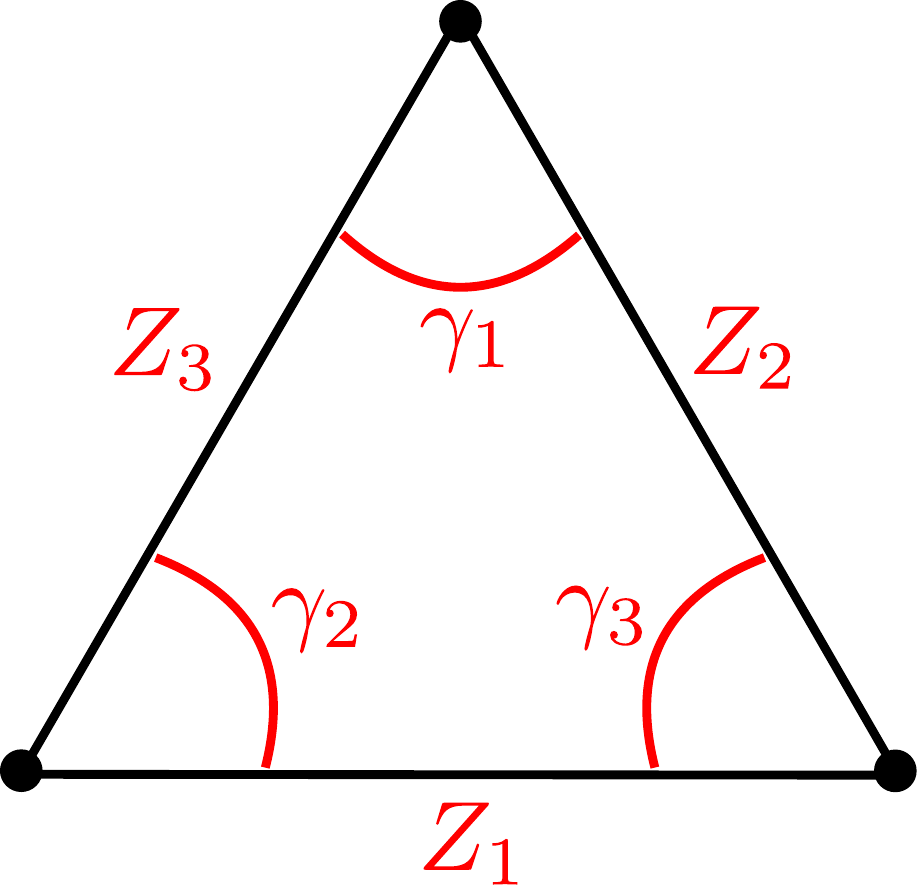}\caption{The generators of the Chekhov--Fock algebras}\label{CFalg_triangle}\end{figure}

\begin{ex}[Triangle case]\label{ex:triangle}
Let $T$ be a triangle, which is a disk with 3 marked points. The stated skein algebra $\mathscr{S}_T^q(\BB)$ is generated by
\begin{align*}
    \{\gamma_i(\ve,\ve') \mid i=1,2,3,~ \ve,\ve' \in \{+,-\}\},
\end{align*}
where $\gamma_i$ are the $\BB$-tangle diagrams shown in Figure \ref{CFalg_triangle}, and $\gamma_i(\varepsilon, \varepsilon')$ denotes $\gamma_i$ with states $\varepsilon$ and $\varepsilon'$ in this counter-clockwise order \cite[Theorem 4.6]{Le_triangular}.

The Chekhov--Fock square-root algebra $\CZ^q(T):=\CZ_\tri^q(T)$ for its unique triangulation $\tri$ is given by
$$\CZ^q(T)=\CR\langle Z_1^{\pm 1}, Z_2^{\pm1}, Z_3^{\pm1} \rangle/ 
(Z_1 Z_2=q^{-1}Z_2 Z_1, Z_2 Z_3=q^{-1}Z_3 Z_2, Z_3 Z_1=q^{-1}Z_1 Z_3). $$
The 
cyclic group $\BZ/3\BZ=\langle \tau \,|\,\tau^3=1\rangle$ acts on each of $\sS^q_T(\BB)$ and $\sSk{T}$
as algebra automorphsims induced by $2\pi/3$-rotations. 
It acts on $\CZ^q(T)$ by $\tau(Z_i)=Z_{i+1}$ with $Z_4:=Z_1$. 
Then, Costantino--L\^e \cite{CL} showed that 
there is a $\tau$-equivariant $\CR$-algebra homomorphism 
\begin{alignat}{3}\label{isom_triangle}
&\widetilde{\mathrm{Tr}}_T\colon \sS^q_{T}(\BB)\to \CZ^q(T)_{\mathrm{bl}}, \quad&
&\gamma_i(+,+)\mapsto [Z_{i+1}^{-1}Z_{i+2}^{-1}],\quad&
&\gamma_i(+,-)\mapsto [Z_{i+1} Z_{i+2}^{-1}],\\
& &
&\gamma_i(-,+)\mapsto 0,&
&\gamma_i(-,-)\mapsto [Z_{i+1}Z_{i+2}], \nonumber
\end{alignat}
where $[\cdot]$ denotes the Weyl normalization. It induces an $\cR$-algebra isomorphism
\begin{align}\label{eq:triangle_isom}
    {\mathrm{Tr}}_T: \sSk{T} \xrightarrow{\sim} \cZ^q(T)_\mathrm{bl}. 
\end{align}
See \cite[Section 7.7]{CL} for details.
\end{ex}

\begin{rem}
The assignment \eqref{isom_triangle} is the Bonahon--Wong's original convention, except for that we use the inverses of $Z$-variables. In the assignment given in \cite[Theorem 7.11]{CL}, the variables $\alpha,\beta,\gamma$ are cluster $K_2$-variables rather than the Chekhov--Fock square-root variables. These two kinds of variables are related by a monomial relation, as we will see in \cref{lem:triangle_compatible}.
\end{rem}

For the reduced stated skein algebra $\sSk{\Sigma}$ of $(\Sigma, \BM)$, 
consider the subset $\overline{\sfB}(\Sigma,\BB)\subset {\sfB}(\Sigma,\BB)$ consisting of stated $\BB$-tangle diagrams without bad arcs. 
\begin{thm}[Costantino--L\^e, {\cite[Theorem 7.1]{CL}}]
The set $\overline{\sfB}(\Sigma,\BB)$ is an $\cR$-basis of $\sSk{\Sigma}$. 
\end{thm}

For later comparison with skein liftings of $\A$-laminations, we slightly modify the basis $\overline{\sfB}(\Sigma,\BB)$ as follows. 

\begin{dfn}\label{def:admissible} 
A stated $\BB$-tangle diagram is {\em admissible} if it consists of 
\begin{enumerate}
    \item corner arcs with the same states, 
    \item non-corner arcs with states $-$
\end{enumerate}
without crossings. 
\end{dfn}

\begin{prop}\label{prop:admissible_span}
In the reduced stated skein algebra $\sSk{\Sigma}$, 
every stated $\BB$-tangle diagram can be presented as an $\cR$-linear sum of admissible stated $\BB$-tangle diagrams. 
\end{prop}
\begin{proof}
To show the claim, it suffices to show that each element of $\overline{\sfB}(\Sigma, \bB)$ can be written as a linear sum of admissible stated $\bB$-tangle diagrams in $\sSk{\Sigma}$. 
It suffices to consider the case that $\partial^\ast\Sigma$ is endowed with the orientation such that $\overline{\sfB}(\Sigma, \bB)$ contains no corner arcs with states $-$ and $+$ in this clockwise order.

We fix $\gamma\in \ol{\sfB}(\Sigma, \bB)$. 
Note that $\gamma$ has no crossings, no bad arcs, and no corner arcs with states $-$ and $+$ in this clockwise order. 
If there is a non-corner arc such that one of its states is $+$, 
we apply the ``sticking trick'' based on the relation (E) as follows:
\begin{align}\label{eq:sticking}
\begin{array}{c}\includegraphics[scale=0.18]{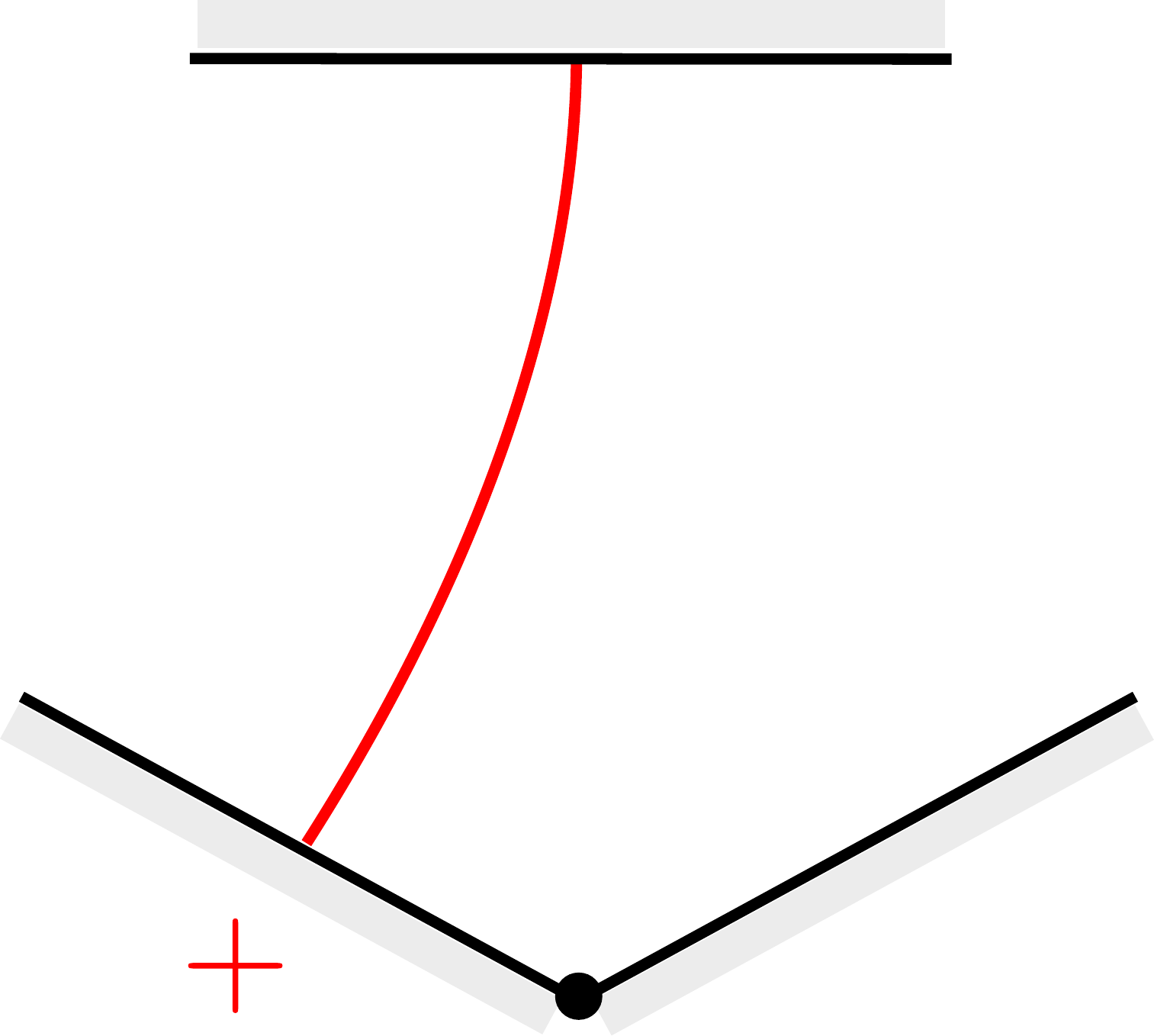}\end{array}=q^{1/2}\begin{array}{c}\includegraphics[scale=0.18]{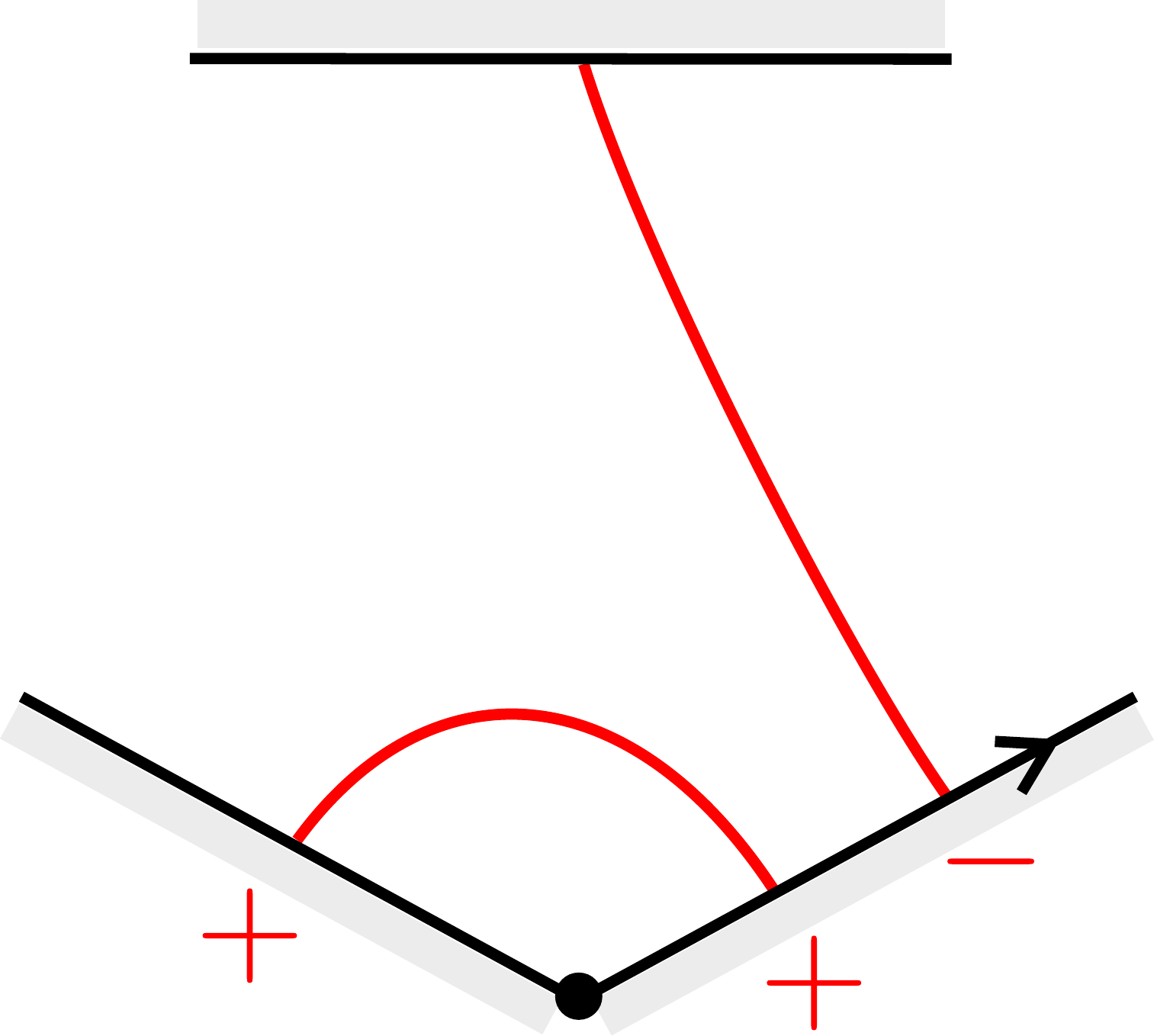}\end{array}
-q^{5/2}
\begin{array}{c}\includegraphics[scale=0.18]{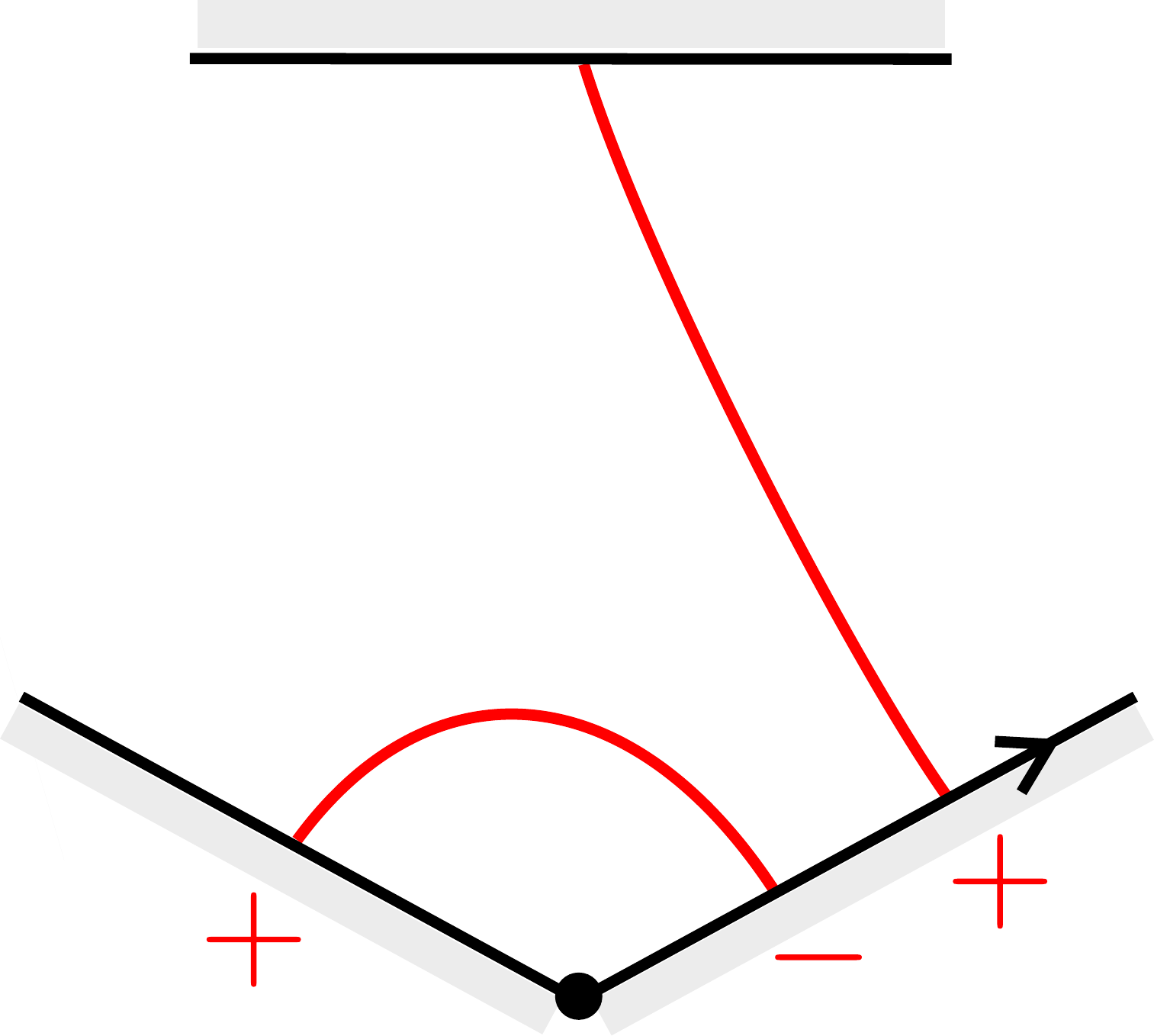}\end{array}
=q^{1/2}\begin{array}{c}\includegraphics[scale=0.18]{lamination_image_06.pdf}\end{array},
\end{align}
where we eliminated the bad arc. 
Note that even if the three boundary intervals appearing in the picture are not distinct, 
the computation still works. 
In this way, one can remove all non-admissible arcs, i.e. non-corner arcs with one of the states is $+$, without affecting the others and additional crossings. 
\end{proof}

\begin{rem}
The linear independence will be discussed in \cref{cor:basis}. 
\end{rem}

\subsection{Splitting map and quantum trace map}

Let $\al$ be an ideal arc on $\Sigma$. 
By splitting $\Sigma$ along $\al$, we obtain a new surface $\Sigma'$ (possibly disconnected), where $\al$ is splitted into two boundary intervals $\al',\al''$. There is a projection $\Sigma'\to \Sigma$ maps the edges $\al'$ and $\al''$ to $\al$ and it is identity except for $\al'\cup \al''$. 
Given a $\BB$-tangle $\gamma$ in $\Sigma \times (-1,1)$, let $\partial_\al\gamma$ be the set of all crossings of $\gamma$ and $\al\times (-1,1)$. 
Then we say that $\gamma$ is {\em vertically transverse} to $\al$ if 
\begin{enumerate}
\item $\gamma$ intersects transversely with $\al\times (-1,1)$ and 
\item the points of $\partial_\al\gamma$ have distinct heights and vertical framing. 
\end{enumerate}

Suppose that $\gamma$ is vertically transverse to $\al$. 
By splitting $\gamma$ along $\al\times (-1,1)$ in $\Sigma\times (-1,1)$, 
we have a new stated $\BB$-tangle $(\tilde{\gamma},\varepsilon)$ in $\Sigma'\times (-1,1)$ equipped with a state $\varepsilon \colon \partial_\al\gamma\to\{+,-\}$ at the corresponding endpoints on $\al'\cup \al''$ \cite[Section 3.1]{LY21}. See \cref{fig:splitting}.

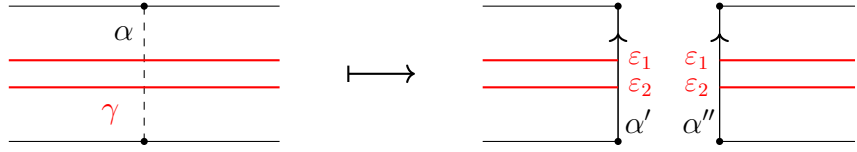
\begin{figure}[ht]
    \centering
\begin{tikzpicture}[scale=0.9]
\begin{scope}
\draw (-2,0) -- (2,0);
\draw (-2,2) -- (2,2);
\draw[dashed] (0,0) -- (0,2);
\draw[red,thick] (-2,0.8) -- (2,0.8);
\draw[red,thick] (-2,1.2) -- (2,1.2);
\node[red] at (-0.5,0.4) {$\gamma$};
\node at (-0.3,1.6){$\al$};
\foreach \j in {0,2} \fill(0,\j) circle(1.5pt);
\draw[thick,|->] (3,1) --++(1,0);
\end{scope}
\begin{scope}[xshift=7cm]
\draw (-2,0) -- (0,0) -- (0,2) -- (-2,2);
\draw (3.5,0) -- (1.5,0) -- (1.5,2) -- (3.5,2);
\draw[->-={0.8}{}] (0,0) -- (0,2);
\draw[->-={0.8}{}] (1.5,0) -- (1.5,2);
\draw[red,thick] (-2,1.2) -- (0,1.2) node[right,scale=0.9]{$\varepsilon_1$};
\draw[red,thick] (3.5,1.2) -- (1.5,1.2) node[left,scale=0.9]{$\varepsilon_1$};
\draw[red,thick] (-2,0.8) -- (0,0.8) node[right,scale=0.9]{$\varepsilon_2$};
\draw[red,thick] (3.5,0.8) -- (1.5,0.8) node[left,scale=0.9]{$\varepsilon_2$};
\node at (0.3,0.3) {$\alpha'$};
\node at (1.2,0.3) {$\alpha''$};
\foreach \i in {0,1.5} \foreach \j in {0,2} \fill(\i,\j) circle(1.5pt);
\end{scope}
\end{tikzpicture}
    \caption{The splitting along an ideal arc $\al$.}
    \label{fig:splitting}
\end{figure}

\begin{thm}[Splitting theorem, {\cite[a part of Theorem 3.1]{Le_triangular}}]\label{Thm_splitting}
Let $\al$ be an ideal arc on $(\Sigma,\bM)$. 
Let $\Sigma'$ be the resulting surface obtained from $\Sigma$ by splitting $\Sigma$ along $\al$. 
Then, we have the following; 
\begin{enumerate}
\item there is a unique $\CR$-algebra homomorphism 
\begin{align}\label{eq:splitting}
    \theta_\al\colon \sSsq(\BB)\to \sS_{\Sigma'}^q(\BB)
\end{align}
such that 
$$\displaystyle\theta_\al(\gamma)=\sum_{\varepsilon\colon \partial_\al\gamma\to\{+,-\}}(\tilde{\gamma},\varepsilon)$$
for a stated $\BB$-tangle $\gamma$ vertically transverse to $\al$.

\item if $\al_1$ and $\al_2$ are ideal arcs on $\Sigma$ which do not intersect except for their endpoints, then $\theta_{\al_1}\circ\theta_{\al_2}=\theta_{\al_2}\circ\theta_{\al_1}$.
\end{enumerate}
\end{thm}

\begin{dfn}\label{def:splitting}
      The homomorphism \eqref{eq:splitting} is called the \emph{splitting map} along $\al$. 
\end{dfn}

By cutting $\Sigma$ along the edges of $\tri$, we obtain the disjoint union $\bigsqcup_{T \in t(\tri)} T$ of triangles. 
\begin{dfn}\label{def:q-trace}
The {\em quantum trace map} with respect to an ideal triangulation $\tri$ is defined to be the composite
$$
\widetilde{\mathrm{Tr}}_\tri\colon \sSsq(\BB)\xhookrightarrow{\,\Psi_\tri\ }  \bigotimes_{T \in t(\tri)} \sS^q_{T}(\BB) \xrightarrow{\otimes_T \widetilde{\mathrm{Tr}}_T} \bigotimes_{T \in t(\tri)}\CZ^q(T), 
$$
where $\Psi_\tri:=\theta_{\al_1}\circ\theta_{\al_2}\circ\dots\circ\theta_{\al_{n(\Sigma)}}$ with $e_{\interior}(\tri)=\{\al_1,\al_2,\dots,\al_{n(\Sigma)}\}$.
\end{dfn}
Note that the image of $\widetilde{\mathrm{Tr}}_\tri$ is in $\CZbl\subset\bigotimes_T\CZ^q(T)$, see \cite{BW16}, \cite{Le_quantum_teich}.

The splitting map $\Psi_\tri$ maps the bad arc ideal of $\Sigma$ to the bad arc ideals of triangles. It induces an injective $\cR$-homomorphism
\begin{align*}
    \overline{\Psi}_\tri: \sSk{\Sigma} \to \bigotimes_T \sSk{T}.
\end{align*}
Moreover, as we saw in Example $\ref{isom_triangle}$, each $\widetilde{\mathrm{Tr}}_T$ descends to an $\CR$-algebra isomorpshism ${\trt}\colon \sSk{T}\xrightarrow{\sim} \CZ^q(T)$. Therefore we get an injective $\cR$-homomorphism
\begin{align}
    \mathrm{Tr}_\tri:=\bigotimes_T \mathrm{Tr}_T \circ \overline{\Psi}_\tri: \sSk{\Sigma} \to \CZbl \subset \bigotimes_T \cZ^q(T),
\end{align}
which fits into the commutative diagram
$$\begin{tikzcd}
\sSsq(\BB) 
\ar[r,"\widetilde{\mathrm{Tr}}_\tri"] \arrow[d] &\CZbl \\
\sSk{\Sigma} \ar[ru,hook,"\mathrm{Tr}_\tri"'] & 
\end{tikzcd}\label{quantum_trace} 
$$

\subsection{Relation between Muller skein algebras and stated skein algebras}

For any marked surface $(\Sigma,\bM)$, the Muller skein algebra of $(\Sigma,\BM)$ can be regarded as 
a subalgebra of the stated skein algebra of $(\Sigma,\BM)$, as follows. 
For an $\BMp$-tangle diagram and each special point $p$, 
we slightly move all strands incident to $p$ to be strands incident to the right boundary interval of $p$ so that the endpoints of the obtained strands are distinct, see Figure \ref{moving_right}. 
Moreover, we impose that the obtained 
$\bB$-tangle diagram is positively ordered and assign states $-$ to all the endpoints. 
The operation implies that an $\BMp$-tangle diagram can be regarded as a $\BB$-tangle diagram. 
Therefore we get an algebra embedding $\sSsq(\bM) \hookrightarrow \sSsq(\bB)$.

\begin{figure}\centering\includegraphics[width=160pt]{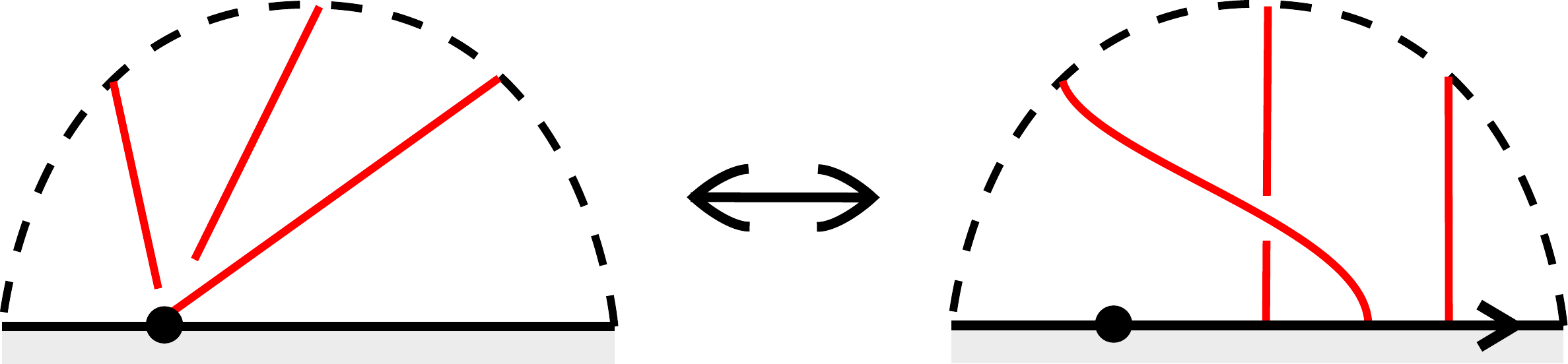}\caption{The operation from left to right is "moving right"}\label{moving_right}\end{figure}

\begin{rem}\label{rem:left_right_moving}
In \cite[Section 4.4]{LY22}, L\^e--Yu consider a similar operation so-called the ``moving left'' with states $+$, so that $\sSsq(\BM)$ is regarded as a subalgebra of $\sSsq(\BB)$ consisting only of $+$ states instead. 
\end{rem}
\begin{thm}[{\cite[Theorem 5.2]{LY22}}]\label{thm:state-clasp}
There is an $\CR$-algebra isomorphism 
\begin{align*}
    \Phi_\Sigma: \sSk{\Sigma} \xrightarrow{\sim} \mathscr{S}^q_\Sigma(\bM)[\partial^{-1}].
\end{align*}
\end{thm}
In accordance with \cite{IY}, we call the isomorphism $\Phi_\Sigma$ the \emph{state-clasp correspondence}. The map $\Phi_\Sigma$ sends each stated end as
\begin{align}\label{eq:state_to_clasp}
\mathord{
\tikz[scale=0.9,baseline=1cm]{
\draw[->-={0.55}{}] (-2,0) -- (0,0) -- (0,2) -- (-2,2);
\draw[red,thick] (-2,1) -- (0,1) node[right,scale=0.8]{$+$};
\foreach \j in {0,2} \fill(0,\j) circle(1.5pt);
\draw[thick,|->] (0.8,1) --node[midway,above]{$\Phi_\Sigma$}++(1,0);
\begin{scope}[xshift=4.3cm]
\draw (-2,0) -- (0,0) -- (0,2) -- (-2,2);
\draw[myblue,thick] (0,0) to[bend left=10] (0,2);
\draw[red,thick,rounded corners=5pt] (-2,1) -- (-0.5,1) -- (0,2);
\foreach \j in {0,2} \fill(0,\j) circle(1.5pt);
\end{scope}
}}\, \qquad 
\mathord{
\tikz[scale=0.9,baseline=1cm]{
\draw[->-={0.55}{}] (-2,0) -- (0,0) -- (0,2) -- (-2,2);
\draw[red,thick] (-2,1) -- (0,1) node[right,scale=0.8]{$-$};
\foreach \j in {0,2} \fill(0,\j) circle(1.5pt);
\draw[thick,|->] (0.8,1) --node[midway,above]{$\Phi_\Sigma$}++(1,0);
\begin{scope}[xshift=4.3cm]
\draw (-2,0) -- (0,0) -- (0,2) -- (-2,2);
\draw[red,thick,rounded corners=5pt] (-2,1) -- (-0.5,1) -- (0,0);
\foreach \j in {0,2} \fill(0,\j) circle(1.5pt);
\end{scope}
}}
\end{align}
Here, each stated end is send to a simultaneous crossing, while the relative heights of different ends are preserved. The blue arc stands for the inverse element of the arc. 
The inverse map $\Phi_\Sigma^{-1}$ is the natural extension of the embedding $\sSsq(\bM) \hookrightarrow \sSsq(\bB)$.

The following compatibility between the cutting map (\cref{def:cutting}) and the splitting map (\cref{def:splitting}) is a key to our main theorem. Assume $(\Sigma,\bM)$ is an unpunctured marked surface. Let $\al$ be an ideal arc on $\Sigma$. We duplicate $\al$ into a pair $\{\al',\al''\}$ of parallel arcs that bound together a biangle $B_\al$. Let $\Sigma':=\Sigma \setminus B_\al$, which can be connected or disconnected. Consider the corresponding splitting map
\begin{align*}
    \widehat{\theta}_\al: \sSk{\Sigma} \to \sSk{\Sigma'} \otimes \sSk{B_\al}.
\end{align*}
Observe that $\sSk{B_\al}$ is a Laurent polynomial ring over $\cR$ generated by $\al_\bot^+$, which is the unique arc transverse to $B_\al$ equipped with the state $+$ on the both endpoints. As an analogue of the state-clasp correspondence, we define an $\cR$-algebra isomorphism
\begin{align*}
    \Phi_{B_\al}: \sSk{B_\al} \xrightarrow{\sim} \Sk{B_\al}[\partial^{-1}], \quad \al_\bot^+ \mapsto \al.
\end{align*}

\begin{thm}[splitting-cutting compatibility]\label{thm:split_cut_compatibility}
In the situation as above, we have the commutative diagram  
\begin{equation*}
    \begin{tikzcd}
    \sSk{\Sigma} \ar[r,"\widehat{\theta}_\al"] \ar[d,"\Phi_\Sigma"'] & \sSk{\Sigma'} \otimes \sSk{B_\al} \ar[d,"{\big[\Phi_{\Sigma'}\otimes\Phi_{B_\al}\big]}"] \\
    \Sk{\Sigma}[\partial^{-1}] \ar[r,"\mathrm{Cut}_\al"'] & \Sk{\Sigma}[\partial^{-1},\al^{-1}].
    \end{tikzcd}
\end{equation*}
Here the right vertical map takes the Weyl normalization of the product of the image of $\Phi_{\Sigma'}\otimes\Phi_{B_c}$. 
\end{thm}

\begin{proof}
It suffices to consider a local diagram $\gamma$ on $\Sigma$ that crosses $\al$ only once. Viewing $\gamma$ as an element of $\sSk{\Sigma}$, we get
\begin{align*}
\begin{tikzpicture}[scale=0.8]
\begin{scope}
\draw (-2,0) -- (2,0);
\draw (-2,2) -- (2,2);
\draw[dashed] (0,0) to[bend right=30] (0,2);
\draw[dashed] (0,0) to[bend left=30] (0,2);
\draw[red,thick] (-2,1) -- (2,1);
\node[red,scale=0.9] at (-1.5,0.7) {$\gamma$};
\node at (-0.55,1.6){$\al'$};
\node at (0.55,1.6){$\al''$};
\foreach \j in {0,2} \fill(0,\j) circle(1.5pt);
\draw[thick,|->] (3,1) --node[midway,above]{$\widehat{\theta}_\al$}++(1,0);
\draw[thick,|->] (3,-2) --node[midway,above=0.3em,anchor=south]{$\Phi_{\Sigma'}\otimes \Phi_{B_\al}$}++(1,0);
\end{scope}
\begin{scope}[xshift=7cm]
\draw (-2,0) -- (0,0) to[bend left=30] (0,2) -- (-2,2);
\draw[red,thick] (-2,1) -- (-0.3,1) node[right=-0.1em,scale=0.7]{$+$};
\foreach \i in {0,0.75,1.5} \foreach \j in {0,2} \fill(\i,\j) circle(1.5pt);
{
\begin{scope}[xshift=0.75cm]
\draw (0,0) to[bend right=30] (0,2);
\draw (0,0) to[bend left=30] (0,2);
\draw[red,thick] (-0.3,1)node[left=-0.1em,scale=0.7]{$+$} -- (0.3,1) node[right=-0.1em,scale=0.7]{$+$};
\end{scope}
}
{
\begin{scope}[xshift=1.5cm]
\draw (2,0) -- (0,0) to[bend right=30] (0,2) -- (2,2);
\draw[red,thick] (2,1) -- (0.3,1) node[left=-0.1em,scale=0.7]{$+$};
\node at (2.75,1) {$+$};
\end{scope}
}
\end{scope}
\begin{scope}[xshift=14cm]
\draw (-2,0) -- (0,0) to[bend left=30] (0,2) -- (-2,2);
\draw[red,thick] (-2,1) -- (-0.3,1) node[right=-0.1em,scale=0.7]{$-$};
\foreach \i in {0,0.75,1.5} \foreach \j in {0,2} \fill(\i,\j) circle(1.5pt);
{
\begin{scope}[xshift=0.75cm]
\draw (0,0) to[bend right=30] (0,2);
\draw (0,0) to[bend left=30] (0,2);
\draw[red,thick] (-0.3,1)node[left=-0.1em,scale=0.7]{$-$} -- (0.3,1) node[right=-0.1em,scale=0.7]{$-$};
\end{scope}
}
{
\begin{scope}[xshift=1.5cm]
\draw (2,0) -- (0,0) to[bend right=30] (0,2) -- (2,2);
\draw[red,thick] (2,1) -- (0.3,1) node[left=-0.1em,scale=0.7]{$-$};
\end{scope}
}
\end{scope}
\begin{scope}[xshift=7cm,yshift=-3cm]
\draw (-2,0) -- (0,0) to[bend left=30] (0,2) -- (-2,2);
\draw[myblue,thick] (0,0) to[bend left=40] (0,2);
\draw[red,thick,rounded corners=5pt] (-2,1) -- (-0.8,1) -- (0,2);
\foreach \i in {0,0.75,1.5} \foreach \j in {0,2} \fill(\i,\j) circle(1.5pt);
{
\begin{scope}[xshift=0.75cm]
\draw (0,0) to[bend right=30] (0,2);
\draw (0,0) to[bend left=30] (0,2);
\draw[red,thick] (0,2)-- (0,0);
\end{scope}
}
{
\begin{scope}[xshift=1.5cm]
\draw (2,0) -- (0,0) to[bend right=30] (0,2) -- (2,2);
\draw[myblue,thick] (0,0) to[bend right=40] (0,2);
\draw[red,thick,rounded corners=5pt] (2,1) -- (0.8,1) -- (0,0);
\node at (2.75,1) {$+$};
\end{scope}
}
\end{scope}
\begin{scope}[xshift=14cm,yshift=-3cm]
\draw (-2,0) -- (0,0) to[bend left=30] (0,2) -- (-2,2);
\draw[red,thick,rounded corners=5pt] (-2,1) -- (-0.8,1) -- (0,0);
\foreach \i in {0,0.75,1.5} \foreach \j in {0,2} \fill(\i,\j) circle(1.5pt);
{
\begin{scope}[xshift=0.75cm]
\draw (0,0) to[bend right=30] (0,2);
\draw (0,0) to[bend left=30] (0,2);
\draw[myblue,thick] (0,2)-- (0,0);
\end{scope}
}
{
\begin{scope}[xshift=1.5cm]
\draw (2,0) -- (0,0) to[bend right=30] (0,2) -- (2,2);
\draw[red,thick,rounded corners=5pt] (2,1) -- (0.8,1) -- (0,2);
\end{scope}
}
\end{scope}
\end{tikzpicture}.
\end{align*}
On the other hand, we compute $\Phi_\Sigma(\gamma)$, whose local picture around $\al$ looks the same, as
\begin{align*}
\begin{tikzpicture}[scale=0.8]
\begin{scope}
\draw (-2,0) -- (2,0);
\draw (-2,2) -- (2,2);
\draw[dashed] (0,0) -- (0,2);
\draw[red,thick] (-2,1) -- (2,1);
\node[red,scale=0.9] at (-1.5,0.6) {$\Phi_\Sigma(\gamma)$};
\node at (-0.3,1.6){$\al$};
\foreach \j in {0,2} \fill(0,\j) circle(1.5pt);
\draw[thick,|->] (3,1) --node[midway,above]{$\mathrm{Cut}_\al$}++(1,0);
\end{scope}
\begin{scope}[xshift=7.75cm]
\draw (-2,0) -- (2,0);
\draw (-2,2) -- (2,2);
\draw[dashed] (0,0) -- (0,2);
\draw[myblue,thick] (0,0) -- (0,2);
\draw[red,thick,rounded corners=5pt,shorten >=5pt] (-2,1) -- (-0.8,1) -- (0,2);
\draw[red,thick,rounded corners=5pt,shorten >=5pt] (2,1) -- (0.8,1) -- (0,0);
\foreach \j in {0,2} \fill(0,\j) circle(1.5pt);
\node[anchor=east] at (-2.3,1) {$q^{-1}$};
\node at (3,1) {$+$};
\end{scope}
\begin{scope}[xshift=14.5cm]
\draw (-2,0) -- (2,0);
\draw (-2,2) -- (2,2);
\draw[dashed] (0,0) -- (0,2);
\draw[myblue,thick] (0,0) -- (0,2);
\draw[red,thick,rounded corners=5pt,shorten >=5pt] (-2,1) -- (-0.8,1) -- (0,0);
\draw[red,thick,rounded corners=5pt,shorten >=5pt] (2,1) -- (0.8,1) -- (0,2);
\foreach \j in {0,2} \fill(0,\j) circle(1.5pt);
\node[anchor=east] at (-2.3,1) {$q$};
\end{scope}
.\end{tikzpicture}
\end{align*}
Observe that both their Weyl normalizations give 
\begin{align*}
\begin{tikzpicture}[scale=0.8]
\begin{scope}[xshift=7.75cm]
\draw (-2,0) -- (2,0);
\draw (-2,2) -- (2,2);
\draw[dashed] (0,0) -- (0,2);
\draw[myblue,thick] (0,0) -- (0,2);
\draw[red,thick,rounded corners=5pt] (-2,1) -- (-0.8,1) -- (0,2);
\draw[red,thick,rounded corners=5pt] (2,1) -- (0.8,1) -- (0,0);
\foreach \j in {0,2} \fill(0,\j) circle(1.5pt);
\node at (3,1) {$+$};
\end{scope}
\begin{scope}[xshift=14cm]
\draw (-2,0) -- (2,0);
\draw (-2,2) -- (2,2);
\draw[dashed] (0,0) -- (0,2);
\draw[myblue,thick] (0,0) -- (0,2);
\draw[red,thick,rounded corners=5pt] (-2,1) -- (-0.8,1) -- (0,0);
\draw[red,thick,rounded corners=5pt] (2,1) -- (0.8,1) -- (0,2);
\foreach \j in {0,2} \fill(0,\j) circle(1.5pt);
\end{scope}
\end{tikzpicture}
.\end{align*}
Thus the assertion is proved.
\end{proof}

\subsection{Congruent subalgebras of stated skein algebras}

A $\BB$-tangle $\al$ is {\em congruent} with respect to an ideal triangulation $\tri$ if the geometric intersection number of a diagram of $\al$ and each edge of $\tri$ is even, \emph{i.e.}, is congruent to $0$ modulo $2$.  
Note that the congruence condition is invariant under isotopy and the relations (A), (B), (E) and (F). 
Moreover, if $\al$ is congruent with respect to $\tri$ then so is $\al$ with respect to another triangulation $\tri'$. 
Hence, the following definition makes sense.

\begin{dfn}\label{def:congruent}
The subalgebra $\sSsq(\bB)_\congr \subset \sSsq(\bB)$ (resp. $\sSk{\Sigma}_\congr \subset \sSk{\Sigma}$) generated by the stated congruent $\BB$-tangle diagrams 
is called the {\em congruent subalgebra} of $\sSsq(\BB)$ (resp. $\sSk{\Sigma}$). 
\end{dfn}

\begin{rem}
In the higher rank case associated with a semisimple Lie algebra $\mathfrak{g}$, we should consider a subalgebra of the stated $\mathfrak{g}$-skein algebra similarly determined by the congruence condition modulo $d$, where $d$ is the determinant of the Cartan matrix of $\mathfrak{g}$. In other words, $d$ is the index of the root lattice in the weight lattice of $\mathfrak{g}$. 
\end{rem}

The congruent subalgebra is related to the cluster Poisson tori:
\begin{prop}\label{prop:congruent_image}
For any ideal triangulation $\tri$, we have $\mathrm{Tr}_\tri(\sSk{\Sigma}_\congr) \subset \X^v_\tri$.
\end{prop}

\begin{proof}
By the congruence condition, each Chekhov--Fock square-root variable $Z_\alpha^\tri$ appears in an even number of times. Recall the assignment \eqref{isom_triangle}. Hence the image $\mathrm{Tr}_\tri(\sSk{\Sigma}_\congr)$ lies in $\X^v_\tri \subset \CZbl$. 
\end{proof}

Let $\sfB(\Sigma,\BB)_\congr\subset \sfB(\Sigma,\BB)$ (resp. $\overline{\sfB}(\Sigma,\BB)_\congr\subset \overline{\sfB}(\Sigma,\BB)$) be the subset consisting of stated congruent $\BB$-tangle diagrams. 
\begin{prop}
The set $\sfB(\Sigma,\BB)_\congr$ (resp. $\overline{\sfB}(\Sigma,\BB)_\congr$) is an $\cR$-basis of $\sSsq(\BB)_\congr$ (resp. $\sSk{\Sigma}_\congr$). 
\end{prop}
\begin{proof}
The proof follows from an observation that the congruence condition of a congruent stated $\BB$-tangle diagram is invariant under the relations {\rm (A)}, {\rm (B)}, {\rm (E)} and {\rm (F)} and isotopy. 
\end{proof}

For a triangulation $\tri$ of $(\Sigma, \BM)$ and a stated $\BB$-tangle diagram $\gamma$, 
a triangle in $t(\tri)$ is {\em congruent} with respect to $\gamma$ if the number of corner arcs of $\gamma$ is even around each corner. 

\begin{prop}
For any ideal triangulation $\tri$ of $(\Sigma, \BM)$, 
every element of $\sfB(\Sigma,\BB)_\congr$
can be presented as an $\cR$-linear sum of congruent stated $\BB$-tangle diagrams $\gamma_i$ in $\sSsq(\BB)$ such that each triangle of $\tri$ is congruent with respect to $\gamma_i$. 
\end{prop}
\begin{proof}
By using isotopy, we can suppose that $\gamma\in \sfB(\Sigma,\BB)_\congr$ intersects minimally
with each edge of $\tri$. 
Using the sticking trick as in \eqref{eq:sticking}, 
$\gamma$ can be presented as a linear sum of stated $\BB$-tangle diagrams $\gamma_i\ (i=1,2,\dots,n)$ consisting only of arcs. 
Pick up one such diagram $\gamma_i$, and let us focus on a non-congruent triangle $T \in t(\tri)$ with respect to $\gamma_i$. 
Then, one of the corners of $T$ has an odd number of corner arcs of $\gamma_i$. 
Fix one of such corners $\angle$. 
At the corner $\angle$, we take the outermost arc $a$, as depicted the middle in Figure \ref{even_corner_arcs}. 
Here, the outermost arc means the arc nearest to the other corners among the corner arcs around $\angle$. 
There is an embedded path in $\Sigma$ connecting $\al$ to a boundary interval $\beta \in \BB$ such that 
\begin{enumerate}
    \item the path first crosses the edge of $t$ opposite to $\angle$;  
    \item the path intersects minimally with each of $e(\tri)$. 
\end{enumerate}
See the middle in Figure \ref{even_corner_arcs}. 
We move $\al$ near $\beta$ along the path then use the sticking trick, see the right in Figure \ref{even_corner_arcs}.
Although the non-congruent triangle turns into a congruent one, the other triangles do not change their congruence. 
Moreover, from the above second condition, 
the obtained stated $\BB$-tangle diagrams realize the geometric intersection numbers with each of $e(\tri)$. 
By repeating the same procedure to other non-congruent triangles, we finally have a linear sum of stated $\BB$-tangle diagrams with only congruent triangles. 
\begin{figure}\centering\includegraphics[width=280pt]{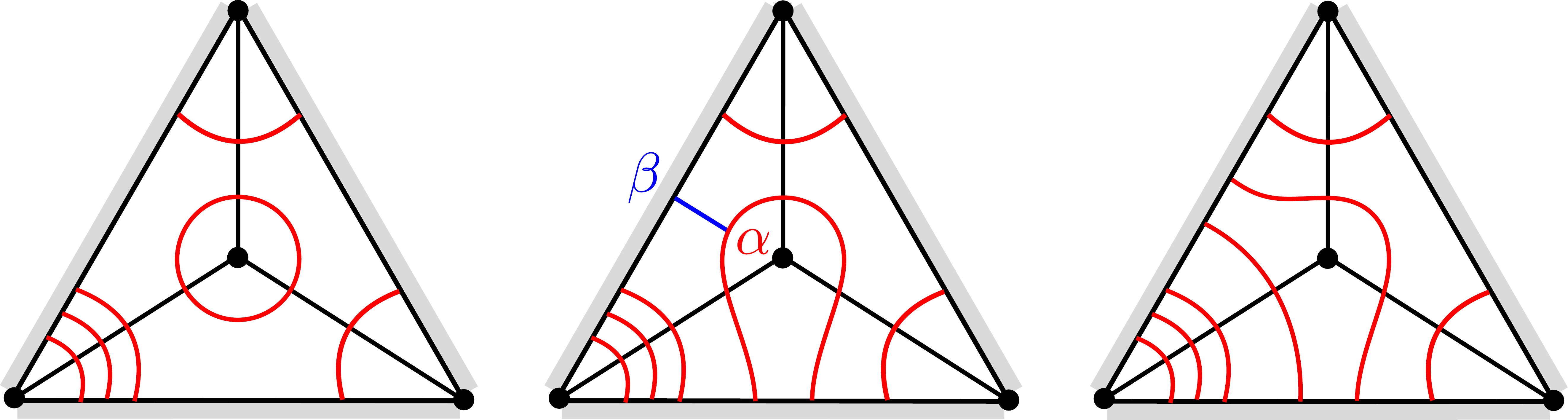}\caption{Procedure to reduce non-congruent triangles}\label{even_corner_arcs}\end{figure}
\end{proof}
\begin{rmk}
The procedure in the above proof might be redundant since it does not matter if there are loops to show the claim. 
\end{rmk}

\section{Quantum cluster algebras associated with a marked surface}\label{sec:cluster}

\subsection{The Ptolemy graph}
Let $\Sigma$ be a marked surface. The \emph{Ptolemy graph} of $\Sigma$ is the graph $\Tri_\Sigma$ such that
\begin{itemize}
    \item the vertices of $\Tri_\Sigma$ are isotopy classes of ideal triangulations $\tri$ of $\Sigma$ without self-folded triangles;
    \item two vertices are connected by an edge if the corresponding ideal triangulations $\tri,\tri'$ are related by a flip along an interior edge. See \cref{fig:flip}.
\end{itemize}

\begin{figure}[ht]
    \centering
\begin{tikzpicture}[>=latex]
\path(0,0) node [fill, circle, inner sep=1.2pt] (x1){};
\path(135:2) node [fill, circle, inner sep=1.2pt] (x2){};
\path(0,2*1.4142) node [fill, circle, inner sep=1.2pt] (x3){};
\path(45:2) node [fill, circle, inner sep=1.2pt] (x4){};
\draw[blue](x1) to node[midway,left=0.3em,black]{$\beta$} (x2) 
to node[midway,left=0.3em,black]{$\alpha$} (x3) 
to node[midway,right=0.3em,black]{$\delta$} (x4) 
to node[midway,right=0.3em,black]{$\gamma$} (x1) 
to node[midway,left=0.3em,black]{$\kappa$} (x3);
{\color{mygreen}
\draw($(x1)!0.5!(x2)$) circle(2pt) coordinate(y2);
\draw($(x2)!0.5!(x3)$) circle(2pt) coordinate(y1);
\draw($(x1)!0.5!(x4)$) circle(2pt) coordinate(y3);
\draw($(x1)!0.5!(x3)$) circle(2pt) coordinate(y0);
\draw($(x4)!0.5!(x3)$) circle(2pt) coordinate(y4);
\qarrow{y1}{y2}
\qarrow{y2}{y0}
\qarrow{y0}{y1}
\qarrow{y3}{y4}
\qarrow{y4}{y0}
\qarrow{y0}{y3}
}
\draw(x1)++(1,0) node{$\tri$};
\draw[->,thick](2,1.4142) --node[midway,above]{$f_\kappa$} (3,1.4142);
\begin{scope}[xshift=5cm]
\path(0,0) node [fill, circle, inner sep=1.2pt] (x1){};
\path(135:2) node [fill, circle, inner sep=1.2pt] (x2){};
\path(0,2*1.4142) node [fill, circle, inner sep=1.2pt] (x3){};
\path(45:2) node [fill, circle, inner sep=1.2pt] (x4){};
\draw[blue](x1) to node[midway,left=0.3em,black]{$\beta$} (x2)
to node[midway,left=0.3em,black]{$\alpha$} (x3) 
to node[midway,right=0.3em,black]{$\delta$} (x4) 
to node[midway,right=0.3em,black]{$\gamma$} (x1);
\draw[blue] (x2) to node[midway,above=0.3em,black]{$\kappa'$} (x4);
{\color{mygreen}
\draw($(x1)!0.5!(x2)$) circle(2pt) coordinate(y2);
\draw($(x2)!0.5!(x3)$) circle(2pt) coordinate(y1);
\draw($(x1)!0.5!(x4)$) circle(2pt) coordinate(y3);
\draw($(x1)!0.5!(x3)$) circle(2pt) coordinate(y0);
\draw($(x4)!0.5!(x3)$) circle(2pt) coordinate(y4);
\qarrow{y4}{y1}
\qarrow{y1}{y0}
\qarrow{y0}{y4}
\qarrow{y2}{y3}
\qarrow{y3}{y0}
\qarrow{y0}{y2}
}
\draw(x1)++(1,0) node{$\tri'$};
\end{scope}
\end{tikzpicture}
\caption{The flip along an interior edge $\kappa \in e_{\interior}(\tri)$.}
    \label{fig:flip}
\end{figure}
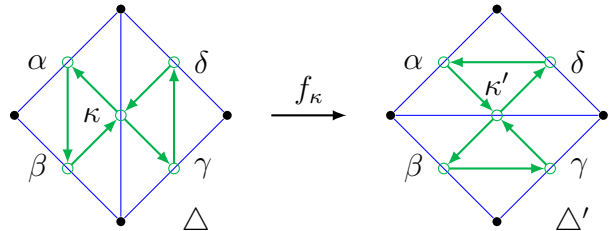

It is known that $\Tri_\Sigma$ is a connected graph (see, for instance, \cite[Proposition 3.8 and Corollary 3.9]{FST}). We denote a vertex of $\Tri_\Sigma$ by $\tri \in \Tri_\Sigma$, and the oriented edge given by the flip along $\kappa \in e(\tri)$ by $\tri \xrightarrow{f_\kappa} \tri'$. Recall the exchange matrix $\ve^\tri$ associated with $\tri$ (\cref{subsub:exchange}).

\begin{lem}\label{lem:mutation_matrix}
For any oriented edge $\tri \xrightarrow{f_\kappa} \tri'$, the exchange matrices $\ve:=\ve^{\tri}$ and $\ve':=\ve^{\tri'}$ are related by the matrix mutation \cite[(12)]{FG09}
\begin{align*}
    \ve'_{\alpha\beta}&=
    \begin{cases}
        -\ve_{\alpha\beta} & \mbox{if $\alpha=\kappa'$ or $\beta=\kappa'$,}\\
        \varepsilon_{\alpha\beta}+\dfrac{\left|\varepsilon_{\alpha\kappa}\right| \varepsilon_{\kappa\beta}+\varepsilon_{\alpha\kappa}\left|\varepsilon_{\kappa\beta}\right|}{2} & \mbox{otherwise}.
    \end{cases}
\end{align*}
\end{lem}

When $\bM_\circ=\emptyset$, the (quantum) cluster charts will be parametrized by the Ptolemy graph.
In the presence of punctures, we are led to consider the \emph{tagged triangulations} (as well as those with self-folded triangles) and their graph, which we will not describe in this paper. See \cite{FST} for details.

\subsection{Quantum cluster Poisson algebra}
Let $\Sigma$ be a marked surface, and $\tri \in \Tri_\Sigma$ (the isotopy class of) its ideal triangulation. Recall from \cref{def:Poisson_torus} the based quantum torus $\X_\tri^v$ over $\bZ_v$. 
The framing $\mathbf{X}^\tri: N^\tri \to \X^v_\tri$ is called a \emph{quantum seed} associated with $\tri$. The mutations of quantum seeds are defined as follows.

\begin{dfn} 
Let $\tri \xrightarrow{f_\kappa} \tri'$ be an oriented edge corresponding to a flip. Then we relate the two quantum seeds associated with $\tri,\tri'$ by the following rule. 
\begin{description}
\item[Lattice mutation] Define a lattice isomorphism $\mu'_\kappa: N^{\tri'} \to N^{\tri}$ by
\begin{align*}
    \mu'_\kappa(\sfe^{\tri'}_\alpha):=\begin{cases}
    -\sfe^{\tri}_\kappa & \mbox{if $\alpha=\kappa'$}, \\
    \sfe^\tri_\alpha + [\ve^\tri_{\alpha\kappa}]_+\sfe^\tri_\kappa & \mbox{otherwise}.
    \end{cases}
\end{align*}

\item[Quantum cluster Poisson transformation]
Let $\Psi_v(x):= \prod_{k=1}^\infty (1+v^{2k-1}x)^{-1}$
be the quantum dilogarithm power series. Then we define an algebra isomorphism $\mu_\kappa^\ast: \Frac \X^v_{\tri'} \to \Frac \X^v_\tri$ by
\begin{align}\label{eq:q-cluster_transf_X}
    \mu_{\kappa}^\ast(\mathbf{X}^{\tri'}(\lambda)) := \mathrm{Ad}_{\Psi_v(X^\tri_\kappa)} (\mathbf{X}^\tri(\mu'_\kappa(\lambda))).
\end{align}
Although $\Psi_v(X^\tri_\kappa)$ is a formal power series, its adjoint action can be computed by using the difference relation
\begin{align}\label{eq:q-difference}
    \Psi_v(v^2 x) = (1+vx)\Psi_v(x)
\end{align}
and the result is rational. See \cite[Lemma 3.4]{FG09} for an explicit formula.
\end{description}
\end{dfn}

\begin{lem}
The map \eqref{eq:q-cluster_transf_X} indeed defines an algebra isomorphism $\mu_\kappa^\ast: \Frac \X^v_{\tri'} \xrightarrow{\sim} \Frac \X^v_\tri$.
\end{lem}

\begin{proof}
In view of \cref{lem:mutation_matrix}, one can verify that $\mu'_\kappa: N^{\tri'} \to N^{\tri}$ preserves the skew-symmetric forms $\omega_X$. Hence it induces an algebra homomorphism $\mu'_\kappa: \X^v_{\tri'} \to \X^v_{\tri}$ such that $\mu'_\kappa(\mathbf{X}^{\tri'}(\lambda)) = \mathbf{X}^{\tri}(\mu'_\kappa(\lambda))$. Then the composite $\mu_\kappa^\ast=\mathrm{Ad}_{\Psi_v(X^\tri_\kappa)} \circ \mu'_\kappa$ is also an algebra homomorphism. By considering the involutive sequence $\tri \xrightarrow{f_\kappa} \tri' \xrightarrow{f_{\kappa'}} \tri$ of flips, one can see that $\mu_{\kappa}^\ast \circ \mu_{\kappa'}^\ast (X^\tri_\alpha) = X^\tri_\alpha$ for all $\alpha \in e(\tri)$. In particular, $\mu_\kappa^\ast$ is an isomorphism. 
\end{proof}

\begin{rem} 
\begin{enumerate}
    \item As noted in \cite[Section 1.2.2]{FG09}, the composite $\mu'_{\kappa'} \circ \mu'_\kappa(\sfe_\alpha^{\tri}) = \sfe_\alpha^{\tri} + \ve^\tri_{\alpha\kappa}\sfe_\kappa^{\tri}$ is not an identity, while it preserves the form on $N^\tri$.

    \item In the original formulation of \cite{FG09}, a lattice $N$ is fixed while its bases $(\sfe_i)_{i \in I}$ transform under mutation. The two formulations are equivalent, and we choose to distinguish the lattices associated with different triangulations in this paper.
\end{enumerate}
\end{rem}

We identify the fields $\Frac \X^v_\tri$ via the quantum cluster Poisson transformations, which is denoted by $\cF^X_\Sigma$. Observe that for each $\tri \in \Tri_\Sigma$, we have a subalgebra $\X^v_\tri \subset \cF^X_\tri$. 
\begin{dfn}
For any 
marked surface, the subalgebra
\begin{align*}
    \cO_v(\X_\Sigma^\circ):=\bigcap_{\tri\in \Tri_\Sigma} \X^v_\tri 
\end{align*}
of $\cF^X_\Sigma$ is called the \emph{quantum cluster Poisson algebra}, a.k.a. \emph{quantum universally Laurent algebra} (without tagged triangulations) associated with $\Sigma$.
\end{dfn}

\begin{rem}
When $\bM_\circ\neq \emptyset$, we need to include the quantum tori associated with tagged triangulations to define the genuine quantum cluster Poisson algebra $\cO_v(\X_\Sigma) \subset \cO_v(\X_\Sigma^\circ)$. In the case $\bM_\circ=\emptyset$, we write $\cO_v(\X_\Sigma) :=\cO_v(\X_\Sigma^\circ)$. In the latter case, it is proved by Shen \cite{Shen22} that $\cO_1(\X_\Sigma)$ coincides with the function algebra of the moduli space $\P_{PGL_2,\Sigma}$ of framed $PGL_2$-local systems with pinnings in the classical limit $v=1$. 
\end{rem}

We will see that the quantum traces can be seen as embeddings $\mathrm{Tr}_\tri:\sSk{\Sigma}_{\mathrm{cong}} \to \X_\tri^v$ under the relation $v=q^{-2}$, which combine to give an embedding $\sSk{\Sigma}_{\mathrm{cong}} \to \cO_v(\X_\Sigma^\circ)$.

\subsection{Quantum upper cluster algebra}\label{subsec:q-CA}
Here it is essential to assume that $\Sigma$ is unpunctured ($\bM_\circ=\emptyset$). Following \cite{Muller}, we are going to construct quantum (upper) cluster algebras inside the skew-field $\cF^A_\Sigma:=\Frac \mathscr{S}_\Sigma^q(\bM)$. 

For each edge $\alpha \in e(\tri)$, let $A_\alpha=[\alpha] \in \Sk{\Sigma}$. Then they satisfy the $q$-commutation relation \eqref{eq:skein_cluster_A_relation}. 
Hence they generate a based quantum torus $\A^q_\tri \subset \cF^A_\Sigma$, equippped with a framing $\mathbf{A}^{\!\tri}:M^\tri \to \A^q_\tri$ such that $\mathbf{A}^{\!\tri}(\sff^\tri_\alpha):=A_\alpha$ for $\alpha \in e(\tri)$. It is isomorphic to the quantum cluster $K_2$-torus introduced abstractly in \cref{def:K2_torus}.

\begin{prop}[Muller]
For any oriented edge $\tri \xrightarrow{f_\kappa} \tri'$, 
the quantum cluster $K_2$-variables satisfy the \emph{quantum exchange relation} \cite{BZ}
\begin{align}\label{eq:q-exchange}
    A_{\kappa'} &= 
    \begin{cases}
   A_\alpha & \mbox{if $\alpha \neq \kappa$}, \\
   \mathbf{A}^{\!\tri}\big(-\sff^\tri_\kappa+\sum_{\beta \in e(\tri)} [\ve^\tri_{\kappa\beta}]_+ \sff^\tri_\beta\big) + \mathbf{A}^{\!\tri}\big(-\sff^\tri_\kappa+\sum_{\beta \in e(\tri)} [-\ve^\tri_{\kappa\beta}]_+ \sff^\tri_\beta\big) & \mbox{if $\alpha=\kappa$}.
    \end{cases} 
\end{align}
\end{prop}
Therefore the following definition agrees with the general definition in \cite{BZ}:

\begin{dfn}
The subalgebra
\begin{align*}
    \cO_q(\A_\Sigma):=\bigcap_{\tri\in \Tri_\Sigma} \A^q_\tri 
\end{align*}
of $\cF^A_\Sigma$ is called the \emph{quantum upper cluster algebra} associated with $\Sigma$. The \emph{quantum cluster algebra} is the $\bZ_q$-subalgebra $\mathscr{A}^q_\Sigma\subset \cF^A_\Sigma$ generated by the cluster $K_2$-variables $A_\alpha$ associated with all the ideal arcs and the inverses of $A_\alpha$ for $\alpha \in \bB$ (the ``frozen'' variables). 
\end{dfn}

\begin{rem}
We have the inclusion $\mathscr{A}^q_\Sigma \subset \cO_q(\A_\Sigma)$, known as the quantum Laurent phenomenon \cite[Corollary 5.2]{BZ}.
\end{rem}
From the construction, we have $\mathscr{A}^q_\Sigma \subset \mathscr{S}_\Sigma(\bM)[\partial^{-1}]$. Moreover, we will see that the cutting maps can be seen as embeddings $\mathrm{Cut}_\tri: \mathscr{S}_\Sigma(\bM)[\partial^{-1}] \to \A^q_\tri$, which combine to give an embedding $\mathscr{S}_\Sigma(\bM)[\partial^{-1}] \to \cO_q(\A_\Sigma)$.

\subsection{Relation between $\cO_v(\X_\Sigma)$ and $\cO_q(\A_\Sigma)$}\label{subsec:ensemble}
Let $\Sigma$ be an unpunctured marked surface. 
For each $\tri \in \Tri_\Sigma$, consider the lattice morphism
\begin{align}\label{eq:ensemble_linear}
    p_\tri^\ast: N^\tri \to M^\tri, \quad \sfe^\tri_\alpha \mapsto \sum_{\beta \in e(\tri)} (\ve^\tri_{\alpha\beta} + m_{\alpha\beta}) \sff^\tri_\beta. 
\end{align}
Here $\ve^\tri=(\ve^\tri_{\alpha\beta})$ is the exchange matrix from \cref{subsub:exchange}, and we set $m_{\alpha\beta}:=-\delta_{\alpha,\beta}$ if $\alpha,\beta \in \bB$ and otherwise $m_{\alpha\beta}:=0$. The correction term $m_{\alpha\beta}$ on the boundary intervals is due to Goncharov--Shen (cf. \cite[Theorem 12.5]{GS19}). Let us write $p_{\alpha\beta}^\tri:=\ve^\tri_{\alpha\beta} + m_{\alpha\beta}$. 

\begin{lem}[{\cite[Theorem 2.16 and Lemma 2.17]{Ish}}]
For any unpunctured marked surface, the matrix $(p_{\alpha\beta}^\tri)_{\alpha,\beta \in e(\tri)}$ is invertible. Moreover, we have 
\begin{align}\label{eq:compatibility_Poisson}
    \sum_{\gamma,\delta \in e(\tri)} p^\tri_{\alpha\gamma}\pi_{\gamma\delta}p^\tri_{\beta\delta}=-4\ve_{\alpha\beta}^\tri
\end{align}
for $\alpha,\beta \in e(\tri)$.
\end{lem}

\begin{rem}
The corresponding relation for the $e_{\interior}(\tri)\times e(\tri)$-submatrix is obtained in \cite[(27)]{Le_quantum_teich}.
\end{rem}

\begin{cor} \label{cor:BZ-compatible}
For $(\alpha,\beta) \in e_{\interior}(\tri)\times e(\tri)$, the \emph{compatibility relation} \cite{BZ}
\begin{align*}
    \sum_{\gamma \in e(\tri)} \ve_{\alpha\gamma}^\tri \pi_{\gamma\beta} = 4 \delta_{\alpha\beta}
\end{align*}
holds.
\end{cor}

\begin{proof}
Observe that $p_{\alpha\beta}^\tri=\ve_{\alpha\beta}^\tri$ if $\alpha \in e_{\interior}(\tri)$ or $\beta \in e_{\interior}(\tri)$.  
Let $(q^\tri_{\nu\beta})$ be the inverse matrix of $(p_{\beta\delta}^\tri)$. Multiplying both sides of \eqref{eq:compatibility_Poisson} by it, we get
\begin{align*}
    \sum_{\gamma \in e(\tri)} \ve_{\alpha\gamma}^\tri \pi_{\gamma\nu} &= -4 \sum_{\beta \in e(\tri)} \ve^\tri_{\alpha\beta}q^\tri_{\nu\beta} 
    = 4 \sum_{\beta \in e(\tri)} q^\tri_{\nu\beta}\ve^\tri_{\beta\alpha} 
    = 4 \sum_{\beta \in e(\tri)} q^\tri_{\nu\beta}p^\tri_{\beta\alpha} = 4\delta_{\nu\alpha}
\end{align*}
for $\alpha \in e_{\interior}(\tri)$. 
Here we used the skew-symmetricity of $\ve^\tri$ in the second equality.
\end{proof}

As the (quantum) monomial version of \eqref{eq:ensemble_linear}, we also consider 
\begin{align}\label{eq:ensemble_monomial}
    p_\tri^\ast: &\X^v_\tri \to \A^q_\tri, \\ 
    &\mathbf{X}^\tri(\lambda) \mapsto \mathbf{A}^{\!\tri}(p_\tri^\ast(\lambda)), \quad  \lambda \in N^\tri.\nonumber
\end{align}
Both the maps \eqref{eq:ensemble_linear}, \eqref{eq:ensemble_monomial} are called the \emph{(extended) ensemble map}. See \cref{fig:ensemble_example}.

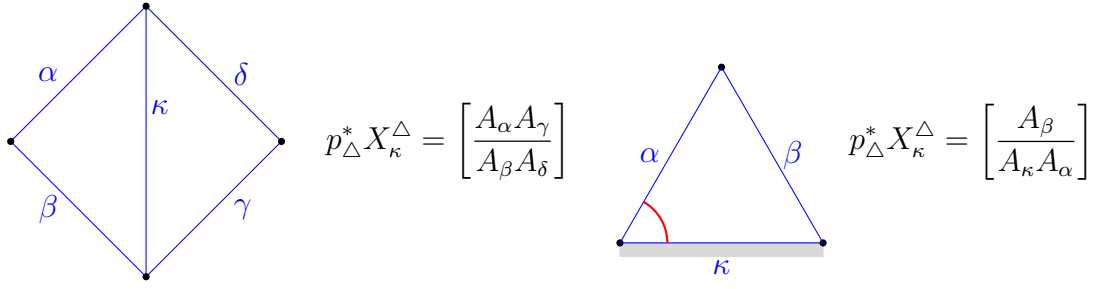
\begin{figure}[ht]
    \centering
\begin{tikzpicture}[scale=0.9]
\draw[blue] (-2,0) --node[midway,left=0.3em]{$\beta$} (0,-2) --node[midway,right=0.3em]{$\gamma$} (2,0) --node[midway,right=0.3em]{$\delta$} (0,2) --node[midway,left=0.3em]{$\alpha$} cycle;
\draw[blue] (0,-2) -- (0,2);
\fill(-2,0) circle(1.5pt);
\fill(2,0) circle(1.5pt);
\fill(0,-2) circle(1.5pt);
\fill(0,2) circle(1.5pt);
\node[blue] at (0.2,0.5) {$\kappa$};
\node[anchor=west] at (2.5,0) {$\displaystyle p_\tri^*X_{\kappa}^\tri = \bigg[\frac{A_{\alpha}A_{\gamma}}{A_{\beta} A_{\delta}}\bigg]$};

\begin{scope}[xshift=7cm,yshift=-1.5cm]
\draw[fill, gray!30]  (0,-0.2) rectangle (3,0);
\fill(0,0) circle(1.5pt) {};
\fill(3,0) circle(1.5pt);
\fill(60:3) circle(1.5pt);
\draw[blue](0,0) --node[midway,below=0.2em]{$\kappa$} (3,0) --node[midway,right]{$\beta$} (60:3) --node[midway,left]{$\alpha$} cycle;
\draw[red,thick] (0.7,0) arc[radius=0.7cm,start angle=0, end angle=60];
\node at (5.2,1.5) {$\displaystyle p_\tri^*X_{\kappa}^\tri = \bigg[\frac{A_{\beta}}{A_{\kappa} A_{\alpha}}$\bigg]};
\end{scope}
\end{tikzpicture}
    \caption{Ensemble map applied to the cluster Poisson variables.}
    \label{fig:ensemble_example}
\end{figure}

\begin{lem}
The map \eqref{eq:ensemble_monomial} is a well-defined algebra homomorphism under the relation $v=q^{-2}$.
\end{lem}

\begin{proof}
During the proof, we omit the scripts $\tri$ from the notation. 
For any $\alpha,\beta \in e(\tri)$, we have
\begin{align*}
    p^\ast(\mathbf{X}_\alpha) \cdot p^\ast(\mathbf{X}_\beta) &= \mathbf{A}(p^\ast\sfe_\alpha)\cdot \mathbf{A}(p^\ast\sfe_\beta) 
    \\
    &= q^{\omega_A(p^\ast\sfe_\alpha,\ p^\ast\sfe_\beta)/2} \mathbf{A}(p^\ast(\sfe_\alpha+\sfe_\beta)) \\
    &=
    q^{\sum p_{\alpha\gamma}\pi_{\gamma\delta}p_{\beta\delta}/2 }p^\ast(\mathbf{X}(\sfe_\alpha+\sfe_\beta)) \\
    &=
    q^{-2\ve_{\alpha\beta}}p^\ast(\mathbf{X}(\sfe_\alpha+\sfe_\beta)) \\
    &= v^{\ve_{\alpha\beta}}p^\ast(\mathbf{X}(\sfe_\alpha+\sfe_\beta)).
\end{align*}
Here we used the relation \eqref{eq:compatibility_Poisson} and $v=q^{-2}$. Thus the assertion is proved.
\end{proof}
In what follows, we always assume the relation $v=q^{-2}$.

\begin{prop}\label{prop:ensemble_mutation}
The ensemble maps \eqref{eq:ensemble_monomial} commute with mutations. In particular, they combine to give an algebra homomorphism $p^\ast_\Sigma: \cO_v(\X_\Sigma) \to \cO_q(\A_\Sigma)$.
\end{prop}
In order to give a proof, we use the following reformulation of the quantum exchange relation \eqref{eq:q-exchange} due to Goncharov--Shen \cite[Section 18]{GS19}. 
Let $\tri \xrightarrow{f_\kappa} \tri'$ be a flip. 
\begin{description}
\item[Lattice mutation] Define a lattice isomorphism $\mu'_\kappa: M^{\tri'} \to M^{\tri}$ by
\begin{align*}
    \mu'_\kappa(\sff^{\tri'}_\alpha):=\begin{cases}
        -\sff_\kappa^\tri + \sum_{\beta \in e(\tri)} [-\ve^\tri_{\kappa\beta}]_+ \sff^\tri_\beta & \mbox{if $\alpha=\kappa'$}, \\
        \sff^\tri_\alpha & \mbox{otherwise}.
        \end{cases}
\end{align*}
It is the dual inverse to the lattice isomorphism $\mu'_\kappa: N^{\tri'} \to N^{\tri}$. 

\item[Quantum cluster $K_2$-transformation]
We define an algebra isomorphism $\mu_\kappa^\ast: \Frac \A^q_{\tri'} \to \Frac \A^q_\tri$ by
    \begin{align}\label{eq:q-cluster_transf_A}
        \mu_{\kappa}^\ast(\mathbf{A}^{\!\tri'}(\nu)) := \mathrm{Ad}_{\Psi_v(p_\tri^\ast (X^\tri_\kappa))} (\mathbf{A}^{\!\tri}(\mu'_\kappa(\nu)).
    \end{align}
Here notice that we use $\Psi_v$ with $v=q^{-2}$.
\end{description}

\begin{lem}[cf. {\cite[Theorem 13.9]{GS19}}]\label{lem:GS-BZ}
The formula \eqref{eq:q-cluster_transf_A} coincides with the quantum exchange relation \eqref{eq:q-exchange}.
\end{lem}

\begin{proof}
For $\kappa \in e_{\interior}(\tri)$ and $\alpha \in e(\tri)$, we get
\begin{align*}
    \omega_A(p_\tri^\ast \sfe^\tri_\kappa, \mu'_\kappa\sff_\alpha^{\tri'}) = \omega_A(p_\tri^\ast \sfe^\tri_\kappa, -\sff_\alpha^{\tri}) = -\sum_{\beta \in e(\tri)} p_{\kappa\beta}^\tri \pi_{\beta\alpha} = -4\delta_{\kappa\alpha}
\end{align*}
from \cref{cor:BZ-compatible}. 
Then the element \eqref{eq:q-cluster_transf_A} is computed as
\begin{align*}
    &\Psi_q(\mathbf{A}(p_\tri^\ast \sfe^{\tri}_\kappa)) \cdot\mathbf{A}(\mu'_\kappa\sff^{\tri'}_\kappa)\cdot \Psi_q(\mathbf{A}(p_\tri^\ast \sfe^{\tri}_\kappa))^{-1} \\
    &= \mathbf{A}(\mu'_\kappa\sff^{\tri'}_\kappa) \cdot\Psi_v(q^{-4} \mathbf{A}(p_\tri^\ast \sfe^{\tri}_\kappa))\cdot\Psi_q(\mathbf{A}(p_\tri^\ast \sfe^{\tri}_\kappa))^{-1} \\
    &= \mathbf{A}(\mu'_\kappa\sff^{\tri'}_\kappa)(1+v \mathbf{A}(p_\tri^\ast \sfe^{\tri}_\kappa)) \\
    &=\mathbf{A}(\mu'_\kappa\sff^{\tri'}_\kappa) + \mathbf{A}(\mu'_\kappa\sff^{\tri'}_\kappa + p_\tri^\ast \sfe^{\tri}_\kappa).
\end{align*}
Since $\mu'_\kappa\sff^{\tri'}_\kappa + p_\tri^\ast \sfe^{\tri}_\kappa = -\sff^\tri_\kappa + \sum_{\beta \in e(\tri)}[\ve^\tri_{\kappa\beta}]_+ \sff^\tri_\beta$, the assertion follows.
\end{proof}

\begin{proof}[Proof of \cref{prop:ensemble_mutation}]
In view of \cref{lem:GS-BZ}, it suffices to compare the transformations \eqref{eq:q-cluster_transf_X} and \eqref{eq:q-cluster_transf_A}. It is easily verified that the lattice morphism $p_\tri^\ast$ commutes with $\mu'_\kappa$. Indeed, it is well-known if $m_{\alpha\beta}=0$ (e.g. \cite{FG09}). Then we have
\begin{align*}
    p_\tri^\ast \mu_\kappa^\ast (X^{\tri'}(\lambda)) &= \mathrm{Ad}_{p_\tri^\ast\Psi_v(X^\tri_\kappa)} (p_\tri^\ast\mathbf{X}^\tri(\mu'_\kappa\lambda)) \\
    &= \mathrm{Ad}_{\Psi_v(p_\tri^\ast X^\tri_\kappa)} (\mathbf{A}^\tri(\mu'_\kappa p_\tri^\ast\lambda)) \\
    &=\mu_\kappa^\ast \mathbf{A}^{\tri'}(p_\tri^\ast\lambda)
\end{align*}
for any $\lambda \in N^\tri$. Thus the assertion is proved. 
\end{proof}

\paragraph{\textbf{Extension to balanced Chekhov--Fock algebra.}} 

\begin{lem}\label{lem:balanced_ensemble}
For $\lambda=\sum_{\alpha \in e(\tri)} \lambda_\alpha \sfe_\alpha^\tri \in N^\tri$ balanced, we have $p_\tri^\ast(\lambda)/2 \in M^\tri$. In other words, 
\begin{align}\label{eq:balanced_ensemble}
    \sum_{\alpha \in e(\tri)}\lambda_\alpha(\ve^\tri_{\alpha\kappa}+m_{\alpha\kappa}) \in 2 \bZ
\end{align}
for all $\kappa \in e(\tri)$. 
\end{lem}

\begin{proof}
Suppose $\kappa \in e_{\interior}(\tri)$. Note that $m_{\alpha\kappa}=0$ for all $\alpha \in e(\tri)$. 
Let us focus on a quadrilateral as in the left of \cref{fig:flip}, and compute the left-hand side of \eqref{eq:balanced_ensemble} as
\begin{align*}
    (-\lambda_\alpha+\lambda_\beta)+(-\lambda_\gamma+\lambda_\delta) \equiv \lambda_\kappa + \lambda_\kappa \equiv 0 \quad (\text{mod 2}).
\end{align*}
Here we used the balanced condition for the triangles $(\alpha,\beta,\kappa)$ and $(\gamma,\delta,\kappa)$.

Consider the case $\kappa \in \bB$. Let $(\alpha,\kappa,\beta)$ be the unique triangle containing $\kappa$, whose edges being in this counter-clockwise order. Then the left-hand side of \eqref{eq:balanced_ensemble} is
\begin{align*}
    \lambda_\alpha -\lambda_\beta - \lambda_\kappa \equiv \lambda_\alpha +\lambda_\beta + \lambda_\kappa \equiv 0 \quad (\text{mod 2}).
\end{align*}
\end{proof}
In particular, the ensemble map \eqref{eq:ensemble_monomial} can be extended to 
\begin{align}\label{eq:ensemble_balanced}
    p_\tri^\ast :\CZbl \to \A_\tri^q, \quad \mathbf{Z}^\tri(\lambda) \mapsto \mathbf{A}^{\!\tri}(-p_\tri^\ast(\lambda)/2)
\end{align}
for $\lambda \in N^\tri$ balanced.

\begin{rem}[Relation to the L\^e--Yu's work {\cite{LY22}}]\label{rem:LY_compatibility}
For an ideal triangulation $\tri$ of $\Sigma$, let us consider the lattices
\begin{align*}
    \widetilde{N}^\tri := N^\tri \oplus \bigoplus_{\alpha \in \bB} \bZ\widehat{\sfe}_\alpha, \quad 
    \widetilde{M}^\tri := M^\tri \oplus \bigoplus_{\alpha \in \bB} \bZ\widehat{\sff}_\alpha.
\end{align*}

Let $\widetilde{\cZ}^q_\tri:=\bT_{(\widetilde{N}^\tri,\widetilde{\omega}_\cZ)}$ and $\widetilde{\A}^q_\tri:=\bT_{(\widetilde{M}^\tri,\widetilde{\omega}_\A)}$ denote the associated based quantum tori over $\bZ_q$, where $\widetilde{\omega}_{\cZ}$ and $\widetilde{\omega}_\A$ are certain extensions of $\omega_\cZ$ and $\omega_\A$, respectively \cite[Section 6.3 and Section 4.2]{LY22}.

L\^e--Yu gives a $\bZ_q$-isomorphism \cite[(71)]{LY22}
\begin{align*}
    \psi_\tri: \widetilde{\A}^q_\tri \xrightarrow{\sim} (\widetilde{\cZ}^q_\tri)_{\mathrm{bl}}, 
\end{align*}
where $(\widetilde{\cZ}^q_\tri)_{\mathrm{bl}} \subset \widetilde{\cZ}^q_\tri$ is the subalgebra determined by a certain balanced condition. For $\alpha \in e(\tri)$, the map is given by
\begin{align}\label{eq:LeYu_ensemble}
    \psi_\tri(A_\alpha):= \bigg[ \prod_{\beta \in e(\tri)} (Z_\beta^\tri)^{K_\beta(\alpha)}\bigg].
\end{align}
Here, $K_a(c)=\pi_{ac}$ denotes the exponent appearing in the formula \cite[(70)]{LY22}. The 
transpose is due to the conventional difference explained in \cref{rem:left_right_moving}. 
The isomorphism $\psi_\tri$ intertwins their two versions of ``extended'' quantum trace maps \cite[Theorem 7.1]{LY22}. 

On the precise relation to the present work, first note that the natural inclusions $N^\tri \to \widetilde{N}^\tri$ and $M^\tri \to \widetilde{M}^\tri$ induce embeddings $(\cZ_\tri^q)_{\mathrm{bl}} \to (\widetilde{\cZ}_\tri^q)_{\mathrm{bl}}$ and $\A_\tri^q \to \widetilde{\A}_\tri^q$. Then we have the following commutative diagram:
\begin{equation*}
    \begin{tikzcd}
    (\cZ_\tri^q)_{\mathrm{bl}} \ar[r] \ar[d,"p_\tri^\ast"'] & (\widetilde{\cZ}_\tri^q)_{\mathrm{bl}} \\
    \A_\tri^q \ar[r] & \widetilde{\A}_\tri^q. \ar[u,"\psi_\tri"']
    \end{tikzcd}
\end{equation*}
Indeed, the exponent $K_\beta(\alpha)$ in \eqref{eq:LeYu_ensemble} is nothing but $-2q_{\alpha\beta}$ in \cite[Theorem 2.16]{Ish}. As stated in its proof, we have $\sum_{\beta \in e(\tri)} q_{\alpha\beta}p_{\beta\gamma}^\tri=\delta_{\alpha,\gamma}$. 
Together with the relation $Z_\beta^\tri=(X_\beta^\tri)^{-1/2}$, the asserted commutativity follows. 
\end{rem}

\subsection{Tropical cluster varieties and their realization by laminations}
Recall that for each $\tri \in \Tri_\Sigma$, associated are the based lattices $N^\tri=\bigoplus_{\alpha \in e(\tri)} \bZ \sfe_\alpha^\tri$ and $M^\tri=\bigoplus_{\alpha \in e(\tri)} \bZ \sff_\alpha^\tri$. Since these lattices are dual to each other, the basis vector $\sff_\alpha^\tri$ (resp. $\sfe_\alpha^\tri$) determines a linear coordinate function $\sfx_\alpha^\tri:N^\tri \to \bZ$ (resp. $\sfa_\alpha^\tri:M^\tri \to \bZ$). 

For each $\kappa \in e_{\interior}(\tri)$, define the \emph{tropical cluster transformations} by
\begin{align}
    \mu_{\kappa,x}^{\mathsf{T}}: M^\tri \to M^{\tri'}, \quad
    (\mu_{\kappa,x}^{\mathsf{T}})^* \sfx_\alpha^{\tri'}= \begin{cases}
    -\sfx_\kappa^\tri & \mbox{if $\alpha=\kappa$},\\
    \sfx_\alpha^\tri - \ve_{\alpha\kappa}^\tri[ -\sgn(\ve_{\alpha\kappa}^\tri)\sfx_\kappa^\tri]_+ & \mbox{if $\alpha\neq \kappa$}, 
    \end{cases} \label{eq:tropical x-transf}
\end{align}
and 
\begin{align}
    \mu_{\kappa,a}^{\mathsf{T}}: N^\tri \to N^{\tri'}, \quad
    (\mu_{\kappa,a}^{\mathsf{T}})^* \sfa_\alpha^{\tri'}= \begin{cases}
    -\sfa_\kappa^\tri + \max\left\{\sum_{\beta \in e(\tri)} [\ve_{\kappa\beta}^\tri]_+\sfa_\beta^\tri,\sum_{\beta \in e(\tri)} [-\ve_{\kappa\beta}^\tri]_+\sfa_\beta^\tri\right\}  & \mbox{if $\alpha=\kappa$},\\
    \sfa_\alpha^\tri & \mbox{if $\alpha\neq \kappa$}.
    \end{cases} \label{eq:tropical a-transf}
\end{align}

\begin{dfn}
The \emph{tropical cluster varieties} are defined to be
\begin{align*}
    \X_\Sigma(\bZ^{\mathsf{T}}) :=\bigg(\bigsqcup_{\tri \in \Tri_\Sigma} M^\tri\bigg)\bigg/\!\!\sim,\quad \A_\Sigma(\bZ^{\mathsf{T}}) :=\bigg(\bigsqcup_{\tri \in \Tri_\Sigma} N^\tri\bigg)\bigg/\!\!\sim.
\end{align*}
Here two points $\mu \in M^\tri$ and $\mu' \in M^{\tri'}$ are identified if they correspond to each other by a sequence of mutations \eqref{eq:tropical x-transf} and coordinate permutations. 
Similarly for the $\A$-side. 
\end{dfn}
In the notataion, $\bZ^\sfT=(\bZ,\max,+)$ stands for the max-plus (tropical) semifield. See \cite[Section 1.1]{FG09} for a general definition of semifield-valued points of positive spaces. 

Let us briefly recall the topological realization of $\X_\Sigma(\bZ^\sfT)$ and $\A_\Sigma(\bZ^\sfT)$ in terms of integral laminations on $\Sigma$ \cite{FG07,Ish}. To simplify the presentation, $\X_\Sigma(\bZ^\sfT)$ is described only for the unpunctured case. 
We use the following terminology:

\begin{itemize}
    \item By a \emph{simple loop} in $\Sigma$, we mean the image of a proper embedding $\gamma: S^1 \to \Sigma^\ast$. It is said to be \emph{essential} if it is not homotopic to a point; 
    \emph{peripheral} if it is homotopic to a point in $\bM_\circ$. 
    \item By a \emph{simple (transverse) arc} in $\Sigma$, we mean the image of a proper embedding $\gamma: [0,1] \to \Sigma^\ast$ such that $\gamma(0),\gamma(1) \in \partial^\ast \Sigma$. It is said to be \emph{essential} if it is not based-homotopic to a subarc of a boundary interval; \emph{peripheral} if it is based-isotopic to a corner arc that encircles exactly one special point.
    \item An \emph{ideal arc} in $\Sigma$ is the image of a map $\alpha:[0,1] \to \Sigma^\ast$ such that $\alpha(0),\alpha(1) \in \bM$ and $\alpha|_{(0,1)}$ is a proper embedding into $\Sigma^\ast$. 
\end{itemize}

\begin{dfn}[\cite{FG07}]
Let $\Sigma$ be any marked surface. 
An \emph{integral $\A$-lamination} on $\Sigma$ is the isotopy class of a mutually non-isotopic, disjoint collection $\{\gamma_j\}_j$ of essential simple loops/arcs in $\Sigma$ together with integral weights $w_j \in \bZ\setminus \{0\}$ such that $w_j >0$ if $\gamma_j$ is non-peripheral. 
Let $\cL^a(\Sigma,\bZ)$ denote the set of integral $\A$-laminations.
\end{dfn}
An element of $\cL^a(\Sigma,\bZ)$ is denoted by $L=\{(\gamma_j,w_j)\}_j$. 

\begin{dfn}[\cite{Ish}]
Assume $\bM_\circ=\emptyset$ for simplicity. 
An \emph{integral $\P$-lamination} on $\Sigma$ consists of the following data:
\begin{enumerate}
    \item the isotopy class of a mutually non-isotopic, disjoint collection $\{\gamma_j\}_j$ of essential non-peripheral simple loops/arcs in $\Sigma$ together with integral weights $w_j \in \bZ_{>0}$.

    \item the \emph{pinning} $\nu=(\nu_\alpha)_{\alpha \in \bB}: \bB \to \bZ$. 
\end{enumerate}
Let $\cL^p(\Sigma,\bZ)$ denote the set of integral $\P$-laminations.
\end{dfn}
An element of $\cL^p(\Sigma,\bZ)$ is denoted by $(L=\{(\gamma_j,w_j)\}_j,
\nu)$. Forgetting the pinnings $\nu$, we recover the space of integral $\X$-laminations studied by Fock--Goncharov \cite{FG07}.

\smallskip
\paragraph{\textbf{The intersection coordinates on $\cL^a(\Sigma,\bZ)$.}}
Let $\tri$ be an ideal triangulation of $\Sigma$. 
For $L=\{(\gamma_j,w_j)\}_j \in \cL^a(\Sigma,\bZ)$, isotope each curve $\gamma_j$ so that minimally intersecting with $\tri$. Then for $\alpha \in e(\tri)$ we define 
\begin{align*}
    \sfa_\alpha(L):=\frac 1 2 \sum_j w_j\cdot \boldsymbol{i}(\gamma_j,\alpha),  
\end{align*}
where $\boldsymbol{i}(\gamma_j,\alpha) \in \bZ_{\geq 0}$ denotes the geometric intersection number between the curves $\gamma_j$ and $\alpha$. We say that $L$ is \emph{($\tri$-)congruent} if $\sfa_\alpha(L) \in \bZ$ for all $\alpha \in e(\tri)$, and let $\cL^a(\Sigma,\bZ)_{\congr} \subset \cL^a(\Sigma,\bZ)$ denote the subset of congruent laminations. 
Then 
\begin{align*}
    \sfa_\tri:=(\sfa_\alpha)_{\alpha \in e(\tri)}:\cL^a(\Sigma,\bZ)_\congr \to \bZ^{e(\tri)}
\end{align*}
defines a bijection, and the coordinate transformation $\sfa_{\tri'}\circ \sfa_\tri^{-1}$ coincides with the tropical cluster transformation \eqref{eq:tropical a-transf} via the identification $N^\tri=\bZ^{e(\tri)}$ \cite{FG07}. In particular, it follows that the congruence is independent of $\tri$. 
The coordinate systems $\sfa_\tri$ combine to give a canonical piecewise-linear isomorphism 
\begin{align*}
    \cL^a(\Sigma,\bZ)_\congr \cong \A_\Sigma(\bZ^\sfT).
\end{align*}

\smallskip
\paragraph{\textbf{The shear coordinates on $\cL^p(\Sigma,\bZ)$.}}
For later discussion, we use the Langlands dual version of shear coordinates in \cite[Section 4.4]{Ish}. The interior coordinates are the same as in \cite{FG07}. 
Let $\tri$ be an ideal triangulation of $\Sigma$. 
For $(L=\{(\gamma_j,w_j)\}_j,\nu) \in \cL^p(\Sigma,\bZ)$, isotope each curve $\gamma_j$ so that minimally intersecting with $\tri$. 
For each edge $\alpha \in e(\tri)$ and a curve $\gamma_j$ in $L$, let $(\alpha:\gamma_j)^\vee \in \bZ$ be the integer defined as follows (see \cref{f:intersection sign}): 
\begin{itemize}
    \item if $\alpha \in e_{\interior}(\tri)$, then it is the diagonal of a unique quadrilateral $Q_\alpha$ in $\tri$. An intersection between a portion of $\gamma_j$ and $Q_\alpha$ as in the left (resp. right) of \cref{f:intersection sign} contributes as $+1$ (resp. $-1$), and the others $0$. Then $(\alpha:\gamma_j)^\vee$ is the sum of these local contributions. 
    \item if $\alpha \in \bB$, then $(\alpha:\gamma_j)^\vee
    :=+1$ if $\gamma_j$ contains a corner arc around the terminal marked point $m^-_\alpha$ as its portion, and otherwise $0$.  
\end{itemize}

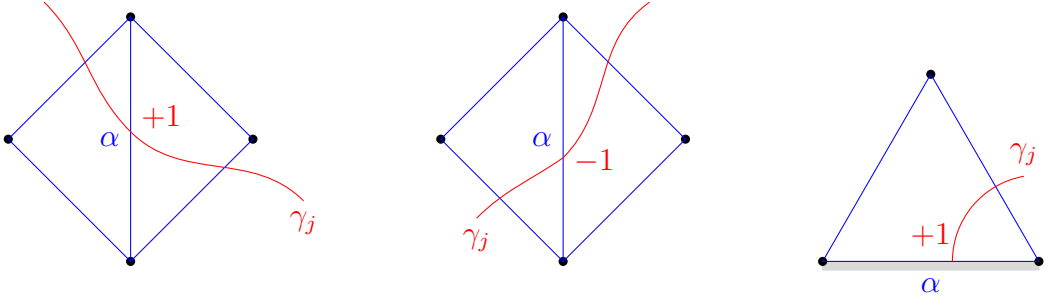
\begin{figure}[ht]
\centering
\begin{tikzpicture}[scale=1.15]
\fill(0,0) circle(1.5pt) coordinate(x1);
\fill(135:2) circle(1.5pt) coordinate(x2){};
\fill(0,2*1.4142) circle(1.5pt) coordinate(x3){};
\fill(45:2) circle(1.5pt) coordinate(x4){};
\draw[blue](x1) to (x2) to (x3) to (x4) to (x1) to node[midway,left]{$\alpha$} node[midway,above right,red]{$+1$} (x3);
\draw [red] (2,0.7) to[out=135,in=-45] (0,1.5) to[out=135,in=-45] (-1,3);
\draw [red] (2,0.7) node[below]{$\gamma_j$}; 

\begin{scope}[xshift=5cm]
\fill(0,0) circle(1.5pt) coordinate(x1){};
\fill(135:2) circle(1.5pt) coordinate(x2){};
\fill(0,2*1.4142) circle(1.5pt) coordinate(x3){};
\fill(45:2) circle(1.5pt) coordinate(x4){};
\draw[blue] (x1) to (x2) to (x3) to (x4) to (x1) to node[midway,left]{$\alpha$} node[midway,below right,red]{$-1$} (x3);
\draw [red] (-1,0.5)  to[out=45,in=215] (0,1.2) to[out=45,in=215] (1,3);
\draw [red] (-1,0.5) node[below]{$\gamma_j$};
\end{scope}

\begin{scope}[xshift=8cm]
\draw[fill, gray!30]  (0,-0.1) rectangle (2.5,0);
\fill(0,0) circle(1.5pt) {};
\fill(2.5,0) circle(1.5pt);
\fill(60:2.5) circle(1.5pt);
\draw[blue](0,0) --node[midway,below=0.2em]{$\alpha$} (2.5,0) -- (60:2.5) --cycle;
\draw[red] (1.5,0) arc[radius=1cm,start angle=180, end angle=100] node[above]{$\gamma_j$};
\node[red] at (1.25,0.3) {$+1$};
\end{scope}
\end{tikzpicture}
\caption{Contributions to $(\alpha:\gamma_j)^\vee$.}
\label{f:intersection sign}
\end{figure}

Define the \emph{shear coordinates} by
\begin{align*}
    \check{\sfx}^\tri_\alpha(L,\nu):=\begin{cases}
       \sum_j w_j (\alpha:\gamma_j)^\vee & \mbox{if $\alpha \in e_{\interior}(\tri)$}, \\
       \nu_\alpha+ \sum_j w_j (\alpha:\gamma_j)^\vee 
       & \mbox{if $\alpha \in \bB$}. 
    \end{cases}
\end{align*}
Then 
\begin{align*}
    \check{\sfx}_\tri:=(\check{\sfx}^\tri_\alpha)_{\alpha \in e(\tri)}:\cL^p(\Sigma,\bZ) \to \bZ^{e(\tri)}
\end{align*}
defines a bijection, and the coordinate transformation $\check{\sfx}_{\tri'}\circ \check{\sfx}_\tri^{-1}$ coincides with the tropical cluster transformation \eqref{eq:tropical x-transf} via the identification $M^\tri=\bZ^{e(\tri)}$ \cite{FG07,Ish}. 
The coordinate systems $\check{\sfx}_\tri$ combine to give a canonical piecewise-linear isomorphism
\begin{align*}
    \cL^p(\Sigma,\bZ) \cong \X_\Sigma(\bZ^\sfT).
\end{align*}

\paragraph{\textbf{The ensemble map.}}
Suppose $\Sigma$ is unpunctured. 
Let us consider the Langlands dual ensemble map \cite[Section 4.4]{Ish}
\begin{align}\label{eq:ensemble_lamination}
    \check{p}_\Sigma^\sfT: \cL^a(\Sigma,\bZ) \to \cL^p(\Sigma,\bZ)
\end{align}
defined by forgetting the peripheral components, and defining the pinning $\nu_\alpha \in \bZ$ to be 
the weight of the peripheral component around the terminal marked point $m^-_\alpha$. Since we can uniquely recover the weights on peripheral components from the data of pinnings, the map \eqref{eq:ensemble_lamination} is bijective. 

\begin{lem}[{\cite[Theorems 4.13 and 3.10]{Ish}}]
For any ideal triangulation $\tri$, we have
\begin{align*}
    (\check{p}_\Sigma^\sfT)^\ast \check{\sfx}^\tri_\kappa = \sum_{\alpha \in e(\tri)} (\ve_{\kappa\alpha}^\tri -m_{\kappa\alpha}) \sfa_\alpha. 
\end{align*}
Moreover, the presentation matrix $(\ve_{\kappa\alpha}^\tri -m_{\kappa\alpha})_{\kappa,\alpha \in e(\tri)}$ is invertible over $\frac 1 2 \bZ$.
\end{lem}
The name ``Langlands dual'' comes from the fact that the presentation matrix of $\check{p}_\Sigma^\sfT$ appearing in the lemma above is minus the transpose of $p^\tri_{\kappa\alpha}$ (cf. \cite[Section 1.2.10]{FG09}).

\section{Quantum duality maps}
In this section, we are going to prove \cref{introthm:compatibility}. Our strategy is to decompose the asserted diagram into

\begin{equation}\label{eq:strategy}
    \begin{tikzcd}
    \A_\Sigma(\bZ^\sfT) \ar[r,"S_\A"] \ar[d,"\check{p}_\Sigma^\sfT"'] & \sSk{\Sigma}_{\congr} \ar[r,"\mathrm{Tr}_\Sigma"] \ar[d,"\Phi_\Sigma"] & \cO_v(\X_\Sigma) \ar[d,"p_\Sigma^\ast"] \\
    \X_\Sigma(\bZ^\sfT) \ar[r,"S_\X"'] & \Sk{\Sigma}[\partial^{-1}] \ar[r,"\mathrm{Cut}_\Sigma"'] & \cO_q(\A_\Sigma),
    \end{tikzcd}
\end{equation}
where the vertical maps are already defined. In \cref{subsec:lifting}, we 
define the \emph{skein lifting maps} $S_\A$, $S_\X$. In \cref{subsec:trace_cut}, we define the \emph{quantum trace map} $\mathrm{Tr}_\Sigma$ by combining $\mathrm{Tr}_\tri$ (\cref{def:q-trace}), and the \emph{cutting map} $\mathrm{Cut}_\Sigma$ by combining $\mathrm{Cut}_\tri$ (\cref{def:cutting}). The commutativity of the two blocks in the diagram \eqref{eq:strategy} are proved in these subsections, respectively. 
Finally, the quantum duality maps are defined to be $\bI_\A:=\mathrm{Tr}_\Sigma\circ S_\A$ and $\bI_\X:=\mathrm{Cut}_\Sigma \circ S_\X$. 

\subsection{Skein lifting maps}\label{subsec:lifting}
The following is verified by a direct computation:
\begin{lem}
We have the following relations in the reduced stated skein algebra for $\ve_1\geq \ve_2$:
\begin{align}
    \mathord{
    \tikz[baseline=2ex]{
    \fill[gray!20] (-1.5,0) -- (1.5,0) -- (1.5,-0.1) -- (-1.5,-0.1) --cycle;
    \draw[thick,->-={0.7}{}] (-1.5,0) -- (1.5,0);
    \foreach \i in {-0.75,0.75} \filldraw(\i,0) circle(1.5pt);
    \draw[red,thick] (0.25,1) -- (-0.25,0);
    \fill[white] (0,0.5) circle(0.15cm);
    \draw[red,thick] (-0.25,1) -- (0.25,0);
    \node[red,scale=0.8] at (0.25,-0.2){$\ve_2$};
    \node[red,scale=0.8] at (-0.25,-0.2){$\ve_1$};
    }
    }\ 
    &= q^{-\ve_1\ve_2}\ 
    \mathord{
    \tikz[baseline=2ex]{
    \fill[gray!20] (-1.5,0) -- (1.5,0) -- (1.5,-0.1) -- (-1.5,-0.1) --cycle;
    \draw[thick,->-={0.7}{}] (-1.5,0) -- (1.5,0);
    \foreach \i in {-0.75,0.75} \filldraw(\i,0) circle(1.5pt);
    \draw[red,thick] (-0.25,1) -- (-0.25,0);
    \draw[red,thick] (0.25,1) -- (0.25,0);
    \node[red,scale=0.8] at (0.25,-0.2){$\ve_1$};
    \node[red,scale=0.8] at (-0.25,-0.2){$\ve_2$};
    }
    }\ ,\label{eq:Weyl_2}\\
    \mathord{
    \tikz[baseline=2ex]{
    \fill[gray!20] (-1.5,0) -- (1.5,0) -- (1.5,-0.1) -- (-1.5,-0.1) --cycle;
    \draw[thick,->-={0.7}{}] (-1.5,0) -- (1.5,0);
    \foreach \i in {-0.75,0.75}\filldraw(\i,0) circle(1.5pt);
    \draw[red,thick] (0,1) -- (0,0);
    \draw[white,line width=0.2cm,shorten >=0.3cm, shorten <=0.3cm] (-1.05,0) arc(180:0:0.7);
    \draw[red,thick] (-1.05,0) arc(180:0:0.7);
    \node[red,scale=0.8] at (0,-0.2){$-$};
    \node[red,scale=0.8] at (-1.05,-0.2){$+$};
    \node[red,scale=0.8] at (-1.05+1.4,-0.2){$+$};
    }
    }\ 
    &= q\ 
    \mathord{
    \tikz[baseline=2ex]{
    \fill[gray!20] (-1.5,0) -- (1.5,0) -- (1.5,-0.1) -- (-1.5,-0.1) --cycle;
    \draw[thick,->-={0.7}{}] (-1.5,0) -- (1.5,0);
    \foreach \i in {-0.75,0.75} \filldraw(\i,0) circle(1.5pt);
    \draw[red,thick] (0,1) -- (0,0);
    \draw[red,thick] (-0.45,0) arc(0:180:0.3);
    \node[red,scale=0.8] at (0,-0.2){$-$};
    \node[red,scale=0.8] at (-0.45,-0.2){$+$};
    \node[red,scale=0.8] at (-1.05,-0.2){$+$};
    }
    }\ .\label{eq:Weyl_1} 
\end{align}
\end{lem}
Note that \eqref{eq:Weyl_2} is nothing but \cite[(20)]{Le_triangular}, while \eqref{eq:Weyl_1} is valid for the case where the end with the state $+$ belongs to a peripheral arc (cf. \cite[Figure 28]{Le_triangular}).

\begin{dfn}\label{dfn_Weyl}
We define the \emph{Weyl normalization} of the two stated arcs appearing in \eqref{eq:Weyl_1} to be $q^{-1/2}$ times the left-hand side, or equivalently, $q^{1/2}$ times the diagram in the right-hand side. Similarly for those appearing in \eqref{eq:Weyl_2}, replacing $q^{\pm 1/2}$ with $q^{\mp \ve_1\ve_2/2}$. 
\end{dfn}

Given a disjoint collection $C=\{\gamma_1,\dots,\gamma_k\}$ of simple loops, simple arcs and ideal arcs, its \emph{skein lift} is the disjoint union $\bigsqcup_{i=1}^k (\gamma_i \times \{t_k\}) \subset \Sigma\times (-1,1)$ of lifts with constant elevations $t_i \in (-1,1)$ and vertical framing. 
We call this element the \emph{constant-elevation lift} of the collection $C$.

\begin{dfn}[skein lifting of integral $\A$-laminations]\label{def:skein_lift_A}
Let $\Sigma$ be an unpunctured marked surface,\footnote{The case with punctures is briefly discussed in \cref{rem:peripheral}.} and $L=\{(\gamma_i,w_i)\} \in \cL^a(\Sigma,\bZ)$ be an integral $\A$-lamination. 
We define the corresponding element $S_\A(L) \in \sSk{\Sigma}$ in the reduced stated skein algebra as follows.
\begin{itemize}
    \item For each weighted 
    loop $(\gamma_i,w_i)$, associate the element 
    \begin{align*}
        T_{w_i}([\gamma_i]) \in \sSk{\Sigma},
    \end{align*}
    where $[\gamma_i]$ is the constant-elevation lift of $\gamma_i$ and $T_{w_i}$ denotes the $w_i$-th Chebyshev polynomial.
    \item For each weighted arc $(\gamma_i,w_i)$, associate the element 
    \begin{align*}
        [\gamma_i^-]^{w_i} \in \sSk{\Sigma},
    \end{align*}
    namely the constant-elevation lift of the curve $\gamma_i$ equipped with the state $-$ on both of its endpoints. Here note that $[\gamma_i^-]^{-1}=[\gamma_i^+]$ if $\gamma$ is peripheral, hence the the assignment makes sense even if $w_i$ is negative. 
\end{itemize}
Then $S_\A(L)$ is defined to be the product of these elements, where 
\begin{itemize}
    \item the elevation of each component is chosen so that their projection diagrams do not intersect with each other;
    \item the Weyl normalization is applied.
\end{itemize}
\end{dfn}
The following is clear from the correspondence $\A_\Sigma(\bZ^\sfT) \cong \cL^a(\Sigma,\bZ)_\congr$:

\begin{lem}
We have $S_\A(\A_\Sigma(\bZ^\sfT)) \subset \sSk{\Sigma}_\congr$.
\end{lem}
In order to construct $S_\X$, we use the following shifting operation on the curves. 

\begin{dfn}[negative $\bM$-shifting of curves]\label{def:shift_curve}
For a curve $\gamma$ in $\Sigma$ having its endpoints on $\partial^\ast\Sigma$, we define its \emph{(negative) $\bM$-shift} to be the ideal arc $\gamma^{\bM}$ obtained from $\gamma$ by shifting its endpoints to the nearest special point in the negative direction along the boundary. 
See \cref{fig:shifting_curve}.
\end{dfn}

\begin{figure}[ht]
    \centering
\begin{tikzpicture}
\fill[gray!20] (0,1.5) -- (-0.2,1.5) -- (-0.2,-1.5) -- (0,-1.5) --cycle;
\fill[gray!20] (4,1.5) -- (4+0.2,1.5) -- (4+0.2,-1.5) -- (4,-1.5) --cycle;
\draw[thick] (0,1.5) -- (0,-1.5);
\draw[thick] (4,-1.5) -- (4,1.5);
\filldraw(0,1) circle(1.5pt); 
\filldraw(0,0) circle(1.5pt);
\filldraw(0,-1) circle(1.5pt);
\filldraw(4,1) circle(1.5pt);
\filldraw(4,0) circle(1.5pt);
\filldraw(4,-1) circle(1.5pt);
\draw[red,thick] (0,-0.5) to[out=0,in=180] node[midway,above]{$\gamma$} (4,0.5);
\draw[thick,|->] (4.5,0) -- (5.5,0);
\begin{scope}[xshift=6cm]
\fill[gray!20] (0,1.5) -- (-0.2,1.5) -- (-0.2,-1.5) -- (0,-1.5) --cycle;
\fill[gray!20] (4,1.5) -- (4+0.2,1.5) -- (4+0.2,-1.5) -- (4,-1.5) --cycle;
\draw[thick] (0,1.5) -- (0,-1.5);
\draw[thick] (4,-1.5) -- (4,1.5);
\filldraw(0,1) circle(1.5pt); 
\filldraw(0,0) circle(1.5pt);
\filldraw(0,-1) circle(1.5pt);
\filldraw(4,1) circle(1.5pt);
\filldraw(4,0) circle(1.5pt);
\filldraw(4,-1) circle(1.5pt);
\draw[red,thick] (0,0) to[out=0,in=180] node[midway,above]{$\gamma^{\bM}$} (4,0);
\end{scope}
\end{tikzpicture}
    \caption{The negative $\bM$-shift of a curve.}
    \label{fig:shifting_curve}
\end{figure}
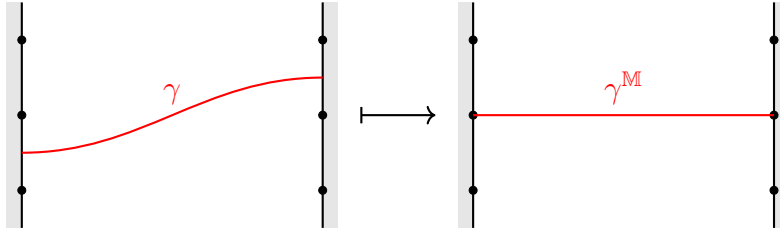

\begin{rem}
The inverse of $\bM$-shifting is the 2-dimensional version of ``moving right'' operation (\cref{moving_right}). 
\end{rem}

\begin{dfn}[skein lifting of integral $\P$-laminations]\label{def:skein_lift_X}
Let $\Sigma$ be an unpunctured marked surface, and $(L=\{(\gamma_i,w_i)\},\nu) \in \cL^p(\Sigma,\bZ)$ be an integral $\P$-lamination, 
where $\nu=(\nu_\alpha) \in \bZ^{\bB}$ is a pinning. We define the corresponding element $S_\X(L,\nu) \in \Sk{\Sigma}[\partial^{-1}]$ in the localized Muller skein algebra, as follows.
\begin{itemize}
    \item For each weighted non-peripheral loop $(\gamma_i,w_i)$, associate the element 
    \begin{align*}
        T_{w_i}([\gamma_i]) \in \Sk{\Sigma}[\partial^{-1}]
    \end{align*}
    as above.
    \item For each weighted non-peripheral arc $(\gamma_i,w_i)$, associate the element 
    \begin{align*}
        [\gamma^\bM]^{w_i} \in \Sk{\Sigma}[\partial^{-1}],
    \end{align*}
    namely the $w_i$-th power of the constant-elevation lift of the $\bM$-shift $\gamma^\bM$.
    \item For each boundary interval $E \in \bB$, associate the element 
    \begin{align*}
        [\alpha]^{\nu_\alpha} \in \Sk{\Sigma}[\partial^{-1}],
    \end{align*}
    namely the $\nu_\alpha$-th power of the boundary ideal arc $\alpha$. 
\end{itemize}
Then $S_\X(L,\nu)$ is defined to be the Weyl-normalized product of these elements. 
\end{dfn}
For an ideal triangulation $\tri$ and an ideal arc $\alpha \in e(\tri)$, the \emph{elementary lamination} is the unique tropical point $\ell_\alpha \in \X_\Sigma(\bZ^\sfT)$ characterized by $\check{\sfx}_\beta^\tri(\ell_\alpha)=\delta_{\beta,\alpha}$ for $\beta \in e(\tri)$. As the notation indicates, it can be verified that $\ell_\alpha$ does not depend on the triangulation $\tri$ that contains the ideal arc $\alpha$. 

\begin{prop}\label{prop:X-lamination_Muller_basis}
The skein lifting map gives a bijection
\begin{align*}
    S_\X:\X_\Sigma(\bZ^\sfT) \xrightarrow{\sim} \mathsf{B}^b(\Sigma,\bM)[\partial^{-1}],
\end{align*}
where the right-hand side denotes the basis of $\Sk{\Sigma}[\partial^{-1}]$ extending the bracelets basis $\mathsf{B}^b(\Sigma,\bM)$ of $\Sk{\Sigma}$.
Moreover, we have $S_\X(\ell_\alpha)=[\alpha]$ for any elementary lamination $\ell_\alpha$. 
\end{prop}

\begin{proof}
The first statement is clear from the construction: observe that the pinnings give rise to boundary arcs (and their inverses). 
For $\alpha \in e_{\interior}(\tri)$, the corresponding elementary lamination $\ell_\alpha$ is given by an arc $\gamma$ with weight one such that $\gamma^\bM=\alpha$, together with the zero pinning. Hence $S_\X(\ell_\alpha)=[\alpha]$ by the second rule in \cref{def:skein_lift_X}. 

For $\alpha \in \bB$, the corresponding elementary lamination $\ell_\alpha$ is given by the empty lamination together with the pinning $\nu=(\nu_\beta)_{\beta \in \bB}$ defined by $\nu_\beta:=\delta_{\beta,\alpha}$. Hence $S_\X(\ell_\alpha)=[\alpha]$ by the third rule in \cref{def:skein_lift_X}.  
\end{proof}

\begin{thm}\label{prop:lifting_compatible}
For any unpunctured marked surface $\Sigma$, we have the commutative diagram
\begin{equation*}
    \begin{tikzcd}
    \cL^a(\Sigma,\bZ) \ar[r,"S_\A"] \ar[d,"\check{p}_\Sigma^\sfT"'] & \overline{\mathscr{S}}_\Sigma(\bB) \ar[d,"\Phi_\Sigma"] \\
    \cL^p(\Sigma,\bZ) \ar[r,"S_\X"'] & \Sk{\Sigma}[\partial^{-1}].
    \end{tikzcd} 
\end{equation*}
Here the horizontal maps are given by skein lifting maps (\cref{def:skein_lift_A,def:skein_lift_X}), and the right vertical map is the state-clasp correspondence (\cref{thm:state-clasp}). In particular, it restricts to the commutative diagram
\begin{equation*} 
    \begin{tikzcd}
    \A_\Sigma(\bZ^\sfT) \ar[r,"S_\A"] \ar[d,"\check{p}_\Sigma^\sfT"'] & \overline{\mathscr{S}}_\Sigma(\bB)_{\congr} \ar[d,"\Phi_\Sigma"] \\
    \X_\Sigma(\bZ^\sfT) \ar[r,"S_\X"'] & \Sk{\Sigma}[\partial^{-1}].
    \end{tikzcd}
\end{equation*}
\end{thm}

\begin{proof}
Let $L \in \A_\Sigma(\bZ^\sfT)$ be a congruent integral $\A$-lamination. Since $\Phi_\Sigma$ maps a Weyl-normalized product to a Weyl-normalized product, 
we may assume that $L$ consists of a single weighted curve $(\gamma,k)$ without loss of generality. 
\begin{itemize}
    \item If $\gamma$ is a non-peripheral loop, the assertion is obvious.
    \item If $\gamma$ is a non-peripheral arc, then $\check{p}_\Sigma^\sfT(\gamma,k)$ is the same weighted arc. Then we need the equality 
    \begin{align*}
        \Phi_\Sigma([\gamma^-])=[\gamma^\bM],
    \end{align*}
    which immediately follows from the definition of the state-clasp correspondence.
    \item If $\gamma$ is a peripheral arc around a special point $m \in \bM$, then $\check{p}_\Sigma^\sfT(\gamma,k)$ consists of the empty lamination together with the pinning $\nu_\alpha=k$ assigned to the boundary interval $\alpha \in \bB$ such that $m=m_\alpha^-$. Then we need the equality 
    \begin{align*}
        \Phi_\Sigma([\gamma^-])=[\alpha],
    \end{align*}
    which also follows from the state-clasp correspondence. See \cref{fig:11_proof}.
\end{itemize}
Thus the assertion is proved. 
\end{proof}

\begin{figure}[h]
    \centering
\begin{tikzpicture}
\fill[gray!20] (0,0) -- (0,-0.1) -- (4,-0.1) -- (4,0) --cycle;
\draw[thick] (0,0) -- (4,0);
\filldraw(1,0) circle(1.5pt);
\filldraw(3,0) circle(1.5pt) node[below=0.2em]{$m$};
\node at (2,-0.5) {$\alpha$};
\node[red,scale=0.8] at (3-0.6,-0.25) {$-$};
\node[red,scale=0.8] at (3+0.6,-0.25) {$-$};
\draw[red,thick] (3-0.6,0) arc(180:0:0.6) node[midway,above]{$\gamma$};
\begin{scope}[xshift=5.5cm]
\fill[gray!20] (0,0) -- (0,-0.1) -- (4,-0.1) -- (4,0) --cycle;
\draw[thick] (0,0) -- (4,0);
\filldraw(1,0) circle(1.5pt);
\filldraw(3,0) circle(1.5pt) node[below=0.2em]{$m$};
\node at (2,-0.5) {$\alpha$};
\draw[red,thick] (1,0) to[bend left=30] node[midway,above]{$[\gamma^\bM]$} (3,0);
\end{scope}
\end{tikzpicture}
    \caption{Shift of a peripheral arc.}
    \label{fig:11_proof}
\end{figure}
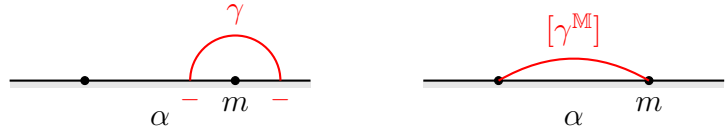

\begin{cor}\label{cor:basis}
For any unpunctured marked surface $\Sigma$, the images $S_\A(\A_\Sigma(\bZ^\sfT))\subset \sSk{\Sigma}_\congr$ and $S_\X(\X_\Sigma(\bZ^\sfT))\subset \Sk{\Sigma}[\partial^{-1}]$ give $\bZ_q$-bases.
\end{cor}

\begin{proof}
The assertion for $S_\X$ is a direct consequence of \cref{prop:Muller_basis,prop:X-lamination_Muller_basis}. For $S_\A$, observe that the states of the $\bB$-tangle diagrams in its image are admissible (\cref{def:admissible}). Then \cref{prop:admissible_span} tells us that $S_\A(\A_\Sigma(\bZ^\sfT))$ spans $\sSk{\Sigma}_\congr$ over $\bZ_q$. Moreover, since we know that $S_\X(\X_\Sigma(\bZ^\sfT))$ is linearly independent, our compatibility diagram in \cref{prop:lifting_compatible} implies that $S_\A(\A_\Sigma(\bZ^\sfT))$ is also linearly independent. Thus $S_\A(\A_\Sigma(\bZ^\sfT))$ gives a $\bZ_q$-basis of $\sSk{\Sigma}_\congr$.
\end{proof}

\subsection{Quantum trace and cutting maps.}\label{subsec:trace_cut}
Recall from \cref{sec:skein} that for each ideal triangulation $\tri$ of $\Sigma$, we have the quantum trace map (\cref{def:q-trace})
\begin{align*}
    \mathrm{Tr}_\tri: \sSk{\Sigma}_\congr \to \X_\tri^v
\end{align*}
and the cutting map (\cref{def:cutting})
\begin{align*}
    \mathrm{Cut}_\tri: \Sk{\Sigma}[\partial^{-1}] \to \A_\tri^q.
\end{align*}
Also recall from \cref{sec:cluster} that we have quantum cluster transformations among the quantum tori on their targets. These maps are compatible:

\begin{prop}[Bonahon--Wong \cite{BW11} and Muller \cite{Muller}]
For any flip $\tri \xrightarrow{f_\kappa} \tri'$, we have the commutative diagrams
\begin{equation*}
    \begin{tikzcd}
        & \Frac\X_{\tri'}^v \ar[dd,"\mu_\kappa^\ast"] &     & \Frac\A_{\tri'}^q \ar[dd,"\mu_\kappa^\ast"] \\
    \sSk{\Sigma}_\congr \ar[ru,"\mathrm{Tr}_{\tri'}"] \ar[rd,"\mathrm{Tr}_{\tri}"']& & \Sk{\Sigma}[\partial^{-1}]  \ar[ru,"\mathrm{Cut}_{\tri'}"] \ar[rd,"\mathrm{Cut}_{\tri}"']\\
        & \Frac\X_{\tri}^v  &     & \Frac\A_{\tri}^q.
    \end{tikzcd}
\end{equation*}
In particular, the quantum trace/cutting maps combine to give
\begin{align*}
    \mathrm{Tr}_\Sigma:\sSk{\Sigma}_\congr \to \cO_v(\X_\Sigma), \quad \mathrm{Cut}_\Sigma:\Sk{\Sigma}[\partial^{-1}] \to \cO_q(\A_\Sigma).
\end{align*}
\end{prop}

\begin{proof}
The commutative diagram for the quantum traces follows from Bonahon--Wong's work \cite{BW11}. Indeed, they have shown that the diagram 
\begin{equation}\label{eq:BW-diagram}
    \begin{tikzcd}
        & \Frac(\cZ_{\tri'}^q)_{\mathrm{bl}} \ar[dd,"\widetilde{\mu}_\kappa^\ast"]  \\
    \sSk{\Sigma} \ar[ru,"\mathrm{Tr}_{\tri'}"] \ar[rd,"\mathrm{Tr}_{\tri}"']\\
        & \Frac(\cZ_{\tri}^q)_{\mathrm{bl}} 
    \end{tikzcd}
\end{equation}
commutes, where $(\cZ_{\tri}^q)_{\mathrm{bl}}$ denotes the \emph{balanced Chekhov--Fock algebra}. The torus embedding given in \cref{eq:emb_to_CF} satisfies $\iota(\X_\tri^v) \subset (\cZ_{\tri}^v)_{\mathrm{bl}}$. The map $\widetilde{\mu}_\kappa^\ast$ is the Hiatt's morphism \cite{Hi}, which restricts to the quantum cluster Poisson transformation on $\X_{\tri'}^v$. Thus our assertion follows by restricting \eqref{eq:BW-diagram} to the congruent subalgebra by \cref{prop:congruent_image}.

The commutative diagram for the cutting maps is a part of \cite[Theorem 7.15]{Muller}. 
\end{proof}

\begin{thm}\label{prop:trace_cut_compatible}
For any unpunctured marked surface $\Sigma$, the following diagram commutes:
\begin{equation*}
    \begin{tikzcd}
    \sSk{\Sigma}_\congr \ar[d,"\Phi_\Sigma"'] \ar[r,"\mathrm{Tr}_\Sigma"] & \cO_v(\X_\Sigma) \ar[d,"p_\Sigma^\ast"] \\
    \Sk{\Sigma}[\partial^{-1}] \ar[r,"\mathrm{Cut}_\Sigma"'] & \cO_q(\A_\Sigma).
    \end{tikzcd}
\end{equation*}
\end{thm}
Below we give a proof of this statement.

\begin{lem}\label{lem:triangle_compatible}
\cref{prop:trace_cut_compatible} holds true for the triangle case $\Sigma=T$. 
\end{lem}

\begin{figure}[ht]
    \centering
\begin{tikzpicture}[scale=0.8]
\draw[blue] (-30:2) -- (90:2) -- (210:2) --cycle;
\foreach \i in {0,120,240} \rotatebox{\i}{
\draw[red,thick] (0,2)++(-60:1) arc(-60:-120:1);
}
\node[red] at (90:0.7) {$\gamma_1$};
\node[red] at (210:0.7) {$\gamma_2$};
\node[red] at (-30:0.7) {$\gamma_3$};
\foreach \i in {-30,90,210} \fill(\i:2) circle(2pt);
\begin{scope}[xshift=5cm,>=latex]
\draw[blue] (-30:2) -- (90:2) -- (210:2) --cycle;
\foreach \i in {-30,90,210}{
{\color{mygreen}
\draw($(\i:2)!0.5!(\i+120:2)$) circle(2pt);
\qarrow{$(\i:2)!0.5!(\i+120:2)$}{$(\i+120:2)!0.5!(\i+240:2)$}
}}
\node[mygreen] at (30:1.3) {$2$};
\node[mygreen] at (150:1.3) {$3$};
\node[mygreen] at (-90:1.3) {$1$};
\foreach \i in {-30,90,210} \fill(\i:2) circle(2pt);
\end{scope}
\end{tikzpicture}
    \caption{Triangle case: corner arcs (Left) and the quiver (Right).}
    \label{fig:triangle_computation}
\end{figure}
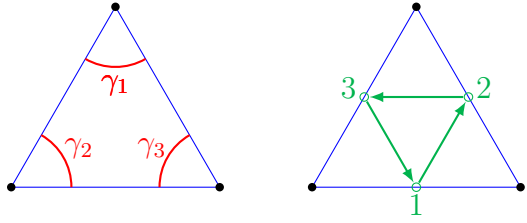

\begin{proof}
Let us label the corner arcs in $T$ and the quiver vertices as in \cref{fig:triangle_computation}. Then recall from \cref{ex:triangle} that the algebra $\sSk{T}$ is generated by $\{\gamma_i(\ve,\ve') \mid i=1,2,3,\ \ve,\ve' \in \{+,-\}\}$. 
By applying \cref{prop:admissible_span}, we see that $\sSk{T}$ can be generated by the elements $\gamma_i^\ve:=\gamma_i(\ve,\ve)$ for $i=1,2,3$ and $\ve \in \{+,-\}$. 
Note that $\gamma_i^-=(\gamma_i^+)^{-1}$ . 
On the other hand, the Muller skein algebra $\Sk{T}[\partial^{-1}]$ coincides with $\cO_q(\A_T)$ via $\mathrm{Cut}_T=\mathrm{id}$.  

Let $\tri$ denote the unique triangulation of $T$. 
Recall the relations $Z_i Z_j= q^{-\ve^\tri_{ij}/2}[Z_i Z_j]$ and $A_i A_j = q^{\pi_{ij}/2}[A_i A_j]$ for $i,j=1,2,3$. We compute
\begin{align*}
    p_T^\ast \circ \mathrm{Tr}^{\tri}(\gamma_1^+) &= p_T^\ast ([Z_2^{-1} Z_3^{-1}]) = p_T^\ast(q^{1/2}Z_2^{-1}Z_3^{-1}) \\
    &=q^{1/2}\mathbf{A}((f_3-f_1-f_2)/2) \cdot \mathbf{A}((f_1-f_2-f_3)/2) \\
    &=q^{1/2} \cdot q^{-\omega_\A(-\sff_2,\sff_3-\sff_1)/8}\mathbf{A}(-\sff_2/2)\mathbf{A}((\sff_3-\sff_1)/2)\\
    &\qquad\qquad\cdot q^{-\omega_\A(\sff_1-\sff_3,-\sff_2)/8}\mathbf{A}((\sff_1-\sff_3)/2)\mathbf{A}(-\sff_2/2) \\
    &=q^{1/2}\cdot q^{-1/2} \mathbf{A}(-\sff_2) \\
    &= A_2^{-1} =\Phi_T(\gamma_1^+).
\end{align*}
By the cyclic symmetry, we get $p_T^\ast \circ \mathrm{Tr}^{\tri}(\gamma_i^+) = \Phi_T(\gamma_i^+)$ for $i=1,2,3$. Thus the assertion is proved.
\end{proof}
We are going to ``glue" the above triangle cases. 
Given an ideal tringulation $\tri$ of $\Sigma$, let $\widehat{\tri}$ be the \emph{split triangulation} obtained from $\tri$ by duplicating each interior edge $\kappa \in e_{\interior}(\tri)$ into a pair $\{\kappa',\kappa''\}$ of parallel ideal arcs that bound a biangle $B_\kappa$. See \cref{fig:split_triangle}. For each $T \in t(\tri)$ and its edge $\kappa$, let us denote by $X_\kappa^T \in \X_T^v$ the corresponding cluster Poisson variable on $T$. 

\begin{figure}[ht]
    \centering
\begin{tikzpicture}[scale=0.8]
\draw[blue] (-2,0) --node[midway,left=0.3em]{$\beta$} (0,-2) --node[midway,right=0.3em]{$\gamma$} (2,0) --node[midway,right=0.3em]{$\delta$} (0,2) --node[midway,left=0.3em]{$\alpha$} cycle;
\draw[blue] (0,-2) -- (0,2);
\fill(-2,0) circle(1.5pt);
\fill(2,0) circle(1.5pt);
\fill(0,-2) circle(1.5pt);
\fill(0,2) circle(1.5pt);
\node[blue] at (0.2,0.5) {$\kappa$};
\node at (-1,0) {$T'$};
\node at (1,0) {$T''$};
\node at (1,-2) {$\triangle$};
\draw[squigarrow] (2.5,0) -- (3.5,0);
\begin{scope}[xshift=6cm]
\draw[blue] (-2,0) --node[midway,left=0.3em]{$\beta'$} (0,-2) --node[midway,right]{$\kappa'$} (0,2) --node[midway,left=0.3em]{$\alpha'$} cycle;
\fill(-2,0) circle(1.5pt);
\fill(0,-2) circle(1.5pt);
\fill(0,2) circle(1.5pt);
\node at (-1,0) {$T'$};
\end{scope}
\begin{scope}[xshift=7.5cm]
\draw[blue] (0,-2) to[bend left=30] (0,2);
\draw[blue] (0,-2) to[bend right=30] (0,2);
\fill(0,-2) circle(1.5pt);
\fill(0,2) circle(1.5pt);

\node at (0,0) {$B_\kappa$};
\end{scope}
\begin{scope}[xshift=9cm]
\draw[blue] (2,0) --node[midway,right=0.3em]{$\gamma''$} (0,-2) --node[midway,left]{$\kappa''$} (0,2) --node[midway,right=0.3em]{$\delta''$} cycle;
\fill(2,0) circle(1.5pt);
\fill(0,-2) circle(1.5pt);
\fill(0,2) circle(1.5pt);
\node at (1,0) {$T''$};
\node at (1,-2) {$\widehat{\triangle}$};
\end{scope}
\end{tikzpicture}
    \caption{The split triangulation $\widehat{\tri}$ associated with an ideal triangulation $\tri$.}
    \label{fig:split_triangle}
\end{figure}
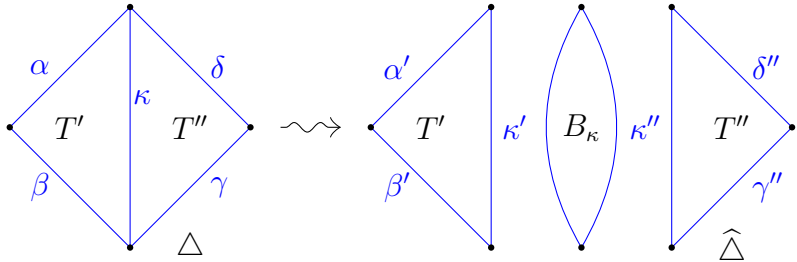

For each $\kappa \in e_{\interior}(\tri)$, we introduce the commutative algebra
\begin{align*}
    \X_{B_\kappa}^v:=\bZ_v[Y_\kappa^{\pm 1}]
\end{align*}
with a formal variable $Y_\kappa$,
and consider 
\begin{align*}
    \X_\wtri^v:=\bigotimes_{T \in t(\tri)} \X_T^v \otimes \bigotimes_{\kappa \in e_{\interior}(\tri)} \X_{B_\kappa}^v.
\end{align*}
We have the embedding (or the ``coproduct'') $\delta_\wtri: \X_\tri^v \to  \X_\wtri^v $
such that 
\begin{align*}
    \delta_\wtri(X_\kappa^\tri):=X_{\kappa'}^{T'}\otimes Y_\kappa \otimes X_{\kappa''}^{T''} \in \X_{T'}^v \otimes \X_{B_\kappa}^v \otimes \X_{T''}^v
\end{align*}
for $\kappa \in e_{\interior}(\tri)$, where $T'$ and $T''$ are the triangles that share the edge $\kappa$ (see \cref{fig:split_triangle}). For a boundary interval $\kappa \in \bB$ contained in a unique triangle $T$, we set $\delta_\wtri(X_\kappa^\tri):=X_\kappa^T$. 

Consider the multi-grading on $\X_\wtri^v$ valued in the lattice $N^\wtri:=\bigoplus_{\kappa \in e_{\interior}(\tri)} (\bZ \sfe_{\kappa'} \oplus \bZ \sfe_{\kappa''})$, where we set
\begin{align*}
    \deg(X_{\kappa'}^{T'}):= \sfe_{\kappa'}, \quad \deg(Y_\kappa):=-\sfe_{\kappa'}-\sfe_{\kappa''}
\end{align*}
in the notation of \cref{fig:split_triangle}. Let $(\X_\wtri^v)_0 \subset \X_\wtri^v$ be the subalgebra of degree $0$. Observe that the image of $\delta_\wtri$ lies in $(\X_\wtri^v)_0$ by $\mathrm{deg}(X_{\kappa'}^{T'}\otimes Y_\kappa \otimes X_{\kappa''}^{T''})=0$. 

Let $p_\wtri^\ast: (\X_\wtri^v)_0 \to \A_{\tri}^q$ be an algebra homomorphism such that 
\begin{align*}
    p_\wtri^\ast\left(\bigotimes_T \xi_T \otimes \bigotimes_\kappa (Y_\kappa)^{n_\kappa}\right):=\bigg[\prod_T p_T^\ast(\xi_T) \cdot A_\kappa^{2n_\kappa}  \bigg],
\end{align*}
where $\xi_T \in \X_T^v$ for $T \in t(\tri)$, and $n_\kappa \in \bZ$ for $\kappa \in e_{\interior}(\tri)$. Let us call $p_\wtri^\ast$ the \emph{split ensemble map}. 

\begin{lem}\label{lem:coproduct}
The following diagram commutes:
\begin{equation*}
    \begin{tikzcd}
    (\X_\wtri^v)_0 \ar[dr,"p_\wtri^\ast"'] &  \X_{\tri}^v \ar[l,"\delta_\wtri"] \ar[d,"p_{\tri}^\ast"] \\
    & \A_\tri^q.
    \end{tikzcd}
\end{equation*}
\end{lem}

\begin{proof}
It suffices to prove the equation for each generator $X_\kappa^\tri$, $\kappa \in e_{\interior}(\tri)$. Let $T',T''$ denote the triangles that share the edge $\kappa$, and use the labeling of edges as in \cref{fig:split_triangle}. We have $p_{\tri}^\ast X_{\kappa}^\tri=\left[\frac{A_\alpha A_\gamma}{A_\beta A_\delta} \right]$, while
\begin{align*}
    &p_\wtri^\ast \delta_\wtri (X_{\kappa}^\tri) =p_\wtri^\ast (X_{\kappa'}\otimes Y_\kappa \otimes X_{\kappa''})
    = \left[\left[\frac{A_{\alpha'}}{A_{\beta'}A_{\kappa'}}\right] \cdot(\overline{A}_\kappa)^2 \cdot\left[\frac{A_{\gamma''}}{A_{\delta''}A_{\kappa''}} \right]\right]
    =\left[\frac{A_\alpha A_\gamma}{A_\beta A_\delta} \right],
\end{align*}
as desired. 
\end{proof}
Let us consider the splitting map
\begin{align*}
    \Psi_\wtri: \sSk{\Sigma}_\congr \to \bigotimes_{T \in t(\tri)} \sSk{T}_\congr \otimes \bigotimes_{\kappa \in e_{\interior}(\tri)} \sSk{B_\kappa}_\congr 
\end{align*}
with respect to the splitting triangulation $\widehat{\tri}$, and the composite
\begin{align*}
    \mathrm{Tr}_\wtri: \sSk{\Sigma}_\congr \xrightarrow{\Psi_\wtri} \bigotimes_{T \in t(\tri)} \sSk{T}_\congr \otimes \bigotimes_{\kappa \in e_{\interior}(\tri)} \sSk{B_\kappa}_\congr \xrightarrow{\bigotimes \mathrm{Tr}_T \otimes\ \bigotimes \mathrm{Tr_{B_\kappa}}} \X_\wtri^v,
\end{align*}
where $\mathrm{Tr}_T: \sSk{T}_\congr \xrightarrow{\sim} \X_T^v$ is the isomorphism given by the restriction of \eqref{eq:triangle_isom} (recall \cref{eq:emb_to_CF}) for each $T \in t(\tri)$, and 
\begin{align*}
    \mathrm{Tr}_{B_\kappa}: \sSk{B_\kappa}_\congr \xrightarrow{\sim} \X_{B_\kappa}^v, \quad (\kappa_\perp^\pm)^2 \mapsto Y_{\kappa}^{\pm 1},
\end{align*}
for each $\kappa \in e(\tri)$. Here $\kappa^\pm_\perp$ denotes the unique (up to isotopy) $\bB$-tangle diagram in $B_\kappa$ equipped with the state $\pm$ on both of its ends. Its square generates $\sSk{B_\kappa}_\congr$, together with the relation $\kappa^+_\perp\cdot \kappa^-_\perp=1$. 
Then it is easy to see that 
\begin{align}\label{eq:q-trace_enhance}
    \mathrm{Tr}_\wtri = \delta_\wtri \circ \mathrm{Tr}_\tri.
\end{align}
In particular, the image of $\mathrm{Tr}_\wtri$ lies in $(\X_\wtri^v)_0$. 

\begin{lem}\label{lem:compatible_skein}
The following diagram commutes:
\begin{equation}\label{eq:compatible_skein}
    \begin{tikzcd}
        \sSk{\Sigma}_\congr \ar[r,"\mathrm{Tr}_\wtri"] \ar[d,"\Phi_\Sigma"'] & (\X_\wtri^v)_0 \ar[d,"p_\wtri^\ast"]\\
        \Sk{\Sigma}[\partial^{-1}] \ar[r,"\mathrm{Cut}_{\tri}"'] & \A_\tri^q.
    \end{tikzcd}
\end{equation}
\end{lem}

\begin{proof}
We repeatedly apply the compatibility of splitting/cutting maps (\cref{thm:split_cut_compatibility}) for every edge of $\tri$. Obviously we can restrict the diagram to the congruent subalgebras, and obtain the commutative diagram
\begin{equation*}
    \begin{tikzcd}
        \sSk{\Sigma}_\congr \ar[r,"\Psi_\wtri"] \ar[d,"\Phi_\Sigma"'] & \bigotimes_{T \in t(\tri)} \sSk{T}_\congr \otimes \bigotimes_{\kappa \in e_{\interior}(\tri)} \sSk{B_\kappa}_\congr  \ar[d,"{\big[\bigotimes_T \Phi_T \otimes \bigotimes_\kappa \Phi_{B_\kappa}\big]}"]\\
        \Sk{\Sigma}[\partial^{-1}] \ar[r,"\mathrm{Cut}_{\tri}"'] & \Sk{\Sigma}[\tri^{-1}].
    \end{tikzcd}
\end{equation*}
Observe that the algebra in the top-right is identified with $\X_{\wtri}^v$, and the one in the bottom-right is the same as $\A_\tri^q$. By \cref{lem:triangle_compatible}, the map $\big[\bigotimes_T \Phi_T \otimes \bigotimes_\kappa \Phi_{B_\kappa}\big]$ coincides with $p^\ast_\wtri$. 
\end{proof}

\begin{proof}[Proof of \cref{prop:trace_cut_compatible}]
Combining \cref{lem:coproduct,lem:compatible_skein} and \eqref{eq:q-trace_enhance}, we get the commutativity diagram
\begin{equation*}
    \begin{tikzcd}
     \sSk{\Sigma}_\congr \ar[r,"\mathrm{Tr}_\wtri"'] \ar[d,"\Phi_\Sigma"'] \ar[rr,bend left,"\mathrm{Tr}_\tri"] & (\X_\wtri^v)_0 \ar[dr,"p_\wtri^\ast"'] & \X_\tri^v \ar[d,"p_\tri^\ast"] \ar[l,"\delta_\wtri"]\\
    \Sk{\Sigma}[\partial^{-1}] \ar[rr,"\mathrm{Cut}_{\tri}"'] && \A_\tri^q.
    \end{tikzcd}
\end{equation*}
Then the commutativity of the outer-most square is the desired one. 
\end{proof}
Now we define the \emph{quantum duality maps} by 
\begin{align*}
    &\bI_\A:=\mathrm{Tr}_\Sigma \circ S_\A: \A_\Sigma(\bZ^\sfT) \to \cO_v(\X_\Sigma), \\
    &\bI_\X:=\mathrm{Cut}_\Sigma \circ S_\X: \X_\Sigma(\bZ^\sfT) \to \cO_q(\A_\Sigma).
\end{align*}
Here is our main theorem:

\begin{thm}\label{thm:compatibility}
For any unpunctured marked surface $\Sigma$, we have the commutative diagram
\begin{equation*}
    \begin{tikzcd}
    \A_\Sigma(\bZ^\sfT) \ar[r,"\bI_\A"] \ar[d,"\check{p}_\Sigma^\sfT"'] & \cO_v(\X_\Sigma) \ar[d,"p_\Sigma^\ast"] \\
    \X_\Sigma(\bZ^\sfT) \ar[r,"\bI_\X"'] & \cO_q(\A_\Sigma).
    \end{tikzcd}
\end{equation*}
\end{thm}

\begin{proof}
The diagram can be decomposed as 
\begin{equation*}
    \begin{tikzcd}
    \A_\Sigma(\bZ^\sfT) \ar[r,"S_\A"] \ar[d,"\check{p}_\Sigma^\sfT"'] & \overline{\mathscr{S}}_\Sigma(\bB)_{\congr} \ar[d,"\Phi_\Sigma"] \ar[r,"\mathrm{Tr}_\Sigma"] & \cO_v(\X_\Sigma) \ar[d,"p_\Sigma^\ast"] \\
    \X_\Sigma(\bZ^\sfT) \ar[r,"S_\X"'] & \Sk{\Sigma}[\partial^{-1}] \ar[r,"\mathrm{Cut}_\Sigma"'] & \cO_q(\A_\Sigma),
    \end{tikzcd}
\end{equation*}
whose commutativity follows from \cref{prop:lifting_compatible,prop:trace_cut_compatible}. 
\end{proof}

\begin{rem}[The case with punctures]\label{rem:peripheral}
Let $\Sigma$ be a marked surface with punctures. 
The definition of $\bI_\A(L)$ for an $\A$-lamination $L$ without peripheral loops is still the same as \cref{def:skein_lift_A}. 
Following \cite[Definition 3.1]{AK}, we associate the Casimir element $C_m^{-w_i} \in \cO_v(\X_\Sigma)$ to a weighted peripheral loop $(\gamma_i,w_i)$ around $m \in \bM_\circ$. Here, for each ideal triagulation $\tri$ of $\Sigma$, $C_m$ has the expression
\begin{align*}
    C_m = \bigg[ \prod_{\alpha \in e(\tri)} (X_\alpha^\tri)^{-n_\alpha} \bigg],
\end{align*}
where $n_\alpha$ is the number of half-edges of $\alpha$ incident to $m$. Then for an $\A$-lamination $L$ with peripheral loops, the element $\bI_\A(L) \in \cO_v(\X_\Sigma)$ is defined to be the (commutative) product of these elements.  

One may wonder if this construction factors through a ``skein lifting'' $S_\A$: this point will be discussed in our forthcoming paper \cite{IKp}. 
\end{rem}

The following are properties generally expected for quantum duality maps:

\begin{thm}\label{thm:lowest_terms}
Let $\Sigma$ be any marked surface. 
\begin{enumerate}
    \item For any $L \in \A_\Sigma(\bZ^\sfT)$, the quantum Laurent expression of $\bI_\A(L) \in \cO_v(\X_\Sigma)$ has the unique lowest term 
    \begin{align}\label{eq:lowest_term}
        \bigg[\prod_{\alpha \in e(\tri)}(X_\alpha^\tri)^{-\sfa_\alpha(L)} \bigg]
    \end{align}
    with respect to the polynomial grading for any ideal triangulation $\tri$.
    \item Assume that $\Sigma$ is unpunctured. Then for any $(L,\nu) \in \X_\Sigma(\bZ^\sfT)$, the quantum Laurent expression of $\bI_\X(L,\nu) \in \cO_q(\A_\Sigma)$ has the form\footnote{This property is called ``parametrized by the tropical point $(L,\nu)$'' by Qin \cite{Qin}. The monomial term is the highest term with respect to the \emph{dominance order} \cite{Qin}.}
    \begin{align}\label{eq:pointed}
        \bI_\X(L,\nu)= \bigg[\prod_{\alpha \in e(\tri)} A_\alpha^{\check{\sfx}_\alpha^\tri(L,\nu)} \bigg] \cdot F_\X^\tri(L,\nu)
    \end{align}
    for any ideal triangulation $\tri$, where $F_\X^\tri(L,\nu)$ is a quantum polynomial in $\{p_\tri^\ast X_\alpha^\tri \mid \alpha \in e(\tri)\}$. 
\end{enumerate}
\end{thm}

\begin{proof}
(1): Consider the element $S_\A(L) \in \sSk{\Sigma}$. First assume that $L$ has no peripheral components with negative weights. 
Given an ideal triangulation $\tri$ of $\Sigma$, the splitting map $\Psi_\tri$ expands $S_\A(L)$ into a state-sum. The unique term obtained by assigning the state $-$ to every intersection of $L$ with the edges of $\tri$ produces the asserted term \eqref{eq:lowest_term}, up to a $q^{1/2}$-factor. Here recall the assignment \eqref{isom_triangle}. To see that this $q^{1/2}$-factor is actually $1$, observe that after splitting $S_\A(L)$ along an edge $\alpha \in e(\tri)$, the resulting strands on the new edge $\alpha'$ and those on $\alpha''$ have reversed height orders to each other. See \cref{fig:splitting}. So the extra $q^{1/2}$-factors arising from their Weyl normalization (\cref{dfn_Weyl}) cancels. It shows that the image of each term after splitting under $\mathrm{Tr}_\tri$ has the Weyl-normalized form. In particular, the putative lowest term exactly has the asserted form \eqref{eq:lowest_term}.

If we replace the state $-$ with $+$ at an edge $E \in e(\tri)$, say corresponding to the variable $Z_{i+1}$, then the corresponding assignment $[Z_{i+1}Z_{i+2}]$ is replaced with $[Z_{i+1}^{-1} Z_{i+2}]$. Their monomial difference is $Z_{i+1}^{-2}=X_{i+1}$ up to a $q^{1/2}$-factor. Hence the other terms necessarily have higher polynomial degrees than the term \eqref{eq:lowest_term}, as asserted. 

Observe that a negative peripheral arc component contributes as a monomial corresponding to the state $+$ assigned to every intersection with the edges, due to the bad arc condition. 
The monomial has positive powers in $X_\alpha^\tri$, which agrees with the assertion since the component contributes negatively to the coordinates $\sfa_\alpha(L)$. 

(2): Consider the element $S_\X(L,\nu) \in \Sk{\Sigma}[\partial^{-1}]$, and its pull-back $\Phi_\Sigma^{-1}(S_\X(L,\nu)) = S_\A(\widetilde{L})$ with $\widetilde{L}:=(\check{p}_\Sigma^\sfT)^{-1}(L,\nu) \in \cL^a(\Sigma,\bZ)$. Here the equality follows from the first diagram in \cref{prop:lifting_compatible}. Given an ideal triangulation $\tri$, the splitting map $\Psi_\tri$ expands $S_\A(\widetilde{L})$ into a state-sum as in the previous paragraph. 
Then the unique lowest term is $[\prod_{\kappa \in e(\tri)} (Z^\tri_\kappa)^{\bi(\widetilde{L},\kappa)}]$, which is sent via the ensemble map \eqref{eq:ensemble_balanced} as
\begin{align*}
    p_\tri^\ast\bigg(\bigg[\prod_{\kappa \in e(\tri)} (Z^\tri_\kappa)^{\bi(\widetilde{L},\kappa)}\bigg] \bigg) = \bigg[\prod_{\kappa,\alpha \in e(\tri)} A_\alpha^{-\frac 1 2 \bi(\widetilde{L},\kappa) \ve_{\kappa\alpha}^\tri}\bigg] = \bigg[\prod_{\alpha \in e(\tri)} A_\alpha^{-\sum_{\kappa} a_\kappa(\widetilde{L}) \ve_{\kappa\alpha}^\tri}\bigg],
\end{align*}
and the exponent of $A_\alpha$ in the last term is computed as
\begin{align*}
    -\sum_{\kappa \in e(\tri)} a_\kappa(\widetilde{L}) \ve_{\kappa\alpha}^\tri = \sum_{\kappa \in e(\tri)} a_\kappa(\widetilde{L}) \ve_{\alpha\kappa}^\tri = \check{\sfx}_\alpha^\tri(\check{p}_\tri^\sfT(\widetilde{L})) = \check{\sfx}_\alpha^\tri(L,\nu).
\end{align*}
Here we used the skew-symmetry of $\ve^\tri$. Again, the other terms can be written as the lowest term multiplied by a monomial of $\X$-variables up to $q^{1/2}$-factors. The assertion is proved.  
\end{proof}

\begin{thm}\label{thm:skein_q-Poisson_isom}
Assume that $\Sigma$ is unpunctured and $|\bM_\partial| \geq 2$.
Then the quantum trace/cutting maps give isomorphisms
\begin{align*}
    \mathrm{Tr}_\Sigma: \sSk{\Sigma}_\congr \xrightarrow{\sim} \cO_v(\X_\Sigma), \quad \mathrm{Cut}_\Sigma: \Sk{\Sigma}[\partial^{-1}] \xrightarrow{\sim} \cO_q(\A_\Sigma)
\end{align*}
with $v=q^{-2}$. In particular,
the images $\bI_\A(\A_\Sigma(\bZ^\sfT)) \subset \cO_v(\X_\Sigma)$ and $\bI_\X(\X_\Sigma(\bZ^\sfT)) \subset \cO_q(\A_\Sigma)$ give $\bZ_v$- and $\bZ_q$-bases respectively.
\end{thm}

\begin{proof}
The assertion for $\cO_q(\A_\Sigma)$ is exactly \cite[Theorem 9.8]{Muller}. 

The proof of the assertion for $\cO_v(\X_\Sigma)$ relies on the recent result of Mandel--Qin \cite{MQ}. By comparing with their construction given in \cite[Section 10.4]{MQ}, we see that the elements $\bI_\A(L) \in \cO_v(\X_\Sigma)$ for $L \in \A_\Sigma(\bZ^\sfT)$ are precisely the quantum theta functions \cite[Theorem 10.11]{MQ}. Moreover, thanks to \cite[Proposition 0.14]{GHKK} and \cite[Theorem 1.2 (5)]{DM} (and \cite[Theorem 1.3]{GS18} for the existence of a cluster Donaldson--Thomas transformation), the Full Fock--Goncharov Conjecture holds in our case so that the quantum theta functions give an $\bZ_v$-basis of $\cO_v(\X_\Sigma)$. In particular, the elements $\bI_\A(L) =\mathrm{Tr}_\Sigma \circ S_\A(L)$ span $\cO_v(\X_\Sigma)$, hence $\mathrm{Tr}_\Sigma$ is surjective. 
\end{proof}

\begin{rem}[Relation to the Allegretti's work \cite{All} in the disk case]\label{rem:Allegretti}

In \cite{All}, Allegretti proposed a construction of quantum duality map $\A_\Sigma(\bZ^\sfT) \to \cO_v(\X_\Sigma)$ when $\Sigma$ is a marked disk purely in terms of quantum cluster varieties. Let us explain how our skein theoretic construction agrees with his one.

In his work, integral $\A$-laminations on a marked disk are defined as isotopy classes of weighted collections $\ell$ of ideal arcs (having endpoints at $\bM$) \cite[Definition 4.13]{All}. Such a collection is obtained as the negative $\bM$-shift $\ell=L^{\bM}$ of an integral $\A$-lamination $L \in \A_\Sigma(\bZ^\sfT)$ in our sense. Take an ideal triangulation $\tri=\tri_\ell$ that contains all the arcs in $\ell$ as its edges, and define \cite[Definition 5.4]{All}
\begin{align}\label{eq:duality_Allegretti}
    \widetilde{\bI}'_\A(\ell):= \mathbf{A}^{\!\tri}\bigg(\sum_{\alpha \in e(\tri)} w_\alpha(\ell)\sff_\alpha^{\tri}\bigg) = \bigg[\prod_{\alpha \in e(\tri)} A_\alpha^{w_\alpha(\ell)}\bigg] \in \cO_q(\A_\Sigma),
\end{align}
where $w_\alpha(\ell) \in \bZ$ denotes the weight of the arc in $\ell$ which is parallel to the edge $\alpha$; it is $0$ if there is no arcs in $\ell$ parallel to $\alpha$. Then it is verified that the element $\bI'_\A(\ell) \in \cO_q(\A_\Sigma)$ acturally lies in the image $p^\ast_\Sigma(\cO_v(\X_\Sigma))$ \cite[Theorem 5.6]{All}. Then define 
\begin{align*}
    \bI'_\A(\ell):=(p^\ast_\Sigma)^{-1}(\widetilde{\bI}'_\A(\ell)) \in \cO_v(\X_\Sigma).
\end{align*}
Here are some remarks:
\begin{itemize}
    \item the relation between his quantum parameters is given by $q=\omega^4$, which agrees with our relation $q=v^{-2}$ since the $\omega$-exponent in the defining relation of his quantum cluster $K_2$-torus (\cite[Definition 4.11]{All}) is $-2$ times ours. 
    \item the $\mathbf{y}$-terms appearing in the proof of \cite[Lemma 5.5]{All} corresponds to the correction by $m_{\alpha\beta}$ in our \eqref{eq:ensemble_linear}. 
\end{itemize}
Then for $\ell=L^{\bM}$ with $L \in \A_\Sigma(\bZ^\sfT)$, taking $\tri=\tri_\ell$, we compute 
\begin{align*}
    p_\Sigma^\ast (\bI_\A(L)) &= \bI_\X(p_\Sigma^\sfT(L)) & &\mbox{(by \cref{thm:compatibility})} \\
    &= \bigg[\prod_{\alpha \in e(\tri)} A_\alpha^{\check{\sfx}_\alpha^\tri(p_\Sigma^\sfT(L))}\bigg] & &\mbox{(by \cref{thm:lowest_terms} (2))}\\
    &=\bigg[\prod_{\alpha \in e(\tri)} A_\alpha^{w_\alpha(\ell)}\bigg] & &\mbox{(see \cref{fig:comparison_Allegretti})} \\
    &=p^\ast_\Sigma(\bI'_\A(\ell)). & &\mbox{(compare to \eqref{eq:duality_Allegretti})}
\end{align*}
Since $p_\Sigma^\ast$ is injective, we get $\bI_\A(L)=\bI'_\A(\ell)$. 
\end{rem}

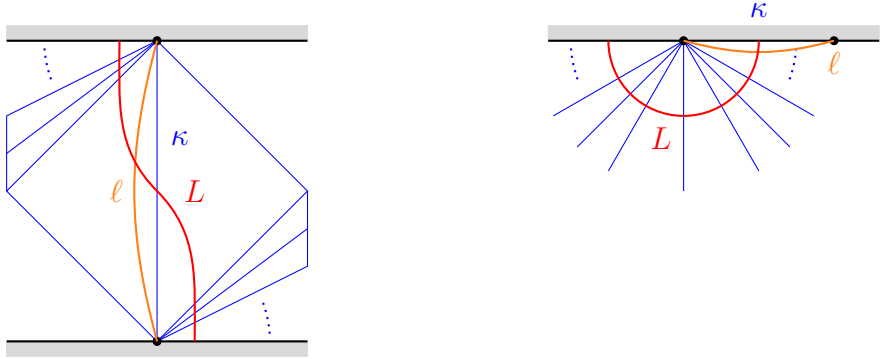
\begin{figure}[ht]
    \centering
\begin{tikzpicture}
\bline{-2,-2}{2,-2}{0.2};
\tline{-2,2}{2,2}{0.2};
\draw[blue] (-2,0) -- (0,-2) -- (2,0) -- (0,2) -- cycle;
\draw[blue] (0,2) -- (0,-2);
\foreach \i in {1,2} \draw(-2,0)++(0,\i*0.5) coordinate(A\i);
\draw[blue] (-2,0) -- (A1) -- (0,2);
\draw[blue] (A1) --(A2) -- (0,2);
\foreach \i in {1,2} \draw(2,0)++(0,-\i*0.5) coordinate(B\i);
\draw[blue] (2,0) -- (B1) -- (0,-2);
\draw[blue] (B1) --(B2) -- (0,-2);
\draw[blue,thick,dotted] (0,2)++(-160:1.5) arc(-160:-178:1.5);
\draw[blue,thick,dotted] (0,-2)++(20:1.5) arc(20:2:1.5);
\filldraw (0,2) circle(1.5pt);
\filldraw (0,-2) circle(1.5pt);
\draw[myorange,thick] (0,-2) to[bend left=15] node[midway,left]{$\ell$} (0,2);
\draw[red,thick] (0.5,-2) to[out=90,in=-45] (0,0) to[out=135,in=-90] (-0.5,2);
\node[blue] at (0.3,0.7) {$\kappa$};
\node[red] at (0.5,0) {$L$};

\begin{scope}[xshift=7cm]
\tline{-1.8,2}{2.6,2}{0.2};
\foreach \i in {-30,-45,-60,-90,-120,-135,-150} \draw[blue] (0,2) --++(\i:2);
\draw[blue,thick,dotted] (0,2)++(-160:1.5) arc(-160:-178:1.5);
\draw[blue,thick,dotted] (0,2)++(-20:1.5) arc(-20:-2:1.5);
\filldraw (0,2) circle(1.5pt);
\filldraw (2,2) circle(1.5pt);
\draw[myorange,thick] (0,2) to[bend right=15]  (2,2) node[below]{$\ell$};
\draw[red,thick] (-1,2) arc(-180:0:1);
\node[blue] at (1,2.4) {$\kappa$};
\node[red] at (-0.3,0.7) {$L$};
\end{scope}
\end{tikzpicture}
    \caption{Comparison of $L$ and $\ell$, where $\ell$ is parallel to an edge $\kappa \in e(\tri)$. In both cases, $\check{\sfx}_\alpha^\tri(p_\Sigma^\sfT(L))=\delta_{\alpha,\kappa}=w_\alpha(\ell)$ for $\alpha \in e(\tri)$.}
    \label{fig:comparison_Allegretti}
\end{figure}
\section{Skein proof of the strong positivity}\label{sec:positivity}

\subsection{Positivity of structure constants of the bracelets basis}

In the rest of the paper, we assume $\CR=\BZ[q^{\pm1/2}]$, the Laurent polynomial ring in $q^{1/2}$ over $\BZ$. 
Let $\BZ_{\geq0}$ denote the set of non-negative integers, and $\BZ_{\geq0}[q^{\pm1/2}]\subset \BZ[q^{\pm1/2}]$ the submonoid consisting of Laurent polynomials with non-negative coefficients. 
Let $\BD$ and $\BA$ denote the disk and the annulus respectively.

\begin{dfn}
Let $A$ be an $\cR$-algebra 
equipped with an $\cR$-basis $\sfB$.
\begin{enumerate}
    \item A finite $\cR$-linear sum $\sum_i c_i B_i$ with $B_i \in \sfB$ is said to have {\em non-negative coefficients} if  $c_i \in \BZ_{\geq0}[q^{\pm 1/2}]$ for all $i$. 
    \item If the structure constants of $\sfB$ are non-negative, namely
    \begin{align*}
        B_1 \cdot B_2 = \sum_{B \in \sfB} c(B_1,B_2;B) B
    \end{align*}
    with $c(B_1,B_2;B) \in \BZ_{\geq0}[q^{\pm 1/2}]$ for any $B \in \sfB$, 
    then we call $\sfB$ {\it a positive basis}. 
\end{enumerate}
\end{dfn}

After Theorem 1.3 in \cite{Le_positivity}, 
the bracelets basis $\sfB^b(\Sigma,\BM)$ (see \eqref{eq:bracelets} below) of $\Sk{\Sigma}$ has been a candidate for a positive basis. 
Toward a skein theoretic proof of the positivity of bracelets basis, we prove the following theorem. 

The following theorem corresponds to the claim (1) in Theorem \ref{thm_intro}. 
\begin{thm}\label{thm_skein}
Let $(\Sigma,\BM)$ be a marked disk with at least 3 special points or a marked annulus with exactly one special point on each boundary component. Then the bracelets basis $\sfB^b(\Sigma,\BM)$ of the Muller skein algebra $\Sk{\Sigma}$ is positive. 
\end{thm}

\paragraph{\textbf{(Chebyshev) bunches}}
Let $P(x)=\sum_{i=0}^n c_i x^i$ be a polynomial in $x$ with $c_i\in \CR\ (i=0,1,\dots,n)$. 
For a one-component multicurve $\al$ on $\Sigma$, 
one can consider the linear sum $P(\al)=\sum_{i=0}^n c_i \al^i\in \Sk{\Sigma}$.

Recall that the Chebyshev polynomials of the first kind are given by the following initial data and the recurrence relation: 
\begin{align}\label{Tn_recurrence}
    T_0(x)=2,\quad T_1(x)=x,\quad T_{n+2}(x)=x T_{n+1}(x)-T_{n}(x)\quad (n\geq0).
\end{align}
It is easy to see that $x^n$ for $n\in \BZ_{\geq 0}$ can be written as a linear sum of $1,T_1(x),\dots, T_n(x)$ with non-negative coefficients.

Given a one-component multicurve $\al$ on $\Sigma$, 
\begin{enumerate}
    \item Its {\it $k$-th bunch} 
    is defined to be $P(\alpha)$ with $P(x)=x^k$. Diagrammatically, it is the $k$-parallel copies of $\al$ with simultaneous crossings
    at each of $\BMp$. 
    It is presented by $[\al^k]\in \sS_q(\Sigma)$ using the Weyl normalization.
    \item 
    Its {\it $k$-th Chebyshev bunch} is defined to be $P(\alpha)$, where
    \begin{align*}
        P(x):=\begin{cases}
        x^k & \mbox{if $\al$ is an ideal arc},\\
        T_k(x) & \mbox{if $\al$ is a loop}.
        \end{cases}
    \end{align*}
\end{enumerate}
The Chebyshev bunch $T_k(\alpha)$ of a loop is also called a \emph{bracelet} \cite{Th}.

\begin{dfn}[bracelets basis]\label{def:bracelets}
Consider the set
\begin{align}\label{eq:bracelets}
    \sfB^b(\Sigma,\BM)=\big\{[\al_0\, T_{n_1}(\al_1)\,T_{n_2}(\al_2)\dots T_{n_k}(\al_k)]\,\big|\, k\in \bZ_{\geq 0},\ n_i\in \bZ_{\geq 1},\ \text{$\al_0,\dots,\al_k$ satisfy $(1)$--$(3)$}\big\},
\end{align}
where
\begin{enumerate}
\item $\al_0$ is a (possibly empty) simple multicurve consisting only of ideal arcs on $\Sigma$;
\item if $k \geq 1$, then $\al_1,\dots ,\al_k$ are simple loops, which are not homotopic to each other;
\item $\al_0,\al_1,\dots,\al_k$ do not intersect with each other except possibly at their endpoints.
\end{enumerate}
It is known that the set $\sfB^b(\Sigma,\BM)$ is also a basis of $\Sk{\Sigma}$, called the {\em bracelets basis}. See \cite[Section 2.3]{Le_positivity}) for more details.
\end{dfn}

Note that any element in $\sfB^b(\Sigma,\BM)$ is {\it $q^{1/2}$-equivalent} (\emph{i.e.}, equal up to a power of $q^{1/2}$ in $\sSsq(\BM)$) to a product of Chebyshev bunches in $\Sk{\Sigma}$.

In the course of proof, we use the following observations: 
\begin{itemize}
    \item[(a)] For a product of (Chebyshev) bunches $\beta_1 \beta_2\dots \beta_k$, 
    two bunches $\beta_i, \beta_j$ for $|i-j|=1$ are said to be {\it consecutive}. 
    Observe that if the product $\beta_1 \beta_2\dots \beta_k$ with the minimal crossing number has at least one crossings in $\Int \Sigma$, then by applying the relation (C), one can reorder $\beta_1 \beta_2\dots \beta_k$ so that only consecutive (Chebyshev) bunches have crossings in $\Int \Sigma$ up to $q^{1/2}$-equivalence. 
    \item[(b)] By using the relation (C) in Definition \ref{dfn_Muller} and the Reidemeister move II', 
    we have the following. 
    \begin{align}
    q^{-1}\begin{array}{c}\includegraphics[scale=0.27]{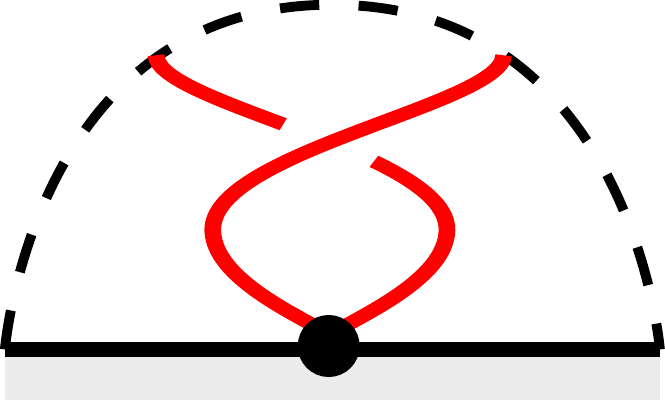}\end{array}=\begin{array}{c}\includegraphics[scale=0.27]{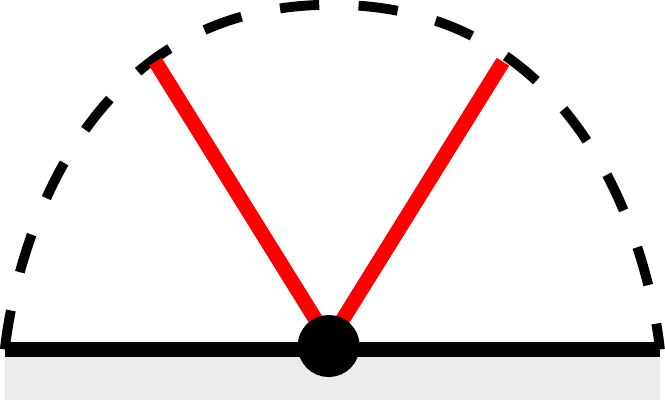}\end{array}=q\begin{array}{c}\includegraphics[scale=0.27]{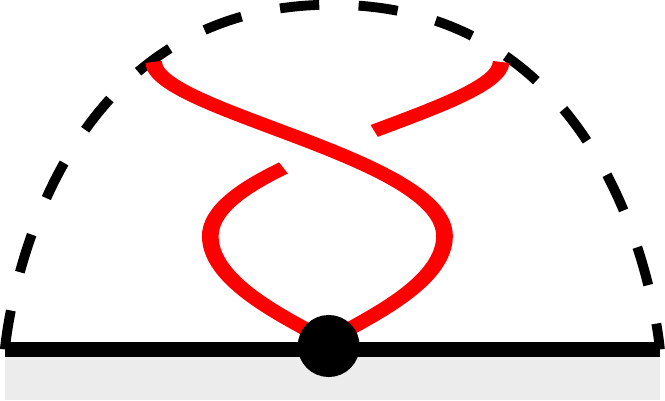}\end{array}\label{rel_nocrossing}
    \end{align}
\end{itemize}

\subsection{Positivity for marked disks}
\begin{proof}[Proof of Theorem \ref{thm_skein} for marked disks]
We consider the case that $\Sigma$ is a marked disk $(\BD,\BM)$ with at least 3 special points and $\BM_\circ =\emptyset$. 
By the observation (a), any element of $\sfB^b(\Sigma,\BM)$ is $q^{1/2}$-equivalent to a product of Chebyshev bunches. 
We claim that for (possibly the same) Chebyshev bunches $\al_1,\al_2,\dots,\al_k\in \sfB^b(\Sigma,\BM)$, 
the product $\al_1\al_2 \dots \al_k$ 
is expanded into a linear sum of finitely many elements of $\sfB^b(\Sigma,\BM)$ with non-negative coefficients.
There are two possible cases:
\begin{enumerate}
    \item Any two of $\al_1,\al_2,\dots ,\al_k$ do not intersect in $\Int \Sigma$\label{positivity_disk_condi1}
    \item $\al_i$ and $\al_j$ intersect in $\Int \Sigma$ for some $i,j\in \{1,2,\dots, k\}$.\label{positivity_disk_condi2}
\end{enumerate}
In the case (\ref{positivity_disk_condi1}), the product is $q^{1/2}$-equivalent to the Weyl normalization $[\al_1\al_2\dots\al_k]$, which is an element of $\sfB^b(\Sigma,\BM)$. 
This proves the claim for the case (1). 
Hence, it is sufficient to show the positivity in the case (\ref{positivity_disk_condi2}).

For the product $\al_1\al_2\dots \al_k$, we can suppose that only consecutive Chebyshev bunches $\al_i,\ \al_{i+1}$ intersect each other in $\Int \BD$ by the observation (a). 
Then, $\al_i$ and $\al_{i+1}$ intersect as depicted in the left-hand side of (\ref{disk_resolution}). 
Using the relation (A), Reidemeister moves II and III and (\ref{rel_nocrossing}), we have 
\begin{eqnarray}
\begin{array}{c}\includegraphics[scale=0.18]{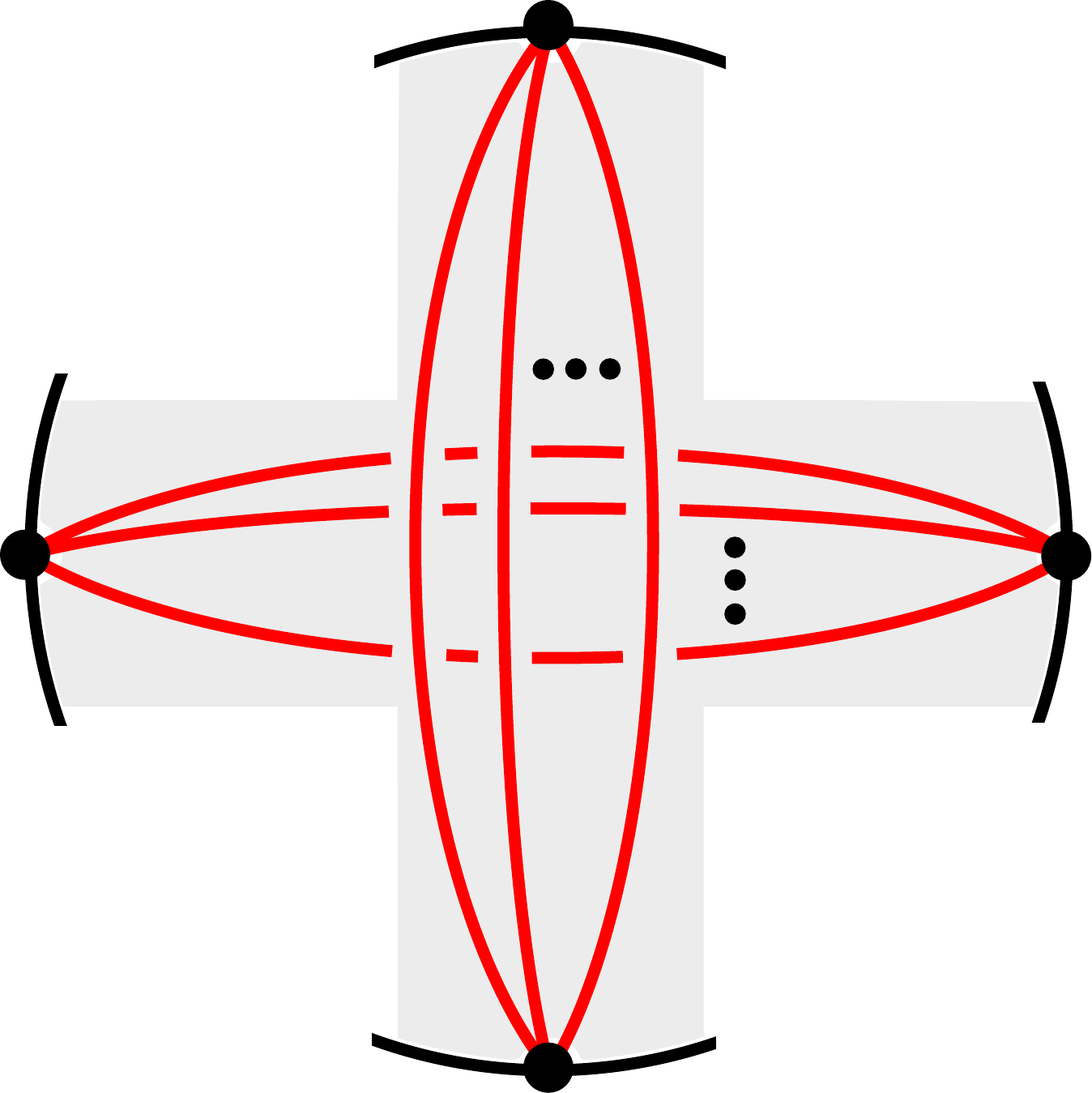}\end{array}
&=&q\begin{array}{c}\includegraphics[scale=0.18]{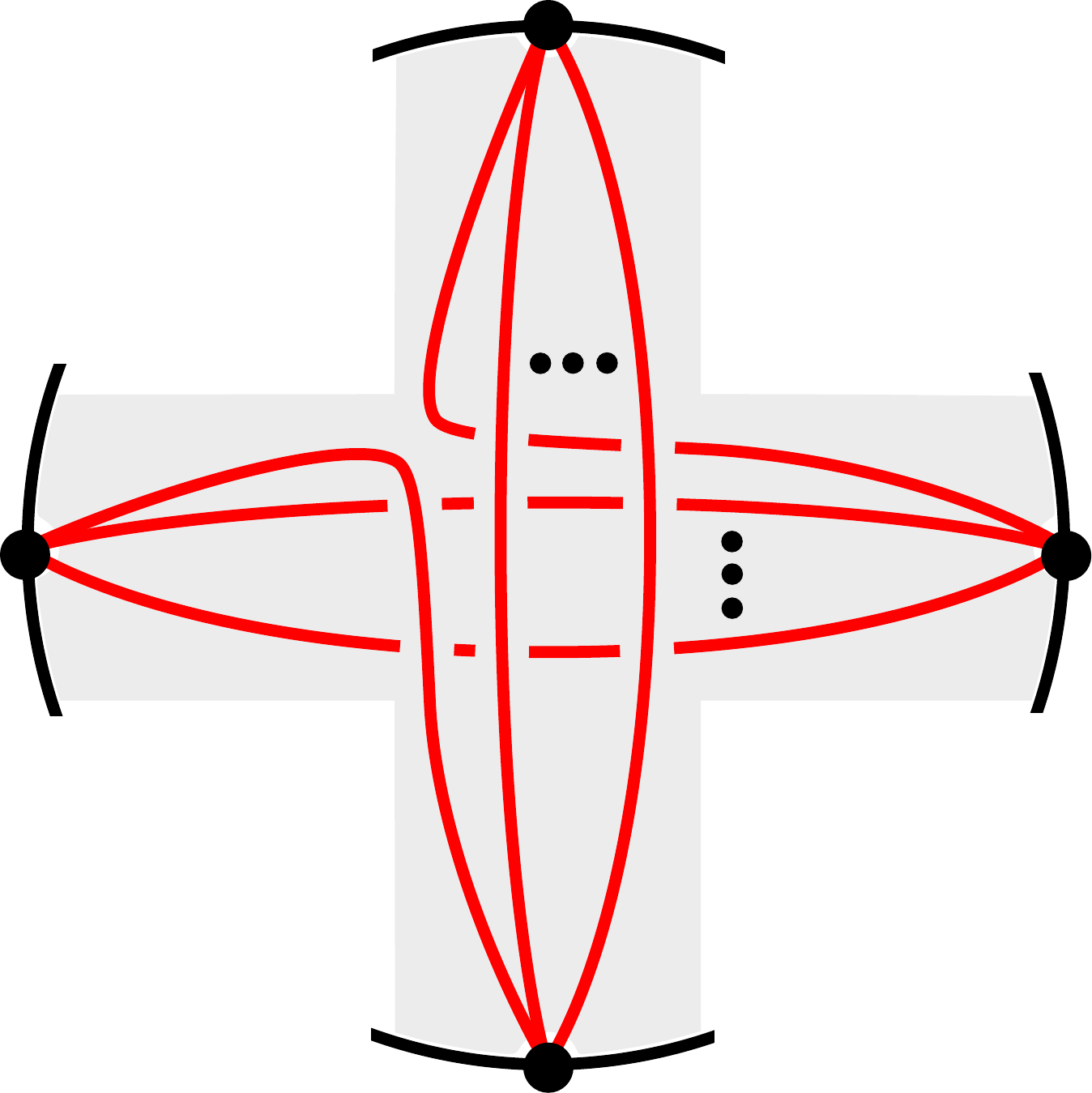}\end{array}+q^{-1}\begin{array}{c}\includegraphics[scale=0.18]{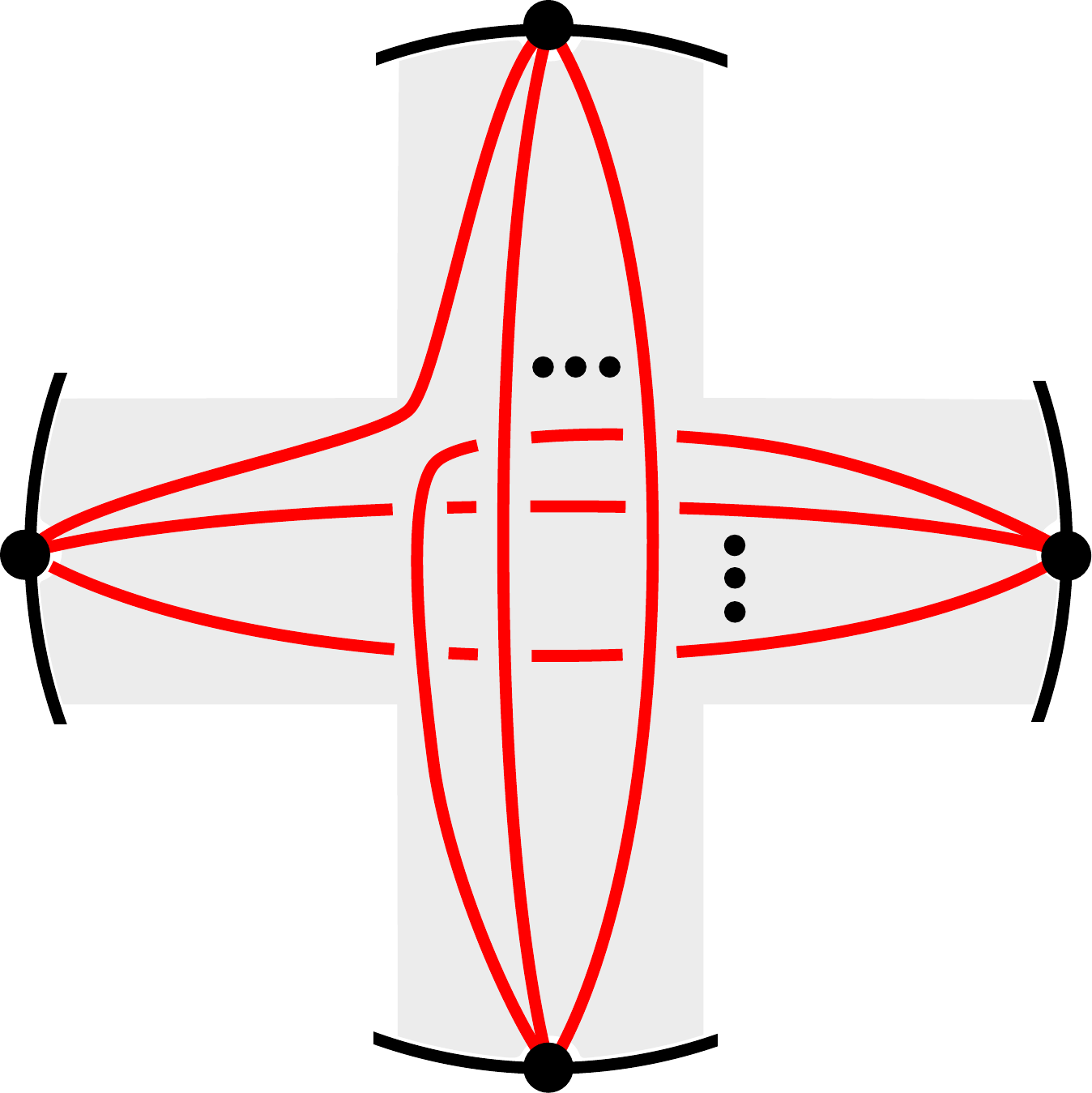}\end{array}\nonumber\\
&=&q^{1+c_1}\begin{array}{c}\includegraphics[scale=0.18]{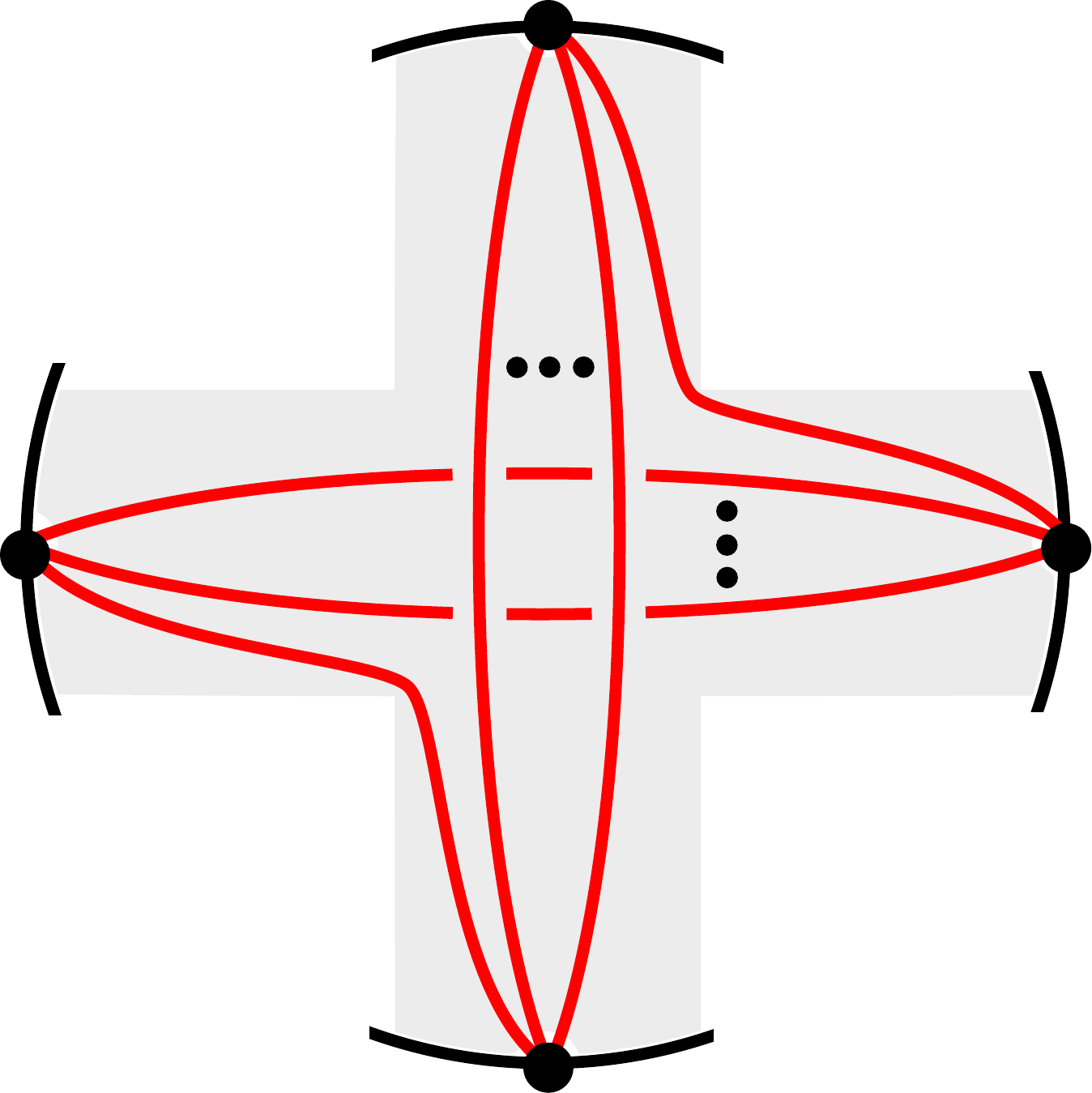}\end{array}+q^{-1+c_2}\begin{array}{c}\includegraphics[scale=0.18]{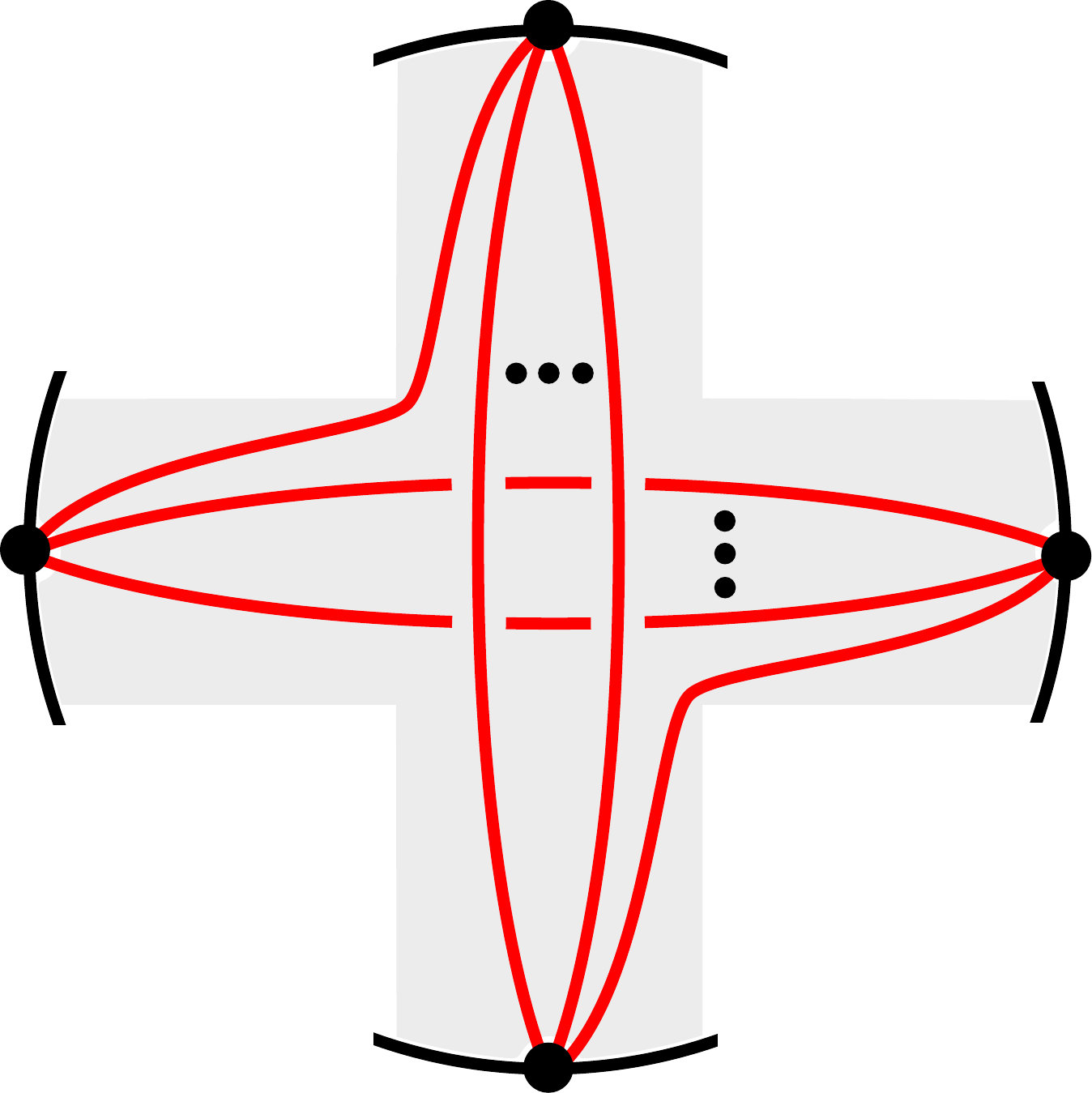}\end{array}, \label{disk_resolution}
\end{eqnarray}
where 
$c_j\in \BZ\ (j=1,2)$ are determined from (\ref{rel_nocrossing}). 
While we do not draw other Chebyshev bunches in (\ref{disk_resolution}), 
the computation (\ref{disk_resolution}) does not affect the other Chebyshev bunches.

From the computation (\ref{disk_resolution}), 
we have a linear sum of products of Chebyshev bunches with strictly less crossings with non-negative coefficients.  
By applying the operation repeatedly, finally we have a linear sum of products of Chebyshev bunches such that each of them has no crossings in $\Int \BD$. 
Moreover, each term is $q^{1/2}$-equivalent to its Weyl normalization. 
This implies that the structure constants of $\sfB^b(\Sigma,\BM)$ are in $\BZ_{\geq0} [q^{\pm 1/2}]$. 
\end{proof}

\begin{rem} 
\begin{enumerate}
    \item A similar formula for products of two bunches with explicit coefficients given by $q$-binomials is known. See \cite{Ya}.
    The positivity also follows from this result.
    \item Note that the computation (\ref{disk_resolution}) can be applied to general marked surfaces since it is local. 
    In this case, we regard the left-hand side of (\ref{disk_resolution}) as a part of $\BD$. 
\end{enumerate}

\end{rem}

\subsection{Positivity for the marked annulus with two special points}
In the subsection, we show the positivity for the marked annulus with exactly one marked point on each boundary component.

Chebyshev polynomials of the second kind are given by the following initial data and the same recurrence relation for those of the first kind: 
\begin{align}\label{Sn_recurrence}
    S_0(x)=1,\quad  S_1(x)=x,\quad  S_{n+2}(x)=x S_{n+1}(x)-S_{n}(x)\quad (n\geq0).
\end{align}

Note that 
\begin{align}
S_1(x)=T_1(x)=x,\quad S_{n}(x)=T_{n}(x)+S_{n-2}(x)
\quad (n\geq 2),\label{Chebyshev's}
\end{align}
where $T_k(x)$ is the $k$-th Chebyshev polynomial of the first kind. 
We will use the well-known relation
\begin{align}
T_m(x)\cdot T_n(x)=T_{m+n}(x)+T_{|m-n|}(x).\label{rel:Chebyshev}
\end{align}

Let $\Gamma$ be a properly embedded arc connecting the two boundary components whose endpoints are not the special points. 
Take an ideal arc $\al$ which connects the two special points and does not intersect with $\Gamma$, see the left of Figure \ref{annulus}. 
\begin{figure}[ht]\centering\includegraphics[width=250pt]{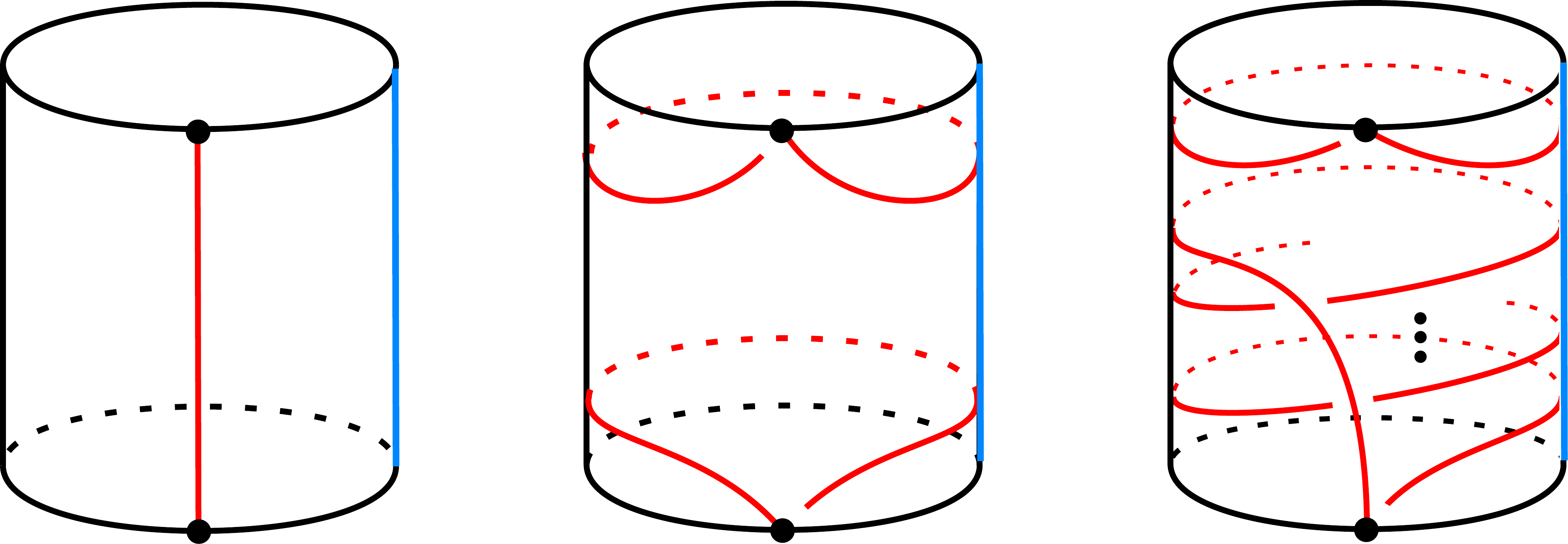}\caption{The blue arc is $\Gamma$. The middle multicurve is $B_1$.}\label{annulus}\end{figure}

Let $z$ be a parallel copy of a boundary component of $\bA$.  
Let $\tau$ denote the right-handed Dehn twist along $z$. 

\begin{prop}[{\cite[Proposition 2.3 (6)]{Le_Kauffman}}]\label{formula_arc_loop}
For $\al$ and $z$ defined as above, 
and a parallel copy $z$ of a boundary component of $\bA$, 
we have  
\begin{align}
T_n(z)\cdot \al =q^n \tau^n(\al)+q^{-n}\tau^{-n}(\al).\label{formula_annulus_loop_arc}
\end{align}
\end{prop}



For $n\in \bZ_{> 0}$, let $B_n$ be a framed $\BMp$-tangle diagram on $\BA$ consisting of two arcs such that 
\begin{enumerate}
    \item the two arcs do not intersect, 
    \item each arc connects the same boundary component, 
    \item one of them intersects once with $\Gamma$ and the other intersects $n$ times with $\Gamma$ and has $n-1$ number of self-crossings,  
\end{enumerate}
as depicted in the middle and the right of Figure \ref{annulus} with an appropriate height information at the special points.

\begin{prop}\label{formulra_B}
For any integer $n \geq 1$, 
\begin{align}\label{eq:Bn_Sn}
    B_n=q^{n-1}S_{n-1}(z)\cdot B_1.
\end{align} 
\end{prop}
\begin{proof}
The case of $n=1$ is trivial. 
It is easy to see that $B_2=qz\cdot B_1$ and 
$$
B_{n}=qz\cdot B_{n-1}-q^2B_{n-2}
$$
by applying the relations (A) and (B) to the top crossing in the right of Figure \ref{annulus}. 
Suppose that the claim holds for $n=k,k+1\ (k\geq 1)$. 
Then, we have 
\begin{eqnarray*}
B_{k+2}&=&qz\cdot B_{k+1}-q^2B_{k}\\
&=&qz\cdot(q^{k}S_{k}(z)\cdot B_1)-q^2(q^{k-1}S_{k-1}(z)\cdot B_1)\\
&=&q^{k+1}S_{k+1}(z)\cdot B_1,
\end{eqnarray*}
where the last equality follows from the recurrence relation \eqref{Sn_recurrence} for $S_n(z)$. 
\end{proof}

\begin{prop}\label{formula_arcs}
The structure constants of $\al \cdot \tau^n (\al)$ with respect to $\sfB^b(\BA,\BM)$ are in $\BZ_{\geq 0}[q^{\pm 1}]$ for any $n \in \bZ$. 
\end{prop}
\begin{proof}
It is obvious that the claim holds for $n=-1,0,1$. 

For any integer $n\geq 2$, we show
\begin{align}\label{formular:twist_arcs}
\al\cdot \tau^n(\al)=\sum_{i=1}^{\lfloor n/2\rfloor}q^{3-2i} B_{n-2i+1}+q^{-2\lceil (n+1)/2\rceil}\tau^{\lfloor n/2\rfloor}(\al)
\cdot\tau^{\lceil n/2\rceil}(\al),
\end{align}
where $\lfloor\cdot \rfloor$ and $\lceil\cdot \rceil$ are the floor and the ceiling function, respectively. 

For any $k\geq 2$, it is easy to see that 
$$\al\cdot \tau^k(\al)=q\, B_{k-1}+q^{-2}\tau(\al)\cdot \tau^{k-1}(\al)$$
and this implies that (\ref{formular:twist_arcs}) holds for $n=2,3$. 

Assume that (\ref{formular:twist_arcs}) holds for $n=k-2\geq 2$. 
Then, 
\begin{eqnarray*}
\al\cdot\tau^{k}(\al)&=&q\, B_{k-1}+q^{-2}\tau(\al)\cdot\tau^{k-1}(\al)\\
&=&q\, B_{k-1}+q^{-2}\tau\big(\al\cdot\tau^{k-2}(\al)\big)\\
&=&q\, B_{k-1}+q^{-2}\tau\bigg(\sum_{i=1}^{\lfloor (k-2)/2\rfloor}q^{3-2i} B_{k-2i-1}+q^{-2\lceil (k-1)/2\rceil}\tau^{\lfloor (k-2)/2\rfloor}(\al)
\cdot\tau^{\lceil (k-2)/2\rceil}(\al)\bigg)\\
&=&\sum_{i=1}^{\lfloor k/2\rfloor}q^{3-2i} B_{k-2i+1}+q^{-2\lceil (k+1)/2\rceil}\tau^{\lfloor k/2\rfloor}(\al)
\cdot\tau^{\lceil k/2\rceil}(\al)
\end{eqnarray*}
where the last equality holds from $\tau (B_k)=B_k$.
By substituting \eqref{eq:Bn_Sn} to (\ref{formular:twist_arcs}), we get
\begin{align}\label{formula:arcs}
    \al\cdot \tau^n(\al)=\sum_{i=1}^{\lfloor n/2\rfloor}q^{n-4i+3} S_{n-2i}(z)\cdot B_1+q^{-2\lceil (n+1)/2\rceil}\tau^{\lfloor n/2\rfloor}(\al)
\cdot\tau^{\lceil n/2\rceil}(\al).
\end{align}
Then by writing $S_k(z)$'s in terms of $T_k(x)$'s by using \eqref{Chebyshev's} repeatedly, we see the desired positivity. 
The proof for the case $n \leq -2$ is similar.
\end{proof}

\begin{proof}[Proof of Theorem \ref{thm_skein} for marked annulus]
The positivity for the product of bunches with no crossings except for $\bM_\partial$ follows from the relation (\ref{rel:Chebyshev}).

It remains to consider the product of bunches with crossings in $\Int \bA$. 
We focus on the number of minimal crossings of such a product. 
We claim that the product of two bunches is a linear sum of products of bunches with strictly less crossings, with non-negative coefficients. 
There are two possible cases for the product of two connected components with crossings in the interior of $\bA$:
\begin{itemize}
    \item the product of an arc and a Chebyshev bunch of $z$, or 
    \item the product of two arcs connecting different boundary components.
\end{itemize}
The positivity for the first case is ensured by Proposition \ref{formula_arc_loop}, and 
that for the second case is ensured by Propositions \ref{formula_arcs}. 
Moreover, in each term in the expansion of the product, the number of interior crossings is strictly less than that of the original product. 
Then the claim follows by induction on the number of interior crossings.

Finally taking the Weyl normalization of each term, we conclude that the positivity of $\sfB^b(\bA,\BM)$ holds. 
\end{proof}

\begin{rem}
\begin{enumerate}
\item 
An interesting observation from Theorem \ref{thm_skein}
is that Chebyshev polynomials $S_k$ of the second kind appear in (\ref{formula:arcs}) instead of Chebyshev polynomials $T_k$ of the first kind. 
A similar observation is in \cite{LTY}. 
\item One can show the positivity for marked annuli with more special points. 
Although there are additional terms for such marked annuli 
and we omit the details here, it is not so difficult to obtain similar results. 
\end{enumerate}
\end{rem}

\begin{thm}\label{thm_S_A}Let $(\Sigma,\BM)$ be a marked disk with at least 3 special points or a marked annulus with exactly one special point on each boundary component. 
Then the images of $S_\X$ and $S_\A$ are positive.
\end{thm}
\begin{proof} 
Since boundary edges are q-commutative with other elements, Theorem \ref{thm_skein} immediately implies that the bracelets basis $\mathsf{B}^b(\Sigma,\bM)[\partial^{-1}]=S_\X(\X_\Sigma(\bZ^\sfT))$ is positive. 

The pull-back $\Phi_\Sigma^{-1}(\mathsf{B}^b(\Sigma,\bM)[\partial^{-1}])$ via the $\CR$-algebra isomorphism 
$\Phi_\Sigma: \sSk{\Sigma} \rightarrow \mathscr{S}^q_\Sigma(\bM)[\partial^{-1}]$ 
in Theorem \ref{thm:state-clasp} is also a positive basis of $\sSk{\Sigma}$.  
By \cref{prop:lifting_compatible}, the image of $S_\A$ consists of the congruent elements in $\Phi_\Sigma^{-1}(\mathsf{B}^b(\Sigma,\bM)[\partial^{-1}])$. 
Since the image of $S_\A$ is an $\cR$-basis of $\sSk{\Sigma}_{\congr}$ (\cref{cor:basis}), the product of any two elements in $S_\A(\A_\Sigma(\bZ^\sfT))$ is presented by a linear sum of elements in $S_\A(\A_\Sigma(\bZ^\sfT))$. Then by the positivity of $\Phi_\Sigma^{-1}(\mathsf{B}^b(\Sigma,\bM)[\partial^{-1}])$, we know that the coefficients in this linear sum are non-negative. 
Thus we conclude that the image of $S_\A$ is a positive basis of $\sSk{\Sigma}_{\congr}$.  
\end{proof}

Via the isomorphisms $\mathrm{Tr}_\Sigma$ and $\mathrm{Cut}_\Sigma$, Theorem \ref{thm_S_A} immediately implies Theorem \ref{thm_intro} (2).


\begin{thebibliography} {GHKK18}

\bibitem[All22]{All} D. G. L. Allegretti, 
\emph{Quantization of canonical bases and the quantum symplectic double},  Manuscripta Math. \textbf{167} (2022), no. 3-4, 613-–651.

\bibitem[AK17]{AK} D. G. L. Allegretti and H. K. Kim, 
\emph{A duality map for quantum cluster varieties from surfaces}, 
Adv. Math. \textbf{306} (2017), 1164--1208.

\bibitem[BZ05]{BZ}
A. Berenstein and A. Zelevinsky,
\emph{Quantum cluster algebras},
Adv. Math. \textbf{195} (2005), no. 2, 405--455.

\bibitem[BL22]{BL}
W. Bloomquist and T. T. Q. L\^e,
\emph{The Chebyshev-Frobenius homomorphism for stated skein modules of 3-manifolds}, 
Math. Z. \textbf{301} (2022), no. 1, 1063--1105. 

\bibitem[BW11]{BW11}
F. Bonahon and H. Wong,
\emph{Quantum traces for representations of surface groups in $\mathrm{SL}_{2}(\mathbb{C})$},
Geom. Topol. \textbf{15} (2011), no.~3, 1569--1615.

\bibitem[BW16]{BW16} F. Bonahon and H. Wong,
\emph{Representations of the Kauffman bracket skein algebra I: invariants and miraculous cancellations}, 
Invent. Math. \textbf{204} (2016), no. 1, 195--243.

\bibitem[Bou23]{Bo} P. Bousseau, 
\emph{Strong positivity for the skein algebras of the 4-punctured sphere and of the 1-punctured torus}, 
Comm. Math. Phys. \textbf{398} (2023), no. 1, 1--58.

\bibitem[Bul97]{Bu} D. Bullock, 
\emph{Rings of $\mathrm{SL}_2({\bf C})$-characters and the Kauffman bracket skein module}, 
Comment. Math. Helv. \textbf{72} (1997), no. 4, 521--542.

\bibitem[CF99]{CF99}
V.V. Fock and L.O. Chekhov, 
\emph{Quantum Teichmüller spaces}, 
translated from Teoret. Mat. Fiz. \textbf{120} (1999), no. 3, 511--528; Theoret. and Math. Phys. \textbf{120} (1999), no. 3, 1245--1259. 

\bibitem[CL22]{CL}
F. Costantino and T. T. Q. L\^e,
\emph{Stated skein algebras of surfaces}, 
J. Eur. Math. Soc. (JEMS) \textbf{24} (2022), no. 12, 4063--4142.

\bibitem[CP07]{CP07}
L. Chekhov and R. C. Penner,
\emph{On quantizing \Teich\ and Thurston theories},
Handbook of \Teich\ theory. Vol. I, 579–-645, 
IRMA Lect. Math. Theor. Phys., 11, Eur. Math. Soc., Z\"urich, 2007.

\bibitem[DM21]{DM}
B. Davison and T. Mandel,
\emph{Strong positivity for quantum theta bases of quantum cluster algebras},
Invent. math. \textbf{226}, 725–843 (2021).


\bibitem[FG07]{FG07}
V. V. Fock and A. B. Goncharov, 
\emph{Dual Teichm\"uller and lamination spaces},
Handbook of Teichm\"uller theory,  Vol. I, 647--684; IRMA Lect. Math. Theor. Phys., \textbf{11}, Eur. Math. Soc., Z\"urich, 2007. 

\bibitem[FG08]{FG08}
V. V. Fock and A. B. Goncharov, 
\emph{The quantum dilogarithm and representations of quantum cluster varieties},
Invent. Math. \textbf{175} (2009), no. 2, 223-–286.

\bibitem[FG09]{FG09}
V. V. Fock and A. B. Goncharov, 
\emph{Cluster ensembles, quantization and the dilogarithm},
Ann. Sci. \'Ec. Norm. Sup\'er., \textbf{42} (2009), 865--930.

\bibitem[FG00]{FG00} C. Frohman and R. Gelca, 
\emph{Skein modules and the noncommutative torus}, 
Trans. Amer. Math. Soc. \textbf{352} (2000), no. 10, 4877--4888.

\bibitem[FST08]{FST}
S. Fomin, M. Shapiro and D. Thurston,
\emph{Cluster algebras and triangulated surfaces. {I}. Cluster complexes},
Acta Math. \textbf{201} (2008), 83--146.

\bibitem[GHKK18]{GHKK}
M. Gross, P. Hacking, S. Keel, and M. Kontsevich,
\emph{Canonical bases for cluster algebras},
J. Amer. Math. Soc. \textbf{31} (2018), no. 2, 497--608. 

\bibitem[GS18]{GS18} 
A. B. Goncharov and L. Shen,
 {\em Donaldson-Thomas transformations of moduli spaces of $G$-local systems},
Adv. Math. \textbf{327} (2018), 225--348. 

\bibitem[GS19]{GS19}
A. B. Goncharov and L. Shen,
\emph{Quantum geometry of moduli spaces of local systems and representation theory},
arXiv:1904.10491v3.

\bibitem[Hi10]{Hi} C. Hiatt, 
\emph{Quantum traces in quantum Teichmüller theory}, 
Algebr. Geom. Topol. \textbf{10} (2010), no. 3, 1245--1283. 

\bibitem[Ish22]{Ish}
T. Ishibashi,
\emph{Teichm\"uller and lamination spaces with pinnings},
arXiv:2212.14780.

\bibitem[IK]{IKp}
T. Ishibashi and H. Karuo,
\emph{Quantum traces and tagged triangulations},
in preparation.

\bibitem[IY]{IY} T. Ishibashi and W. Yuasa, 
\emph{State-clasp correspondence for skein algebras}, 
in preparation. 


\bibitem[L\^e15]{Le_Kauffman} T. T. Q. L\^{e},
\emph{On Kauffman bracket skein modules at roots of unity}, 
Algebr. Geom. Topol. \textbf{15} (2015), no. 2, 1093--1117.

\bibitem[L\^e18]{Le_triangular} T. T. Q. L\^{e},
\emph{Triangular decomposition of skein algebras},
Quantum Topol. \textbf{9} (2018), no.~3, 591--632.

\bibitem[L\^e18]{Le_positivity} T. T. Q. L\^{e},
\emph{On positivity of Kauffman bracket skein algebras of surfaces},
Int. Math. Res. Not. IMRN 2018, no. 5, 1314–1328.

\bibitem[L\^e19]{Le_quantum_teich} T. T. Q. L\^{e},
\emph{Quantum Teichm\"{u}ller spaces and quantum trace map},
J. Inst. Math. Jussieu \textbf{18} (2019), no.2, 249--291.

\bibitem[LY21]{LY21} T. T. Q. L\^{e} and T. Yu, 
\emph{Stated skein modules of marked 3-manifolds/surfaces, a survey}, 
Acta Math. Vietnam. \textbf{46} (2021), no. 2, 265--287.

\bibitem[LY22]{LY22} T. T. Q. L\^{e} and T. Yu, 
\emph{Quantum traces and embeddings of stated skein algebras into quantum tori}, 
Selecta Math. (N.S.) \textbf{28} (2022), no. 4, Paper No. 66, 48 pp.

\bibitem[LTY21]{LTY}
T. T. Q. L\^{e}, D. P. Thurston and T. Yu, 
\emph{Lower and upper bounds for positive bases of skein algebras},
Int. Math. Res. Not. IMRN 2021, no. 4, 3186--3202.

\bibitem[MQ23]{MQ} T. Mandel and F. Qin, 
\emph{Bracelets bases are theta bases}, 
arXiv:2301.11101. 

\bibitem[Mul16]{Muller} G. Muller,
\emph{Skein and cluster algebras of marked surfaces},
Quantum Topol. \textbf{7} (2016), no. 3, 435--503. 

\bibitem[Prz91]{Prz91} J. H. Przytycki, 
\emph{Skein modules of 3-manifolds}, 
Bull. Polish Acad. Sci. Math. \textbf{39} (1991), no.1-2, 91--100.


\bibitem[Qin23]{Qin}
F. Qin,
{\em Cluster algebras and their bases},
Representations of algebras and related structures (International Conference on Representations of Algebras, ICRA 2020, 9–25 November 2020); EMS Ser. Congr. Rep. (2023), 335--369.

\bibitem[Que22]{Qu} H. Queffelec, 
\emph{Gl2 Foam Functoriality and Skein Positivity}, 
arXiv:2209.08794.

\bibitem[She22]{Shen22}
L. Shen,
{\em Duals of semisimple Poisson-Lie groups and cluster theory of moduli spaces of G-local systems}, 
Int. Math. Res. Not. IMRN, (2022), no. 18, 14295--14318.

\bibitem[Thu14]{Th} D. P. Thurston, 
\emph{Positive basis for surface skein algebras}, 
Proc. Natl. Acad. Sci. USA \textbf{111} (2014), no. 27, 9725--9732.


\bibitem[Tur88]{Tu} V. G. Turaev,
\emph{The Conway and Kauffman modules of a solid torus}, 
Zap. Nauchn. Sem. Leningrad. Otdel. Mat. Inst. Steklov. (LOMI)167, Issled. Topol. \textbf{6} (1988), 79--89, 190; translation in
J. Soviet Math. \textbf{52} (1990), no.1, 2799--2805.

\bibitem[Ver90]{Ver90}
H. Verlinde,
\emph{Conformal field theory, two-dimensional quantum gravity and quantization of Teichmüller space},
Nuclear Phys. B \textbf{337} (1990), no. 3, 652--680. 

\bibitem[Yam92]{Ya} S. Yamada, 
\emph{A topological invariant of spatial regular graphs}, 
Knots 90 (Osaka, 1990), 447--454, de Gruyter, Berlin, 1992.

\end{thebibliography}
\end{document}